\documentclass[reqno,11pt]{amsart}
\usepackage{amsmath,amsfonts,amssymb,amsxtra,latexsym,amscd,enumerate,amsthm,verbatim}

\usepackage{graphicx}

\usepackage[margin=1.40in]{geometry}
\numberwithin{equation}{section}

\newcommand{\R}{\mathbb{R}}

\newcommand{\T}{\mathbb{T}}

\newcommand{\Z}{\mathbb{Z}}

\numberwithin{equation}{section} 

\newtheorem{theorem}{Theorem}[section]
\newtheorem{lemma}[theorem]{Lemma}
\newtheorem{proposition}[theorem]{Proposition}

\newtheorem{remark}[theorem]{Remark}

\newtheorem{claim}[theorem]{Claim}
\newtheorem{assumption}[theorem]{Assumption}

\begin{document}

\title{Linear vortex symmetrization: the spectral density function}

\author{Alexandru D. Ionescu}
\address{Princeton University}
\email{aionescu@math.princeton.edu}

\author{Hao Jia}
\address{University of Minnesota}
\email{jia@umn.edu}


\thanks{The first author was supported in part by NSF grant DMS-2007008.  The second author was supported in part by NSF grant DMS-1945179.}

\begin{abstract}
{\small}
We investigate solutions of the $2d$ incompressible Euler equations, linearized around steady states which are radially decreasing vortices. Our main goal is to understand the smoothness of what we call the {\it{spectral density function}} associated with the linearized operator, which we hope will be a step towards proving full nonlinear asymptotic stability of radially decreasing vortices. 

The motivation for considering  the spectral density function is that it is not possible to describe the vorticity or the stream function in terms of one modulated profile. There are in fact two profiles, both at the level of the physical vorticity and at the level of the stream function. The spectral density function allows us to identify these profiles, and its smoothness leads to pointwise decay of the stream function which is consistent with the decay estimates first proved in Bedrossian-Coti Zelati-Vicol in \cite{Bed2}.
\end{abstract}

\maketitle

\setcounter{tocdepth}{1}

\tableofcontents

\section{Introduction and main results}
In this paper we study the two dimensional incompressible Euler equations on the plane, in the vorticity formulation. Consider solutions $\omega(t,x,y):[0,\infty)\times \R^2\to \R$ satisfying 
\begin{equation}\label{INT1}
\partial_t\omega+u\cdot\nabla\omega=0, \quad u=\nabla^\perp\psi, \quad \Delta\psi=\omega,
\end{equation}
for $(x,y,t)\in\R^2\times[0,\infty)$. 
The two dimensional incompressible Euler equation is globally well-posed for smooth initial data, by the classical result of Wolibner \cite{Wolibner}. (See also \cite{Yudovich1, Yudovich2} for global well-posedness results with rough initial data, such as $L^\infty$ vorticity.) The long time behavior of general solutions is however very difficult to understand, due to the lack of a global relaxation mechanism. 

There have been some attempts in building a theory of ``weak turbulence" for the two dimensional Euler equation, to explain the appearance of coherent structures, see e.g., chapter 34 of \cite{SverakNotes}. A proposed mathematical explanation is that, generically, the vorticity $\omega(t)$ converges {\it weakly} but not strongly as $t\to\infty$. A rigorous proof of such a conjecture is way beyond the reach of current PDE techniques. A more realistic approach is to consider the 2D Euler equations in physically relevant perturbative regimes, such as around shear flows and vortices. In this paper we study the case of vortices.

The presence of coherent vortices is a prominent feature in two dimensional fluid flows, such as viscous flows with high Reynolds number and  perfect fluid flows. These vortices are believed to play an important role in the 2D turbulence theory (see  for example \cite{Ba1,Ba2,Benzi,Brachet,McW,McW1}). Numerical and physical experiments and formal asymptotic analysis (see \cite{Ba1,Ba2,Schecter} and references therein) suggest that small perturbations of vortices form spirals around the center of the vortex and the angle-dependent modes of the vorticity vanish in the weak sense as $t\to\infty$, which leads to ``axi-symmetrization" of the vorticity. 

The stability analysis of fluid flows is among the oldest problems studied in hydrodynamics, starting with Kelvin \cite{Kelvin}, Orr \cite{Orr}, Rayleigh \cite{Ray}, and continuing to the present day, see for example \cite{Gallay, Gallay2, Grenier,Hall,Schecter,dongyi,Zillinger1,Zillinger2} and references therein. Arnold \cite{Arnold} proved an important criteria for nonlinear stability of some steady states using monotonicity formulas, but the precise dynamics near these solutions are not known. We also refer to \cite{Choi} for recent results on the nonlinear stability of vortices in this direction. The vortex symmetrization phenomenon has been studied rigorously at the linearized level around a strictly decreasing vortex profile by Bedrossian--Coti Zelati--Vicol \cite{Bed2}, who established axi-symmetrization of the vorticity and optimal rates of decay of the associated stream function. 

\subsection{Motivation and the spectral density function}

Extending the linearized stability analysis for inviscid fluid equations to the full nonlinear setting is a challenging problem. In the case of monotonic shear flows on the bounded channel $\T\times[0,1]$ the nonlinear stability problem has been recently solved by the authors \cite{IJacta}, and independently by Masmoudi-Zhao in \cite{NaZh}, following earlier important work of Bedrossian-Masmoudi \cite{BM}, and the authors \cite{IOJI} on the nonlinear stability of the Couette flow. 

In \cite{IOJI2} we proved nonlinear asymptotic stability near point vortices in the plane, which appears to be the only nonlinear asymptotic stability result for the two dimensional Euler equation in $\R^2$. The case of general vortices, though, presents several new difficulties. To explain these new features, let us recall the main perturbation equations.  

We work with the polar coordinate $(r,\theta)$.  The velocity field $U(r)e_{\theta}$ gives a steady state for $2d$ Euler equations, for any radial function $U(r)$ (with mild assumptions on regularity and decay at $\infty$). Let $\Omega(r)$ be the associated vorticity. Then from Biot-Savart law, $U(r)$ satisfies
\begin{equation}\label{Ur}
\partial_rU(r)+U(r)/r=\Omega(r),\,\, r\in(0,\infty), \quad{\rm with}\,\,\lim_{r\to0}U(r)=\lim_{r\to\infty}U(r)=0.
\end{equation}
The linearized two dimensional Euler equation around $U(r)e_{\theta}$ is
\begin{equation}\label{B1}
\partial_t\omega+\frac{U(r)}{r}\partial_{\theta}\omega-\frac{\Omega'(r)}{r}\partial_{\theta}\psi=0,\qquad (\theta,r)\in\mathbb{T}\times(0,\infty),
\end{equation}
with initial data $\omega_0(\theta,r)$. The stream function $\psi$ is determined through
\begin{equation}\label{B2}
\left(\partial_{rr}+r^{-1}\partial_r+r^{-2}\partial_{\theta}^2\right) \psi=\omega,\qquad (\theta,r)\in\mathbb{T}\times(0,\infty).
\end{equation}
We assume the orthogonality conditions
\begin{equation}\label{B3.1}
\int_{\R^2}\omega_0(\theta,r)(\cos{\theta},\sin{\theta})\,r^2drd\theta=0,\qquad {\rm and }\qquad \int_{\mathbb{T}}\omega_0(\theta,r)\,d\theta=0.
\end{equation}
The general case can be reduced to the case with the orthogonality conditions \eqref{B3.1} by re-centering a slightly modified $\Omega(r)$.

Taking Fourier transform in $\theta$ in the equation \eqref{B1}-\eqref{B2} for $\omega$, we obtain that
 \begin{equation}\label{B4}
 \partial_t\omega_k+ik\frac{U(r)}{r}\omega_k-ik\frac{\Omega'(r)}{r}\psi_k=0, \quad (\partial_r^2+r^{-1}\partial_r-k^2r^{-2})\psi_k=\omega_k,
 \end{equation}
 for $k\in\mathbb{Z}, t\ge0, r\in(0,\infty)$. Here $\omega_k$ and $\psi_k$ are the $k-$th Fourier coefficients of $\omega$ and $\psi$.

The equations \eqref{B1}-\eqref{B2} may be compared with the linearized equation around shear flows in a finite channel, where we have for $k\in\Z\backslash\{0\}$,
\begin{equation}\label{Shear1}
\partial_t\omega_k+ikb(y)\omega-ikb''(y)\psi_k=0,\quad (\partial_y^2-k^2)\psi_k=\omega_k,\quad \psi_k(t,0)=\psi_k(t,1)=0,
\end{equation}
for $(t,x,y)\in[0,\infty)\times\mathbb{T}\times[0,1]$. Under suitable assumptions on the function $b$ and assuming that $\omega|_{t=0}$ is smooth (or Gevrey smooth) and compactly supported in $\mathbb{T}\times(0,1)$, one can prove that the solution $\omega_k(t,y)$ can be represented essentially in the form
\begin{equation}\label{Shear2}
\omega_k(t,y)= f_k(t,y)e^{-ikb(y)t},
\end{equation}
for $y\in[0,1], t\in[0,\infty)$. The point is that the {\it{profile}} $f_k(t,y)$ is smooth (or Gevrey smooth) in $y$, uniformly over $t\in[0,\infty)$ (and stays compactly supported in $(0,1)$ as well). Such formulas and quantitative bounds on $f$ are very important for passing to the nonlinear problem, since the main point of the nonlinear analysis is to identify suitable analogous nonlinear profiles and establish their (sliding) smoothness uniformly in time. See \cite{JiaG} and the recent papers \cite{IJacta,NaZh}.

In our case, the situation is different. The main issue that the linearized flow corresponding to \eqref{B4} has no ``uniformly smooth profile" even after taking off oscillatory factors. Instead, for smooth (or Gevrey smooth) initial data, we should decompose for $r>0$,
\begin{equation}\label{Intro4}
\omega_k(t,r)= f_{k1}(t,r) e^{-ik(U(r)/r)t}+f_{k2}(t,r),
\end{equation}
where $f_{k1}$ and $f_{k2}$ are both uniformly smooth in $r$ over $t\in(0,\infty)$. See Theorem \ref{MTH2'} for the details. This decomposition has natural physical meaning, since the first term comes from the interior of the fluid itself, while the second term is generated by the ``boundary" corresponding to $r=0$ (in the case of shear flows, there is no boundary contribution in \eqref{Shear2} due to the support assumption on $\omega(t)$ and $b''$). This structure presents a new and possibly significant difficulty for the nonlinear problem, since the method in the shear flow case relies crucially on obtaining control in smooth norms on a well defined profile for the vorticity. 

At a more technical level, the presence of the term $ikU(r)/r$, which generates the oscillatory part of the vorticity, has a degeneracy when $r$ approaches $r=0$, that is, $\partial_r(U(r)/r)$ approaches $0$ as $r\to0+$.  This degeneracy coupled with the nonlocal term $ik(\Omega'/r)\psi_k$ in \eqref{B4} leads to a new dynamical phenomenon, namely the {\it{depletion of vorticity}} from the origin, which refers to the fact that the vorticity enjoys better than expected decay near the critical point of the background flow. The vorticity depletion phenomenon was first observed by Bouchet-Morita \cite{Bouchet} and proved by Wei--Zhang--Zhao \cite{Dongyi2} for shear flows with critical points. In the context of vortices, the vortex depletion phenomenon was proved by Bedrossian--Coti Zelati--Vicol \cite{Bed2}. The phenomenon is important in our proof here as well.

In this paper, we take a first step towards understanding the nonlinear vortex symmetrization pro\-blem, and prove linear vortex symmetrization in Gevrey spaces. See Theorems \ref{MTH2}-\ref{MTH2'} below for the main results. More importantly, we propose a different strategy: instead of studying the profile for the vorticity function which is difficult to even define in view of \eqref{Intro4}, we focus on what we call {\it{the spectral density function}}. The spectral density function (which plays the role of a profile) is naturally associated with the linearized equation, and in our setting can be used to track both the vanishing of the various quantities in $r$ (including the vorticity depletion phenomenon), and optimal regularity. Once we have the bounds on this spectral density functions, the estimates on the stream functions, velocity fields and the vorticity can be obtained by simple calculations.

The degeneracy at $r=0$ makes it difficult to use the global change of coordinate $V=U(r)/r$, which is the natural analogue of what was used in the case of shear flows to be able to accurately define  ``resonant times". In our problem it is simpler to work with the variable $v:=\log r$ for $r\in(0,\infty)$. We also work here with Gevrey smoothness instead of Sobolev smoothness, as this is the expected framework of the nonlinear problem.

The use of spectral density functions resolves several conceptual difficulties in the study of the nonlinear axi-symmetrization problem, since we can prove optimal regularity bounds on them, and they capture both physical space decay and regularity. However, it remains open how to define the correct spectral density function in the nonlinear setting, and to derive the right evolution equations and prove bounds. We hope to address these issues in the future.

\subsection{Main equations and assumptions on the background flow}
We assume that the background radial vorticity profile $\Omega(r)$ satisfy the following natural conditions.
\begin{assumption}\label{MainAs}
There exist constants $C_\ast\in(0,\infty)$ and $c_\ast\in(-\infty,0)$ such that for all $r\in(0,\infty)$ and $j\in\Z\cap[0,\infty)$ we have
\begin{equation}\label{Back}
\begin{split}
&0<\Omega(r)\leq \frac{C_\ast}{\langle r\rangle^6}, \quad \partial_r\Omega(r)<0,\quad\Big|(r\partial_r)^j \Big(\frac{\Omega'(r)}{r}-\frac{c_\ast}{\langle r\rangle^8}\Big)\Big|\leq \frac{C_\ast^{j+1}(j!)^{2}}{\langle r\rangle^{10}}r^2.
\end{split}
\end{equation}
In the above, we used the notation $\langle x\rangle:=\sqrt{x^2+2}$ for $x\in\R$.\\
\end{assumption}
Define for $r\in\R^+:=(0,\infty)$,
\begin{equation}\label{B8}
b(r):=U(r)/r,\qquad d(r)=\Omega'(r)/r.
\end{equation}
It follows from \eqref{Back}-\eqref{B8} that for $r\in\R^+$,
\begin{equation}\label{Int0.1}
b(r)=U(r)/r=\int_0^1s\,\Omega(rs)\,ds\approx \frac{1}{\langle r\rangle^2},\qquad b'(r)\approx -\frac{r}{\langle r\rangle^4}, \qquad |b''(r)|\lesssim \frac{1}{\langle r\rangle^4}.
\end{equation}
In the above the implied constants depend only on the constants $c_\ast$ and $C_\ast$ in \eqref{Back}.

Define for $k\in\Z\backslash\{0\}, r,\rho\in (0,\infty)$, the function
\begin{equation}\label{B5}
G_k(r,\rho):=\left\{\begin{array}{ll}
                        \frac{\rho}{2|k|}\Big(\frac{r}{\rho}\Big)^{|k|}&\qquad{\rm for}\,\,r<\rho,\\
                         \frac{\rho}{2|k|}\Big(\frac{\rho}{r}\Big)^{|k|}&\qquad{\rm for}\,\,r\ge \rho.
 \end{array}\right.
\end{equation}
$G_k(r,\rho)$ is the Green's function for the differential operator $-\partial_r^2-r^{-1}\partial_r+k^2r^{-2}$ on $(0,\infty)$ with vanishing boundary conditions.

 For each $k\in\mathbb{Z}\backslash\{0\}$, we set for any $f\in L^2\big(\R^+,r^2|\Omega'(r)|^{-1}dr\big)$,
 \begin{equation}\label{B6}
 L_kf(r)=\frac{U(r)}{r}f(r)+\frac{\Omega'(r)}{r}\int_0^{\infty}G_k(r,\rho)f(\rho)\,d\rho,
 \end{equation}
  The equation \eqref{B4} can be reformulated as
 \begin{equation}\label{B7}
 \partial_t\omega_k+ikL_k\omega_k=0.
 \end{equation}
 
Define the space 
 \begin{equation}\label{spec1}
 X_k:=L^2\Big(\R^+, \frac{r^2}{|\Omega'(r)|}\,dr\Big),
 \end{equation}
 with the natural norm that for any $g\in X_k$,
 \begin{equation}\label{sped1.1}
 \|g\|_{X_k}^2:=\int_0^\infty|g(r)|^2\frac{r^2}{|\Omega'(r)|}\,dr.
 \end{equation}
 It is clear that $L_k: X_k\to X_k$ is bounded and self adjoint. Since $L_k$ is a compact perturbation of the simple multiplication operator $f\to b(r)f$, by general spectral theory, we can conclude that the spectrum of $L_k$ consists of the continuous spectrum $[0, b(0)]$ and possibly some discrete eigenvalues in $\R$. Our main assumptions \eqref{Back}-\eqref{B3} imply that there are no discrete eigenvalues, see section \ref{sec:spec} for a simple proof. 
 
 The space $X_k$ imposes strong conditions on the decay of functions inside $X_k$ if $\Omega'$ decays fast in $r$. For our purposes, we can work with initial data with milder decay properties, see the assumption \ref{MARr1}, since we only use the space $X_k$ in a qualitative way,  see e.g. \eqref{B9}, and our bounds are quantitative, by a standard limiting argument. 
 

By standard theory of spectral projections, we then have
 \begin{equation}\label{B9}
 \begin{split}
 &\omega_k(t,r)=\frac{1}{2\pi i}\lim_{\epsilon\to0+}\int_{\R}e^{-ik\lambda t}\Big\{\left[(\lambda-i\epsilon-L_k)^{-1}-(\lambda+i\epsilon-L_k)^{-1}\right]\omega^k_0\Big\}(r)\,d\lambda\\
 &=\frac{-1}{2\pi i}\lim_{\epsilon\to0+}\int_{0}^{\infty}e^{-ikb(r_0) t}b'(r_0)\Big\{\big[(-b(r_0)-i\epsilon+L_k)^{-1}-(-b(r_0)+i\epsilon+L_k)^{-1}\big]\omega^k_0\Big\}(r)\,dr_0.
 \end{split}
 \end{equation}
We then obtain from \eqref{B4} that
 \begin{equation}\label{B10}
 \begin{split}
 \psi_k(t,r)&=\frac{1}{2\pi i}\lim_{\epsilon\to0+}\int_{0}^{\infty}e^{-ikb(r_0) t}b'(r_0)\int_0^{\infty}G_k(r,\rho)\\
 &\qquad\qquad\times\bigg\{\Big[(-b(r_0)-i\epsilon+L_k)^{-1}-(-b(r_0)+i\epsilon+L_k)^{-1}\Big]\omega^k_0\bigg\}(\rho)\,d\rho\, dr_0\\
 &=\frac{1}{2\pi i}\lim_{\epsilon\to0+}\int_{0}^{\infty}e^{-ikb(r_0) t}b'(r_0)\left[\psi_{k,\epsilon}^{-}(r,r_0)-\psi_{k,\epsilon}^{+}(r,r_0)\right]dr_0.
 \end{split}
 \end{equation}
 In the above, 
 \begin{equation}\label{B11}
 \begin{split}
 &\psi_{k,\epsilon}^{+}(r,r_0):=\int_0^{\infty}G_k(r,\rho)\Big[(-b(r_0)+i\epsilon+L_k)^{-1}\omega^k_0\Big](\rho)\,d\rho,\\
 &\psi_{k,\epsilon}^{-}(r,r_0):=\int_0^{\infty}G_k(r,\rho)\Big[(-b(r_0)-i\epsilon+L_k)^{-1}\omega^k_0\Big](\rho)\,d\rho.
 \end{split}
 \end{equation}
 We note that $\psi_{k,\epsilon}^{+}(r,r_0), \psi_{k,\epsilon}^{-}(r,r_0)$ satisfy for $\iota\in\{+,-\}$ and $r,r_0\in\R^+$,
 \begin{equation}\label{F7}
 \Big[k^2/r^2-r^{-1}\partial_r-\partial_r^2\Big]\psi_{k,\epsilon}^{\iota}(r,r_0)+\frac{d(r)}{b(r)-b(r_0)+i\iota\epsilon}\psi_{k,\epsilon}^{\iota}(r,r_0)=\frac{\omega^k_0(r)}{b(r)-b(r_0)+i\iota\epsilon}.
 \end{equation}

 It is more convenient to work in the variable $v$ with $r=e^v$ for $r\in(0,\infty)$. We therefore introduce for $k\in\mathbb{Z}\backslash\{0\}, \epsilon\in[-1/4,1/4]\backslash\{0\}$ and $\iota\in\{+,-\}$,
 \begin{equation}\label{B12}
\Pi^{\iota}_{k,\epsilon}(v,w):= \psi_{k,\epsilon}^{\iota}(r,r_0),\qquad{\rm where}\,\,r=e^v,\,\,r_0=e^w\,\,{\rm and}\,\,r,r_0\in(0,\infty),
 \end{equation}
 \begin{equation}\label{B12'}
 B(v):=b(r),\qquad D(v):=d(r),\qquad{\rm where}\,\,r=e^v,\,\,r\in(0,\infty).
 \end{equation}
 We also define 
 \begin{equation}\label{B12''}
 f_0^k(v):=\omega_0^k(r),\quad  \phi_k(t,v):=\psi_k(t,r), \qquad{\rm where}\,\,r=e^v,\,\,r\in(0,\infty).
 \end{equation}

 It follows from \eqref{B10}-\eqref{B12''} that for $\iota\in\{+,-\},v,w\in\R$, $k\in\Z\backslash\{0\}$, $\epsilon\in[-1/4,1/4]$,
 \begin{equation}\label{B13}
 k^2\Pi^{\iota}_{k,\epsilon}(v,w)-\partial_v^2\Pi^{\iota}_{k,\epsilon}(v,w)+e^{2v}D(v)\frac{\Pi^{\iota}_{k,\epsilon}(v,w)}{B(v)-B(w)+i\iota\epsilon}=e^{2v}\frac{f_0^k(v)}{B(v)-B(w)+i\iota\epsilon},
 \end{equation}
 and the normalized stream function has the representation formula
 \begin{equation}\label{B14}
 \phi_k(t,v)=\frac{1}{2\pi i}\lim_{\epsilon\to0+}\int_{\mathbb{R}}e^{-ikB(w)t}\partial_wB(w)\left[\Pi_{k,\epsilon}^{-}(v,w)-\Pi_{k,\epsilon}^+(v,w)\right]dw.
 \end{equation}
 The key is to study the regularity properties of the associated spectral density functions $\Pi_{k,\epsilon}^{-}(v,w)$ and $\Pi_{k,\epsilon}^{+}(v,w)$ for $v,w\in\R$, and $\epsilon\in[-1/4,1/4]\backslash\{0\}$ small.
 
 We summarize our calculations in the following proposition.
 \begin{proposition}\label{MPro} 
 Suppose $\Omega(r), r\in\R^+$, is a radial function satisfying the assumption \eqref{Back}. Let $U(r), r\in\R^+$, be given through
 \begin{equation}\label{MPro1}
\partial_rU(r)+U(r)/r=\Omega(r),\,\, r\in(0,\infty), \quad{\rm with}\,\,\lim_{r\to0}U(r)=\lim_{r\to\infty}U(r)=0.
 \end{equation}
 Consider $\Omega(r)$ as the steady vorticity profile for the two dimensional incompressible Euler equation with the associated velocity field $U(r)e_\theta$. The linearized equations around $\Omega$ for the perturbation vorticity $\omega(t,\theta,r)\in C^1([0,\infty), H^1(\R^+, r\langle r\rangle^8 d\theta dr))$ and the associated stream function $\psi(t,\theta,r)$, $t\ge0, \theta\in\mathbb{T}, r\in\R^+$ are given by
 \begin{equation}\label{MPro2}
 \begin{split}
& \partial_t\omega+\frac{U(r)}{r}\partial_{\theta}\omega-\frac{\Omega'(r)}{r}\partial_{\theta}\psi=0,\quad\left(\partial_{rr}+r^{-1}\partial_r+r^{-2}\partial_{\theta}^2\right) \psi=\omega. \end{split}
\end{equation}
We assume that the initial vorticity deviation $\omega_0(\theta,r)$ satisfies the orthogonality conditions
\begin{equation}\label{B3}
\int_{\R^2}\omega_0(\theta,r)(\cos{\theta},\sin{\theta})\,r^2drd\theta=0,\qquad {\rm and }\qquad \int_{\mathbb{T}}\omega_0(\theta,r)\,d\theta=0.
\end{equation}
For $k\in\Z\backslash\{0\}$ and $t\ge0, r\in\R^+$, letting 
\begin{equation}\label{MPro2.1}
\omega_k(t,r):=\frac{1}{2\pi}\int_{\T}\omega(t,\theta,r) e^{-ik\theta}\,d\theta,\quad \psi_k(t,r):=\frac{1}{2\pi}\int_{\T}\psi(t,\theta,r) e^{-ik\theta}\,d\theta, 
\end{equation} 
then $\omega_k(t,r), \psi_k(t,r)$ satisfy for $t\ge0, r\in\R^+$
 \begin{equation}\label{MPro2}
 \begin{split}
& \partial_t\omega_k+ik\frac{U(r)}{r}\omega_k-ik\frac{\Omega'(r)}{r}\psi_k=0,\quad\left(\partial_{rr}+r^{-1}\partial_r-k^2r^{-2}\right) \psi_k=\omega_k. \end{split}
\end{equation}
 Define for $t\ge0, r\in\R^+, v\in\R$ with $r=e^v$, the functions
 \begin{equation}\label{MPro3}
 \begin{split}
& b(r):=U(r)/r, \quad d(r):=\Omega'(r)/r, \quad B(v):=b(r),\quad D(v):=d(r),\\
& f_k(t,v):=\omega_k(t,r),\quad f_{0}^k(v):=\omega_0^k(r),\quad \phi_k(t,v):=\psi_k(t,r).
 \end{split}
 \end{equation}
 We have the representation formula for $t\ge0, v\in\R$,
 \begin{equation}\label{MPro4}
  \phi_k(t,v)=\frac{1}{2\pi i}\lim_{\epsilon\to0+}\int_{\mathbb{R}}e^{-ikB(w)t}\partial_wB(w)\left[\Pi_{k,\epsilon}^{-}(v,w)-\Pi_{k,\epsilon}^+(v,w)\right]dw,
 \end{equation}
 where the spectral density functions $\Pi_{k,\epsilon}^{\iota}(v,w)$ satisfy for $\iota\in\{+,-\}, v,w\in\R$, and $\epsilon\in[-1/8,1/8]\backslash\{0\}$ 
\begin{equation}\label{MPro5}
 k^2\Pi^{\iota}_{k,\epsilon}(v,w)-\partial_v^2\Pi^{\iota}_{k,\epsilon}(v,w)+e^{2v}D(v)\frac{\Pi^{\iota}_{k,\epsilon}(v,w)}{B(v)-B(w)+i\iota\epsilon}=\frac{e^{2v}f_0^k(v)}{B(v)-B(w)+i\iota\epsilon}.
 \end{equation}
 We also record the following bounds for later applications. There exists $C^{\ast}\in(0,\infty)$, depending on $C_\ast, c_\ast$ in \eqref{Back} such that for $v\in\R$ and $j\in\Z\cap[1,\infty)$, 
\begin{equation}\label{LAP100}
\begin{split}
&\Big|D(v)-\frac{c_\ast}{(1+e^{2v})^4}\Big|\leq C^\ast \frac{e^{2v}}{(1+e^{2v})^{5}}, \quad \big|\partial^j_vD(v)\big|\leq (C^\ast)^j(j!)^2 \frac{e^{2v}}{(1+e^{2v})^{5}},\\
&B(v)\approx \frac{1}{1+e^{2v}}, \quad \partial_vB(v)\approx -\frac{e^{2v}}{(1+e^{2v})^2}, \quad \big|\partial^j_vB(v)\big|\lesssim (C^\ast)^j(j!)^2\frac{e^{2v}}{(1+e^{2v})^2},\end{split}
\end{equation}
\begin{equation}\label{LAP101}
\partial_vB(v)=(c_\ast/4)e^{2v}+O\big(e^{4v}\big) \qquad{\rm for}\,\,v<0.
\end{equation}
 \end{proposition}

 \subsection{Main results} 
Fix $k^\dagger\in\Z\cap[5,\infty)$.  Denote $\alpha\wedge\beta:=\min\{\alpha,\beta\}$ for $\alpha, \beta\in\R$.  Define for $k\in\Z\backslash\{0\}$, the numbers $\varkappa_k:=|k|\wedge k^\dagger$, $\mu_k:=\sqrt{k^2+8}$, $\mu_k^\ast:=\frac{9\mu_k+|k|+2}{10}$,  and for $v,\rho, w_\ast\in\R$ the function $d_{w_\ast}(v,\rho)$ as
 \begin{equation}\label{MWF1In}
d_{w_\ast}(v,\rho):=\big|\big[\min\{v,\rho\}, \max\{v,\rho\}\big]\cap\big[\min\{w_\ast,0\}, 0\big]\big|.
\end{equation}
Define also the main weight functions $\varpi_{k,w_\ast}(v,\rho)$ and $\zeta_{k,w_\ast}(v)$ for $v,\rho\in\R$, which are useful in characterizing decay property of the spectral density function and Green's functions below, as the following: for $v,\rho, w_\ast\in\R$,
\begin{equation}\label{MWF2In}
\varpi_{k,w_\ast}(v,\rho):=e^{-|k||v-\rho|-(\mu_k-|k|)d_{w_\ast}(v,\rho)},\quad \zeta_{k,w_\ast}(v,\rho):=\frac{1}{\varpi_{k,w_\ast}(v,\rho)}.
\end{equation}
We take two small constants $0<\delta_1\ll\delta_0$ which will be used to quantify the Gevrey regularity below. $\delta_0$ is determined by the regularity of the background flow $B(v)$, and $\delta_1$ is related to the regularity of the initial data. For our purposes, we assume that the background flow is much smoother than the solution to the linearized equation we consider. (Hence the condition that $\delta_0\gg\delta_1>0$.)
 
 Fix $\Phi_0(v)\in C^\infty(-\infty, -1)$ such that $\Phi_0\equiv 1$ on $(-\infty, -2]$ and 
 $$\sup_{\xi\in\R}\big|e^{\langle\xi\rangle^{8/9}}\widehat{\partial_v\Phi_0}(\xi)\big|\lesssim1.$$ 
 We choose also smooth cutoff functions $\Phi^\ast, \Phi^{\ast\ast}:\R\to[0,1]$ satisfying $\Phi^\ast\in C_0^{\infty}(-4,4)$ and $\Phi^\ast\equiv1$ on $[-2,2]$, $\Phi^{\ast\ast}\in C_0^{\infty}(-5,5)$ and $\Phi^{\ast\ast}\equiv1$ on $[-4,4]$, and 
 $$\sup_{\xi\in\R}e^{\langle\xi\rangle^{8/9}}\Big[\big|\widehat{\,\,\Phi^\ast}(\xi)\big|+\big|\widehat{\,\,\,\Phi^{\ast\ast}}(\xi)\big|\Big]\lesssim1.$$

We make the following assumptions on the initial data.
 \begin{assumption}\label{MARr1}
There exist coefficients $\sigma_k\in\R$ for each $k\in\Z\backslash\{0\}$ with $\sigma_k=0$ for $|k|\ge k^\dagger+1$, and constants $M^\dagger_k\in(0,\infty)$, such that the following statement holds. Define for $k\in\Z\backslash\{0\}$, $j\in\Z$ and $v\in\R$,
 \begin{equation}\label{MARr2}
 F_{0k}(v):=f_0^k(v)-(\sigma_k/c_\ast) D(v)e^{|k|v}\Phi_0(v),\qquad {\rm and }\qquad F^j_{0k}(v):=F_{0k}(v)\Phi^\ast(v-j),
 \end{equation}
 then $F_{0k}^j$ satisfies the bounds for all $k\in\Z\backslash\{0\}, j\in\Z$,
 \begin{equation}\label{MARr3}
 \left\|e^{\delta_1\langle k,\xi\rangle^{1/2}}\widehat{\,F^j_{0k}}(\xi)\right\|_{L^2(\xi\in\R)}\leq M^\dagger_k  \frac{e^{\mu^\ast_{\varkappa_k}j}}{1+e^{(\mu^\ast_{\varkappa_k}+\varkappa_k+8)j}}.\\
 \end{equation}
\end{assumption}

We briefly comment on the assumption \eqref{MARr3}. Notice that \eqref{MARr3} contains conditions both on the regularity and decay of $F_{0k}(v)$. The decay of $F_{0k}(v)$ is assumed to be of the order $e^{-(\varkappa_k+8)v}$ as $v\to +\infty$, and $e^{\mu^\ast_{\varkappa_k}v}$ as $v\to-\infty$. These decay conditions are compatible with the expected decay of smooth, fast decaying initial data $\omega_0(x,y)$. $\mu_{\varkappa_k}^\ast$ needs to be slightly bigger than $\mu_{\varkappa_k}$ (when $|k|\ge2$) which is the index of decay for the solution, see \eqref{MTH23.3}. This does not seem to be a problem since the nonlinear interactions will only produce terms of faster decay. On the other hand, the decay as $v\to +\infty$ is not a concern for us thanks to the fast decay of $D(v)$ as $v\to+\infty$.

Decompose for $k\in\Z\backslash\{0\}$,
\begin{equation}\label{MAR1}
\Pi_{k,\epsilon}^\iota(v,w):=(\sigma_k/c_\ast)e^{|k|v}\Phi_0(v)+\Gamma_{k,\epsilon}^{\iota}(v,w).
\end{equation}
It follows from \eqref{MPro5} that $\Gamma_{k,\epsilon}^{\iota}(v,w)$ satisfies the equation for $v,w\in\R, \epsilon\in[-1/8, 1/8]\backslash\{0\},$ $\iota\in\{+,-\}$, 
\begin{equation}\label{MAR2}
\begin{split}
&(k^2-\partial_v^2)\Gamma_{k,\epsilon}^{\iota}(v,w)+\frac{e^{2v}D(v)}{B(v)-B(w)+i\iota\epsilon}\Gamma_{k,\epsilon}^{\iota}(v,w)\\
&=\frac{e^{2v}F_{0k}(v)}{B(v)-B(w)+i\iota \epsilon}+(\sigma_k/c_\ast)\big(2|k|e^{|k|v}\partial_v\Phi_{0k}(v)+e^{kv}\partial_v^2\Phi_{0k}(v)\big).
\end{split}
\end{equation}

Our main result is the bounds on the profile of the spectral density function. We allow the implied constants to depend on $k^\dagger\in\Z\cap[5,\infty)$ and the background flow (more precisely the constants $c_\ast, C_\ast\in(0,\infty)$ appearing in \eqref{Back} and the structural constant $\kappa\in(0,1)$ coming from the limiting absorption principle, see section \ref{sec:lap}).
\begin{theorem}\label{MTH2}
Assume that $k\in\Z\backslash\{0\}$, $f_k(t,v), \phi_k(t,v)$ for $t\ge0, v\in\R$ are as in proposition \ref{MPro}, and that the assumption \ref{MARr1} holds. Then we have the following conclusions. For some sufficiently small $\epsilon_\ast\in(0,1)$,
 $\iota\in\{+,-\}$, and for all $w\in\R$ and $0<\epsilon<\epsilon_\ast e^{-2|w|}$, we have
  \begin{equation}\label{MTH2.0}
\big\|\Gamma_{k,\epsilon}^{\iota}(v,w)\big\|_{L^2(v\in\R)}\lesssim (M_k^{\dagger}+|\sigma_k|)/|k|.
\end{equation}

 The limiting spectral density function 
\begin{equation}\label{MTH2.1}
\Gamma_k(v,w):=(-i)\lim_{\epsilon\to0+}\big[\Gamma_{k,\epsilon}^{+}(v,w)-\Gamma_{k,\epsilon}^{-}(v,w)\big]=2\lim_{\epsilon\to0+}\Im \,\Gamma_{k,\epsilon}^{+}(v,w)
\end{equation}
exists, as limit of functions in $L^2_{\rm loc}(\R^2)$. Define for $v, w\in\R$, the ``profile" $\Theta_k(v,w)$ for $\Gamma_k(v,w)$,
\begin{equation}\label{MTH2.10}
\Theta_k(v,w):=\Gamma_k(v+w,w).
\end{equation}
Then $\Theta_k(v,w)$ satisfies the following properties.

{\it (i) Bounds when $v$ is away from $0$.} For $w_\ast\in\R, v_\ast\in\R$, define for $v,w\in\R$
\begin{equation}\label{MTH2.2}
\Theta_{k,v_\ast}^{w_\ast}(v,w):=\Theta_k(v,w)\Phi^\ast(v-v_\ast)\Phi^\ast(w-w_\ast)(1-\Phi^\ast(v)),
\end{equation}
then we have
\begin{equation}\label{MTH2.3}
\begin{split}
&\Big\|(|k|+|\xi|)\Big[e^{\delta_0\langle k,\xi\rangle^{1/2}}+e^{\delta_1\langle k,\eta\rangle^{1/2}}\Big]\widehat{\,\,\Theta_{k,v_\ast}^{w_\ast}}(\xi,\eta)\Big\|_{L^2(\xi,\eta\in\R)}\\
&\lesssim\big(M_k^{\dagger}+|\sigma_k|\big)\varpi_{\varkappa_k,w_\ast}(w_\ast,0)\varpi_{\varkappa_k,w_\ast}(v_\ast+w_\ast,w_\ast).
\end{split}
\end{equation}

{\it (ii) Bounds when $v$ is close to $0$.} For $w_\ast\in\R$, define for $v,w\in\R$,
\begin{equation}\label{MTH2.4}
\Theta_{k}^{w_\ast}(v,w):=\Theta_k(v,w)\Phi^\ast(w-w_\ast)\Phi^\ast(v),
\end{equation}
 we have
\begin{equation}\label{MTH2.5}
\begin{split}
&\Big\|\big(|k|+|\xi|\big)e^{\delta_1\langle k,\eta\rangle^{1/2}}\widehat{\,\,\,\Theta_k^{w_\ast}}(\xi,\eta)\Big\|_{L^2(\xi,\eta\in\R)}\lesssim(M_k^{\dagger}+|\sigma_k|)\varpi_{\varkappa_k,w_\ast}(w_\ast,0).
\end{split}
\end{equation}

{\it (iii) Equation for $\Theta_k$.} In addition, $\Theta_k(v,w)$ satisfies for $v,w\in\R$ the equation
\begin{equation}\label{MTH2.6}
(k^2-\partial_v^2)\Theta_k(v,w)+{\rm P.V.} \frac{e^{2v+2w}D(v+w)\Theta_k(v,w)}{B(v+w)-B(w)}=-2\pi\frac{e^{2w}\big(D(w)\digamma_k(w)-F_{0k}(w)\big)}{B'(w)}\delta(v),
\end{equation}
where ${\rm P.V.}$ represents principal value and $\digamma_k\in C^\infty(\R)$ satisfies the bound, for $w_\ast\in\R$ and $\digamma_k^{w_\ast}(w):=\digamma_k(w)\Phi^\ast(w-w_\ast)$,
\begin{equation}\label{MTH2.7}
\Big\|e^{\delta_1\langle k,\xi\rangle^{1/2}}\widehat{\,\,\digamma_k^{w_\ast}}(\xi)\Big\|_{L^2}\lesssim(M_k^{\dagger}+|\sigma_k|)\varpi_{\varkappa_k,w_\ast}(w_\ast,0).
\end{equation}

{\it (iv) Refined regularity property of $\Theta_k(v,w)$ in $v$.} Moreover, for any $\zeta:[-10,10]\to \R$ with Gevrey-2 regularity and $|\zeta'|\gtrsim1$ on $[-10,10]$, more precisely $\zeta\in \widetilde{G}_M^{1/2}(-10,10)$ for some $M\in(0,\infty)$ (see \eqref{Gevr2} for the precise definition of the Gevrey space $\widetilde{G}_M^{1/2}$) then for $v, \rho\in[-4,4]$, the function $\Theta(\zeta(v+\rho)-\zeta(\rho), w)$ is Gevrey regular in $\rho, w$. More precisely, for any $w_\ast\in\R$, define
\begin{equation}\label{MTH2.8}
\mathcal{H}_{w_\ast}(v, \rho, w):=\Theta(\zeta(v+\rho)-\zeta(\rho), w)\Phi^\ast(w-w_\ast)\Phi_k^\ast(v)\Phi^\ast(\rho),
\end{equation}
then for some $\delta_1'\in(0,1)$ depending on $M$ and $\delta_1$ we have
\begin{equation}\label{MTH2.9}
\begin{split}
&\Big\|\big(|k|+|\alpha|\big) \Big[e^{\delta_1'\langle k,\xi\rangle^{1/2}}+e^{\delta_1\langle k,\eta\rangle^{1/2}}\Big]\widehat{\,\,\mathcal{H}_{w_\ast}}(\alpha,\xi,\eta)\Big\|_{L^2(\R^3)}\lesssim(M_k^{\dagger}+|\sigma_k|)\varpi_{\varkappa_k,w_\ast}(w_\ast,0).
\end{split}
\end{equation}

{\it (v) Representation formula for the stream function and vorticity function.} Finally, we have the representation formula 
\begin{equation}\label{MTH2.91}
\begin{split}
&\phi_k(t,v)=-\frac{1}{2\pi }\int_\R e^{-ikB(w)t}\Theta_k(v-w,w) B'(w)\,dw, \qquad f_k(t,v)=-e^{-2v}(k^2-\partial_v^2)\phi_k(t,v).\\
 \end{split}
\end{equation}

\end{theorem}

Using the bounds on the spectral density function, we can prove the following result on the evolution of $f_k(t,v)$.
\begin{theorem}\label{MTH2'}
Decompose
\begin{equation}\label{MTH23.1}
\begin{split}
f_k(t,v)&=f_k^1(t,v)+f_k^2(t,v)\\
&:=\frac{e^{-2v}}{2\pi}\bigg[\int_\R e^{-iktB(w)}(k^2-\partial_v^2)\Theta_k(v-w,w)\Phi^\ast(v-w)B'(w)\,dw\\
&+\int_\R e^{-iktB(w)}(k^2-\partial_v^2)\Theta_k(v-w,w)(1-\Phi^\ast(v-w))B'(w)\,dw\bigg],
\end{split}
\end{equation}
then $f_k^1(t,v)e^{ikB(v)t}$ is Gevrey smooth in $v$ uniformly over $t\in[0,\infty)$, and $f_k^2(t,v)$ is Gevrey smooth in $v$ and decays in $t$. 

More precisely, there exist $\delta_0''\in(0,\infty)$ depending on $\delta_0$, and $\delta_1''\in(0,1)$ depending on $\delta_1$, such that for any $v_\ast\in\R$, defining
\begin{equation}\label{MTH23.2}
F_{k,v_\ast}^1(t,v):=f_k^1(t,v)e^{ikB(v)t} \Phi^\ast(v-v_\ast), \qquad F_{k,v_\ast}^2(t,v):=f_k^2(t,v)\Phi^\ast(v-v_\ast),
\end{equation}
we have the bounds for all $t\ge0$,
\begin{equation}\label{MTH23.3}
\begin{split}
&\Big\|e^{\delta_1''\langle k,\xi\rangle^{1/2}}\widehat{\,\,F_{k,v_\ast}^1}(t,\xi)\Big\|_{L^2(\R)}\lesssim (M_k^{\dagger}+|\sigma_k|)\Big[e^{-\mu_{\varkappa_k}|v_\ast|}{\bf 1}_{v_\ast<0}+e^{-(\varkappa_k+8)|v_\ast|}{\bf 1}_{v_\ast>0}\Big],\\
&\big|B'(v_\ast)\big|\Big\|e^{\delta_0''\langle k,\xi\rangle^{1/2}}\widehat{\,\,F_{k,v_\ast}^2}(t,\xi)\Big\|_{L^2(\R)}\lesssim \frac{M_k^{\dagger}+|\sigma_k|}{\langle t\rangle}\Big[e^{-\mu_{\varkappa_k}|v_\ast|}{\bf 1}_{v_\ast<0}+e^{-(\varkappa_k+8)|v_\ast|}{\bf 1}_{v_\ast>0}\Big].
\end{split}
\end{equation}

\end{theorem}

\subsection{Remarks on the main theorems} We briefly discuss some of the conclusions in our main theorems  \ref{MTH2}-\ref{MTH2'}. 

(i) As discussed earlier, our main objective is to study the spectral density function $\Gamma_k(v,w)$ and more precisely its ``profile" $\Theta_k(v,w)$ which is smooth in $w\in\R$ (but not in $v\in\R$). The crucial points in the bounds \eqref{MTH2.3} and \eqref{MTH2.5} are that  $\Theta_k(v,w)$ is as smooth in $w$ as the initial data $F_{0k}$, in Gevrey spaces. The exact weight $e^{\delta_1\langle k,\xi\rangle^{1/2}}$ is at the right regularity required for nonlinear inviscid damping, see \cite{BM,Deng,IJacta,NaZh}. The bounds \eqref{MTH2.9} provide more refined information on the regularity of $\Theta_k(v,w)$ in $v$. Notice that $\Theta_k(v,w)$ is not smooth in $v$. \eqref{MTH2.9} shows that, however, $\Theta(\zeta(v+\rho)-\zeta(\rho), w)$ is smooth in $\rho$, if $\zeta$ satisfies the natural assumptions in (iv) of Theorem \ref{MTH2}. 

(ii) The decomposition \eqref{MTH23.1}-\eqref{MTH23.2} shows that in general we do not have a uniform profile that we can control over all times. However, as it can be seen from the bounds \eqref{MTH23.3} the ``nonlocal profile" $F_{k,v_\ast}^2(t,v)$ is much smoother than the initial data and its regularity depends only on the smoothness of the background flow. In addition, the nonlocal profile $F_{k,v_\ast}^2(t,v)$ decays over time. This remarkable property may be important for the nonlinear analysis. 

(iii) In the formulation of the bounds \eqref{MTH2.3} and \eqref{MTH2.5}, we included suitable decay in space. Such physical space decay can be improved at the expense of sacrificing the amount of regularity we can prove. For instance we can capture the decay $e^{-|k||v|}$ precisely if we work with the weight $e^{C\langle\xi/k\rangle^{1/2}}$. However, our belief is that, for the nonlinear analysis, it is more important to work with the right regularity space characterized by the weight $e^{\delta_1\langle k,\xi\rangle^{1/2}}$, than with the best physical space decay, especially for large $k$. 

(iv) It follows from \eqref{MTH23.1}-\eqref{MTH23.3} that $f_k(t,v)\rightharpoonup 0$ as $t\to\infty$ for all $k\in\Z\backslash\{0\}$, which is at the heart of the inviscid damping phenomenon. 

(v) The order of decay $e^{-\mu_{\varkappa_k}|v_\ast|}=e^{-\sqrt{k^2+8}\,|v_\ast|}$ for $|k|\leq k^\dagger$ as $v_\ast\to-\infty$, in \eqref{MTH23.3} for $F_{k,v_\ast}^1$, is faster than the expected rate $e^{-|k||v_\ast|}$, and is a manifestation of the vorticity depletion phenomenon. We note that the index for decay $\sqrt{k^2+8}$ we obtained is slightly smaller than $|k|+2 $ obtained in \cite{Bed2}. Incidentally in the explicitly solvable case $|k|=1$, we have $\sqrt{k^2+8}=|k|+2$. It is not clear to us if the optimal rate of enhanced decay for $F_{k,v_\ast}^1$ is important for the nonlinear analysis at this time.

\subsection{Main ideas of proof for Theorem \ref{MTH2}} The basic idea is to study the equation \eqref{MAR2} and obtain coercive bounds using the the limiting absorption principle. This has been done in many works by now, see e.g. \cite{JiaL,dongyi}.  A new feature in the case of vortices is that for $w\ll-1$, the potential $\frac{e^{2v}D(v)}{B(v)-B(w)+i\iota\epsilon}\approx 8$ for $v\in[w,0]$ which is a long interval, as can be seen from simple calculations using \eqref{LAP100}-\eqref{LAP101}. Therefore we need to absorb part of the potential into the main term when solving \eqref{MAR2}. As a result, we have to study a new Green's function adapted to the nonlocal potential, see section \ref{sec:green}. 

Once we obtained the necessary bounds in weighted Sobolev spaces, we can use the commutator argument to obtain the desired Gevrey regularity, as in \cite{JiaG}. The implementation here is, however, much more complicated since we need to work with weighted space and the Green's function is not explicitly given. The key is to obtain refined bounds in Gevrey space for the Green's function associated with the long range potential for $w\leq-10$, see proposition \ref{ENHQ1}, which captures both the crucial Gevrey regularity property and optimal physical space decay of the Green's function simultaneously. 

Another interesting feature is that the limit density function $\Gamma_k(v,w)$ enjoys better decay property than $\Gamma^{\iota}_{k,\epsilon}(v,w)$. The main reason is that in equation \eqref{MAR2} the right hand side is mostly ``real valued" and that implies $\Gamma^{\iota}_{k,\epsilon}(v,w)$ has ``small" imaginary part, which can also be seen from equation \eqref{MTH2.6}.

\subsection{Organization of the paper} The rest of the paper is organized as follows. In section \ref{sec:preliminaries} we review some technical results on Gevrey spaces and elliptic regularity theory in Gevrey spaces. In section \ref{sec:green} we study the Green's function associated with a long range potential that plays an important role in the analysis of the spectral density function $\Gamma_k(v,w)$ for $w\leq0$. In section \ref{sec:spec}, we study the spectrum of the linearized operator which is essential for proving the limiting absorption principle. In section \ref{sec:lap}, we prove the limiting absorption principle that will be used to prove bounds on $\Gamma_{k,\epsilon}^{\iota}(v,w)$ in weighted Sobolev spaces. In section \ref{sec:k=1} we consider the special case when $|k|=1$, which remarkably is explicitly solvable. In section \ref{sec:sdf1} and section \ref{sec:sdf2}, we apply the limiting absorption and use the commutator argument to bound $\Gamma_{k,\epsilon}^\iota(v,w)$ and $\Gamma_k(v,w)$. In section \ref{sec:mainth} we assemble all the bounds and prove theorem \ref{MTH2} and theorem \ref{MTH2'}.

 
 \section{Preliminaries on Gevrey spaces and elliptic Gevrey regularity theory}\label{sec:preliminaries}

\subsection{Notations and conventions} We summarize here some of our main notation. We use the weight $e^{C\delta_0\langle\xi\rangle^{1/2}}$ to characterize the regularity of functions that depend only on the background flow and smoother Gevrey cutoff functions, and $\delta_0$ is chosen sufficiently small depending on the regularity of the background flow. We also allow the implied constants to depend on the background flow which is fixed. We fix $k^\dagger\in\Z\cap[5,\infty)$ throughout the paper, and set for $k\in\Z\backslash\{0\}$, $\varkappa_k:=\min\{|k|, k^\dagger\}$. We also set $\mu_k:=\sqrt{k^2+8}$ and $\mu_k^\ast:=\frac{9\mu_k+|k|+2}{10}$.  Notice that $\mu_k<\mu_k^\ast\leq |k|+2$ for $|k|\ge2$, and $\mu_k=\mu_1^\ast=3$.

\subsubsection{Fourier transforms}\label{sec:FourierTransform} In this paper, we use $\widehat{h}$ to denote the Fourier transform of $h$. If $h$ is a function of many variables and when we need to take Fourier transform in some but not all of the variables, we shall use $\alpha, \beta,\gamma,\xi, \eta$ to indicate the variable after Fourier transform, and $\rho, v, w$ to indicate that we do not take Fourier transform in these variables. For example, suppose $h\in L^2(\R^2)$, then
\begin{equation}\label{Fourier1}
\widehat{h}(\rho,\xi)=\int_\R h(\rho,w) e^{-iw\xi}\,dw.
\end{equation}

\subsubsection{Commutator arguments}
In this paper we use commutator arguments extensively to obtain higher regularity (mostly Gevrey regularity) estimates, starting from a low regularity (mostly Sobolev regularity) estimate. Suppose the Fourier multiplier we use to characterize the high regularity is $A(\xi):=e^{C\langle m,\xi\rangle^{1/2}}$, $\xi\in\R$ for some $C\in(0,\infty), m\in\R$, we often need to assume qualitatively 
\begin{equation}\label{Commut1}
\big\|A(\xi)\widehat{h}(\xi)\big\|_{L^2(\R)}<\infty,
\end{equation}
to obtain quantitatively 
\begin{equation}\label{Commut2}
\big\|A(\xi)\widehat{h}(\xi)\big\|_{L^2(\R)}\lesssim1.
\end{equation}
To remove the qualitative assumption \eqref{Commut1}, we can for example follow the technique in the appendix of \cite{IOJI} and introduce for $\rho\gg1$,
\begin{equation}\label{P1}
\sigma_{\rho}(r):=\left\{\begin{array}{ll}
\frac{1}{2}r^{-1/2}-\frac{1}{2}\rho^{-1/2}& {\rm if} \,\,r\in(0,\rho]\\
0&{\rm if}\,\,r\ge \rho,
\end{array}\right.,\qquad \Pi_{\rho}(r):=\int_0^r\sigma_{\rho}(x)\,dx,
\end{equation}
and define for $\xi\in\R$,
\begin{equation}\label{P2}
A_{\rho}(\xi):=e^{C \Pi_{\rho}(\langle m,\xi\rangle))}.
\end{equation}
Clearly $A_{\rho}(\xi)$ is a bounded function (with a bound that depends on $\rho>1$), $A_{\rho}(\xi)\leq A(\xi)$ and $A_{\rho}(\xi)\to A(\xi)$ as $\rho\to\infty$ for any $\xi\in\R$. The idea is to use $A_{\rho}(\eta)$ in the proof of \eqref{Commut2} and then send $\rho\to\infty$.  We shall use this convention many times in our proof without going through it every time.

 \subsection{Gevrey spaces}\label{appendix} 
 We summarize here some general properties of the Gevrey spaces of functions. See \cite{Rodino,Yamanaka} for more discussion and further references on Gevrey spaces. 
To perform certain algebraic operations, it is very useful to have a related definition in the physical space. For any domain $D\subseteq\T\times\R$ (or $D\subseteq\R$) and parameters $s\in(0,1)$ and $M\geq 1$ we define the spaces
\begin{equation}\label{Gevr2}
\widetilde{\mathcal{G}}^{s}_M(D):=\big\{f:D\to\mathbb{C}:\,\|f\|_{\widetilde{\mathcal{G}}^{s}_M(D)}:=\sup_{x\in D,\,m\geq 0,\,|\alpha|\leq m}|D^\alpha f(x)|M^{-m}(m+1)^{-m/s}<\infty\big\}.
\end{equation}
We start with a lemma connecting the space $\widetilde{\mathcal{G}}^s_M$ with the characterization of Gevrey spaces using Fourier transforms. 

\begin{lemma}{\cite[Subsection A.1]{IOJI}}\label{lm:Gevrey}
(i) Suppose that $s\in(0,1)$, $K>1$, and $f\in C^{\infty}(\mathbb{T}\times \mathbb{R})$ with ${\rm supp}\,f\subseteq \mathbb{T}\times[-L,L]$ satisfies the bounds $\|f\|_{\widetilde{\mathcal{G}}^{s}_K(\mathbb{T}\times \mathbb{R})}\leq 1$.
Then there is $\mu=\mu(K,s)>0$ such that 
\begin{equation}\label{gevreyP}
\big|\widehat{f}(k,\xi)\big|\lesssim_{K,s} Le^{-\mu|k,\xi|^s}\qquad \text{ for all } k\in\mathbb{Z},\,\xi\in \R.
\end{equation}

(ii) Conversely, if $\mu>0$ and $s\in(0,1)$, then there is $K=K(s,\mu)>1$ such that
\begin{equation}\label{eq:four}
\big\|f\big\|_{\widetilde{\mathcal{G}}^{s}_K(\mathbb{T}\times \mathbb{R})}\lesssim_{\mu,s} \big\|e^{\mu \langle k,\xi\rangle^{s}}\widehat{f}(k,\xi)\big\|_{L^2(\Z\times\R)}.
\end{equation}
\end{lemma}

Using this lemma one can construct cutoff functions in Gevrey spaces: for any points $a'<a\leq b<b'\in\mathbb{R}$ and any $s\in (0,1)$ there are functions $\Psi$ supported in $[a',b']$, equal to $1$ in $[a,b]$, and satisfying $\big|\widetilde{\Psi}(\xi)\big|\lesssim e^{-\langle\xi\rangle^s}$ for any $\xi\in\mathbb{R}$. See \cite[Subsection A.1]{IOJI} for an explicit construction of such functions, as well as an elementary proof of Lemma \ref{lm:Gevrey}. We use several functions of this type in the proof of our main theorem.

The physical space characterization of Gevrey functions is useful when studying compositions and algebraic operations of functions.

\begin{lemma}\label{GPF} (i) Assume  $s\in(0,1)$, $M\geq 1$, and $f_1,f_2\in \widetilde{\mathcal{G}}^{s}_M(D)$. Then $f_1f_2\in\widetilde{\mathcal{G}}^{s}_{M'}(D)$ and
\begin{equation*}
\|f_1f_2\|_{\widetilde{\mathcal{G}}^{s}_{M'}(D)}\lesssim \|f_1\|_{\widetilde{\mathcal{G}}^{s}_{M}(D)}\|f_2\|_{\widetilde{\mathcal{G}}^{s}_{M}(D)}
\end{equation*}
for some $M'=M'(s,M)\geq M$. Similarly, if $f_1\geq 1$ in $D$ then $\|(1/f_1)\|_{\widetilde{\mathcal{G}}^{s}_{M'}(D)}\lesssim 1$.

(ii) Suppose $s\in(0,1)$, $M\geq 1$, $I_1\subseteq \R$ is an interval, and $g:\mathbb{T}\times I_1\to \mathbb{T}\times I_2$ satisfies
\begin{equation}\label{gbo1}
|D^\alpha g(x)|\leq M^m(m+1)^{m/s}\qquad \text{ for any }x\in\T\times I_1,\,m\geq 1,\text{ and }|\alpha|\in [1,m].
\end{equation}
If $K\geq 1$ and $f\in \widetilde{G}^s_K(\T\times I_2)$ then $f\circ g\in \widetilde{G}^s_L(\T\times I_1)$ for some $L=L(s,K,M)\geq 1$ and
\begin{equation}\label{Ffgc}
\left\|f\circ g\right\|_{\widetilde{G}^s_L(\T\times I_1)}\lesssim_{s,K,M} \left\|f\right\|_{\widetilde{G}^s_K(\T\times I_2)}.
\end{equation}

(iii) Assume $s\in(0,1)$, $L\in[1,\infty)$, $I,J\subseteq\mathbb{R}$ are open intervals, and $g:I\to J$ is a smooth bijective map satisfying, for any $m\geq 1$,
\begin{equation}\label{gbo2}
|D^\alpha g(x)|\leq L^m(m+1)^{m/s}\qquad \text{ for any }x\in I\text{ and }|\alpha|\in [1,m].
\end{equation}
If $|g'(x)|\geq \rho>0$ for any $x\in I$ then the inverse function $g^{-1}:J\to I$ satisfies the bounds
\begin{equation}\label{gbo2.1}
|D^\alpha (g^{-1})(x)|\leq M^m(m+1)^{m/s}\qquad \text{ for any }x\in J\text{ and }|\alpha|\in [1,m],
\end{equation}
for some constant $M=M(s,L,\rho)\geq L$.
\end{lemma}

Lemma \ref{GPF} can be proved by elementary means using just the definition \eqref{Gevr2}. See also \cite[Theorems 6.1 and 3.2]{Yamanaka} for more general estimates on functions in Gevrey spaces. 

For applications below, we prove the following bound on a Fourier transform.  
\begin{lemma}\label{gbo3}
Assume that $w_\ast\in\R$, $k\in\Z\backslash\{0\}$, $\epsilon\in\R$ with $0<|\epsilon|<e^{-2|w_\ast|}$, and $\Psi_1, \Psi_2\in C_0^{\infty}(-10,10)$ with 
$$\sup_{\xi\in\R, \sigma\in\{1,2\}}\big|e^{\langle\xi\rangle^{5/6}}\widehat{\,\Psi_\sigma}(\xi)\big|\lesssim1.$$

(i) For $\eta,\xi\in\R$, define
\begin{equation}\label{gbo4}
h_1(\eta,\xi):=\int_{\R^2}\Psi_1(\rho)\frac{1}{B(\rho+w)-B(w)+i\epsilon}\Psi_2(w-w_\ast) e^{-i\rho\eta-iw\xi}\,d\rho dw.
\end{equation}
Then we have the bounds
\begin{equation}\label{gbo5}
\sup_{\xi,\eta\in\R} \big|e^{\delta_0\langle\xi\rangle^{1/2}}h(\eta,\xi)\big|\lesssim e^{2|w_\ast|}. 
\end{equation}

(ii) Define for $|\rho|\ge1/4$ and $\xi\in\R$,
\begin{equation}\label{gbo6}
h_2(\rho,\xi):=\int_{\R}\frac{1}{B(\rho+w)-B(w)+i\epsilon}\Psi_{1}(w-w_\ast) e^{-iw\xi}\, dw.
\end{equation}
Then we have the bounds
\begin{equation}\label{gbo7}
\sup_{\xi\in\R}\big|e^{\delta_0\langle\xi\rangle^{1/2}}h_2(\rho,\xi)\big|\lesssim \Big|\frac{1}{B(\rho+w_\ast)-B(w_\ast)+i\epsilon}\Big|.
\end{equation}
\end{lemma}

\begin{proof}
Let us assume $\epsilon<0$ without loss of generality. Note the identity for $\rho\in\R, w\in\R$,
\begin{equation}\label{gbo7.01}
B(\rho+w)-B(w)=\rho\int_0^1B'(w+s\rho)\,ds.
\end{equation}
Setting  for $|\rho|\leq 20$, $|w-w_\ast|\leq 20$,
\begin{equation}\label{gbo7.02}
\delta:=\epsilon/B'(w_\ast)>0, \qquad \psi(\rho,w):=\frac{1}{\int_0^1\big[B'(w+s\rho)/B'(w_\ast)\big]\,ds}\gtrsim1,
\end{equation}
 we can write
\begin{equation}\label{gbo7.03}
\begin{split}
\frac{B'(w_\ast)}{B(\rho)-B(w)+i\epsilon}=\psi(\rho,w) \frac{1}{\rho+i\delta\psi(\rho,w)}=\frac{1}{i}\psi(\rho,w)\int_\R e^{-\delta\psi(\rho,w)\gamma+i\rho\gamma}{\bf 1}_{\gamma>0}\,d\gamma.
\end{split}
\end{equation}
We note that $\psi(\rho,w)$ is Gevrey-2 regular for $|\rho|\leq 20$, $|w-w_\ast|\leq 20$ with bounds depending only on the background flow $B$. 
By Lemma \ref{GPF}, there exist small constants $c,C\in(0,\infty)$ depending only on the background flow, such that for $|\rho|\leq 20$, $|w-w_\ast|\leq 20$, $\gamma>0$, and all $m_1,m_2\in\Z\cap[0,\infty)$,
\begin{equation}\label{gbo7.031}
\big|\partial_{\rho}^{m_1}\partial_w^{m_2}\big[\psi(\rho,w)e^{-\delta\psi(\rho,w)\gamma}\big]\big|\lesssim e^{-c\delta\gamma} C^{m_1+m_2}((m_1+m_2)!)^2.
\end{equation}
Denote for $\rho,w\in\R$ and $\gamma>0$,
\begin{equation}\label{gbo7.032}
\psi_{\gamma}(\rho,w):=\psi(\rho,w)\Psi_1(\rho)\Psi_2(w-w_\ast)e^{-\delta\psi(\rho,w)\gamma}.
\end{equation}
In view of \eqref{gbo7.031}, for $\delta_0\in(0,1)$ sufficiently small depending on the background flow $B$, we can bound for $\gamma>0$,
\begin{equation}\label{gbo7.033}
\sup_{\alpha,\beta\in\R}\big|e^{2\delta_0\langle\alpha,\beta\rangle}\widehat{\,\,\psi_\gamma}(\alpha,\beta)\big|\lesssim1.
\end{equation}
Using \eqref{gbo7.03} and \eqref{gbo7.033} we can now bound for $\xi,\eta\in\R$
\begin{equation}\label{gbo7.04}
\begin{split}
 \big|e^{\delta_0\langle\xi\rangle^{1/2}}h(\eta,\xi)\big| &\lesssim e^{2|w_\ast|}\Big|\int_{\R^3}{\bf 1}_{\gamma>0}\, e^{\delta_0\langle\xi\rangle^{1/2}}\psi_\gamma(\rho,w)e^{i\rho\gamma-i\rho\eta-iw\xi} d\rho dw d\gamma\Big|\\
&\lesssim e^{2|w_\ast|}\Big|\int_{\R}{\bf 1}_{\gamma>0}\, e^{\delta_0\langle\xi\rangle^{1/2}}\widehat{\,\,\psi_\gamma}(\eta-\gamma,\xi)d\gamma\Big|\lesssim e^{2|w_\ast|}.
\end{split}
\end{equation}
This completes the proof of \eqref{gbo5}. The proof of \eqref{gbo7} is simpler, and follows from the observation that 
\begin{equation}\label{gbo7.05}
\frac{B(\rho+w_\ast)-B(w_\ast)+i\iota\epsilon}{B(\rho+w)-B(w)+i\iota\epsilon}
\end{equation}
is Gevrey-2 for $|\rho|\ge 1/4, |w-w_\ast|\leq 20$ with bounds depending only on the background flow, has modulus $\approx1$, and Lemma \ref{GPF}. We omit the standard details. The lemma is then proved.
\end{proof}

\subsection{Gevrey regularity for elliptic equations}
We now prove the following elliptic regularity estimates in Gevrey spaces.
\begin{proposition}\label{GevElli}
Fix $s\in(0,1), m\in\R, v_0\in\R, L>0$. Assume that $a\in C^{\infty}(v_0-2L,v_0+2L)$ belongs to $\widetilde{\mathcal{G}}^s_{M_0}$, more precisely for some $C_0, M_0\in(0,\infty)$, and all $\alpha\in\Z\cap[0,\infty)$, 
\begin{equation}\label{Enh10}
\sup_{v\in(v_0-2L,v_0+2L)}\Big|(L\partial_v)^\alpha a(v)\Big|\leq C_0L^{-2}M_0^\alpha (\alpha!)^{1/s}.
\end{equation}
Suppose that the function $h\in L^2(v_0-2L,v_0+2L)$ satisfies
\begin{equation}\label{Enh11}
(m^2-\partial_v^2)h(v)+a(v)h(v)=0,\quad{\rm for}\,\,v\in(v_0-2L,v_0+2L).
\end{equation}
Then $v\in C^\infty(v_0-2L,v_0+2L)$. Moreover, for suitable $M_1\in(0,\infty)$ (depending on $C_0, M_0,s$),  $h$ satisfies the bounds for all $\alpha\in\Z\cap[0,\infty)$,
\begin{equation}\label{Enh12}
\Big\|(L\partial_v)^\alpha h(v)\Big\|_{L^2(v_0-L/2,v_0+L/2)}\leq M_1^\alpha \big(\alpha!\big)^{1/s} \|h\|_{L^2(v_0-2L,v_0+2L)}e^{-|m|L/20}.
\end{equation}

\end{proposition}

\begin{proof}
By translation invariance and a rescaling, we can assume $v_0=0$ and $L=1$. Standard elliptic regularity theory shows that $h\in C^\infty(-2,2)$, so we only have to prove the quantitative bounds \eqref{Enh12} with $L=1$. Assume that $\|h\|_{L^2(-2,2)}=\sigma$. We first prove the easier bounds 
\begin{equation}\label{Enh12.5}
\|h(v)\|_{L^2(-1,1)}\leq C_ 0'M_0'\,\sigma\, e^{-|m|/20},
\end{equation}
for suitable $C_0', M_0'\in(0,\infty)$ depending on $C_0, M_0$. The bounds \eqref{Enh12.5} follow from a simple comparison argument. Indeed, we can find $c_1\in(-2, -31/16), c_2\in(31/16, 2)$ such that $|h(c_1)|\leq 5\sigma, |h(c_2)|\leq 5\sigma$. Since \eqref{Enh12.5} is nontrivial only for large $|m|$, can we assume that $m^2-|a(v)|\ge m^2/2$ for $v\in[-2,2]$. The desired bounds \eqref{Enh12.5} then follow from comparing $h$ with the solution $h^\ast$ to 
$$(m^2/2-\partial_v^2)h^\ast(v)=0, \quad {\rm for}\,\,v\in[c_1,c_2], \quad{\rm with}\,\,h^\ast(c_1)=5\sigma, \,\,h^\ast(c_2)=5\sigma.$$
It then suffices to prove  for suitable $M_1\in(0,\infty)$ (depending on $C_0, M_0,s$),  $h$ satisfies the bounds for all $j\in\Z\cap[0,\infty)$,
\begin{equation}\label{Enh12.51}
\Big\|\partial_v^j h(v)\Big\|_{L^2(-1/2,1/2)}\leq M_1^j \big(j!\big)^{1/s} \|h\|_{L^2(-1,1)}.
\end{equation}
We assume $\|h\|_{L^2(-1,1)}=1$ and use an induction argument.  Fix $\delta\in(0, (1/s)-1)$. Set for $j\ge 0$,
\begin{equation}\label{Enh13}
r_j=\frac{1}{2}+\frac{1}{2(9+j)^{\delta}}.
\end{equation}
Choose smooth cutoff function $\varphi_j(v)\in C_c^\infty(-r_j,r_j)$, $j\ge 0$ such that $\varphi_j\equiv 1$ on $[-r_{j+1}, r_{j+1}]$, $|\nabla\varphi_j|\leq 4\delta^{-1}(10+j)^{1+\delta}$. We shall prove by induction that for suitable $ M_1\in(1,\infty)$ depending on $C_0, M_0, s$, it holds that for all $j\ge0$,
\begin{equation}\label{Enh14}
\left\|\partial_v^jh\right\|_{L^2(-r_{j},r_{j})}\leq M_1^j(j!)^{1/s}.
\end{equation}
It is clear that \eqref{Enh14} holds for $j=0$.  Suppose that \eqref{Enh14} holds for $j\leq j_0$, for some $j_0\in\mathbb{Z}\cap[0,\infty)$, we need to prove that \eqref{Enh14} holds also for $j=j_0+1$. Taking $\partial_v^{j_0}$ to \eqref{Enh11} we obtain that for $v\in(-1,1)$,
\begin{equation}\label{Enh15}
(m^2-\partial_v^2)\partial_v^{j_0}h+\partial_v^{j_0}\big(a(v)h\big)=0.
\end{equation}
Therefore by integrating against $\partial_v^{j_0}h\,\varphi_{j_0}^2$ and using induction assumption, we obtain that
\begin{equation}\label{Enh16}
\begin{split}
&\int_{-1}^1m^2\big|\partial_v^{j_0}h\big|^2\varphi_{j_0}^2+\big|\partial_v^{j_0+1}h\big|^2\varphi_{j_0}^2\,dv\\
&\leq \left|2\int_{-1}^1\partial_v^{j_0+1}h\,\partial_v\varphi_{j_0}\varphi_{j_0}\partial_v^{j_0}h\,dv\right|+\sum_{j=0}^{j_0}\frac{j_0!}{j!(j_0-j)!} \int_{-1}^1\left|\partial_v^ja(v)\partial_v^{j_0-j}h\,\partial_v^{j_0}h\right|\varphi_{j_0}^2\,dv\\
&\leq\frac{1}{4}\int_{-1}^1\left|\partial_v^{j_0+1}h\right|^2\varphi_{j_0}^2\,dv+64\delta^{-2}(10+j_0)^{2+2\delta} M_1^{2j_0}(j_0!)^{2/s}\\
&\quad+\sum_{j=0}^{j_0}\frac{j_0!}{j!(j_0-j)!}C_0M_0^j(j!)^{1/s}M_1^{j_0-j}\big((j_0-j)!\big)^{1/s}M_1^{j_0}(j_0!)^{1/s}.
\end{split}
\end{equation}
Observe that if we choose $M_1>M_0$, then for $j_0\ge0,  j\in\Z\cap[0,j_0]$, 
\begin{equation}\label{Enh17}
\frac{j_0!}{j!(j_0-j)!}M_0^j(j!)^{1/s}M_1^{j_0-j}\big((j_0-j)!\big)^{1/s}\leq M_1^{j_0}(j_0!)^{1/s}.
\end{equation}
It follows from \eqref{Enh16}-\eqref{Enh17} that
\begin{equation}\label{Enh18}
\int_{-1}^1\left|\partial_v^{j_0+1}h\right|^2\varphi_{j_0}^2\,dv\leq 100\delta^{-2}(10+j_0)^{2+2\delta} M_1^{2j_0}(j_0!)^{2/s}+2C_0M_1^{2j_0}(j_0!)^{2/s}(j_0+1).
\end{equation}
To conclude the proof, recalling that $\delta\in(0,1/s-1)$, we only need to choose $M_1>M_0$ sufficiently large such that for all $j_0\ge0$
\begin{equation}\label{Enh19}
2C_0(j_0+1)+100\delta^{-2}(10+j_0)^{2+2\delta}\leq M_1^2(j_0+1)^{2/s}.
\end{equation}
The proof is then complete.
\end{proof}

We shall also need the following lemma on elliptic estimates for later applications. \begin{lemma}\label{EllipReg}
Let $\alpha\in[1/100,\infty)$. Suppose that $L>0$ and the potential $a(v)$ satisfies for some $C_2, M_2\in(0,\infty)$,
\begin{equation}\label{Enh31.1}
\sup_{v\in[-2L,2L]}\big|(L\partial_v)^ja(v)\big|\leq C_2L^{-2} M_2^j(j!)^2.
\end{equation}
Then for sufficiently small $\delta\in(0,\infty)$ depending on $C_2, M_2$ the following statement holds.
 Fix smooth cutoff functions $\Psi, \Psi^\ast$ with 
 \begin{equation}\label{Enh31.2}
 \begin{split}
& {\rm supp}\,\Psi\subset(-3/4, 3/4), \quad \Psi\equiv 1\,\,{\rm for}\,\,v\in[-1/2, 1/2], \quad \sup_{\xi\in\R}\big|e^{\langle\xi\rangle^{3/4}}\widehat{\,\,\Psi\,\,}(\xi)\big|\lesssim1,\\
 & {\rm supp}\,\Psi^\ast\subset(-1, 1), \quad \Psi^\ast\equiv 1\,\,{\rm for}\,\,v\in[-7/8, 7/8], \quad \sup_{\xi\in\R}\big|e^{\langle\xi\rangle^{3/4}}\widehat{\,\,\Psi^\ast}(\xi)\big|\lesssim1.
 \end{split}
 \end{equation}
 Set  $\Psi_L(v):=\Psi(v/L), \Psi^\ast_L(v):=\Psi(v/L)$ for $v\in\R$.
 Assume that $\phi\in L^2(-2L,2L)$ satisfies the elliptic equation for $v\in(-2L,2L)$,
\begin{equation}\label{Enh31}
(\alpha^2/L^2-\partial_v^2)\phi(v)+a(v)\phi(v)=h(v),
\end{equation}
where non-homogenous term $h$ satisfies 
$\int_{\R}e^{\delta\langle L\xi\rangle^{1/2}}\big|\widehat{\big(\Psi_L^\ast h\big)}(\xi)\big|^2\,d\xi<\infty.$
Then we have the following bounds
\begin{equation}\label{Enh33}
\begin{split}
&\int_{\R}(\alpha^2/L^2+\xi^2)^2e^{2\delta\langle L\xi\rangle^{1/2}} \Big|\widehat{\big(\Psi_L \phi\big)}(\xi)\Big|^2\,d\xi\\
&\lesssim_{C_2,M_2} \int_{\R}e^{2\delta\langle L\xi\rangle^{1/2}} \Big|\widehat{\big(\Psi^\ast_L h\big)}(\xi)\Big|^2\,d\xi+e^{-\alpha /200}\bigg[L^{-4}\int_{-2L}^{2L}|\phi(v)|^2dv+\int_{-2L}^{2L}|h(v)|^2\,dv\bigg].
\end{split}
\end{equation}
\end{lemma} 

\begin{remark}
The main point of the estimate is that we can bound the solution using the same Gevrey norms as the nonhomogeneous term, without any loss of regularity. We remark also that generally the coefficients of the equation need to be more regular than the solutions.
\end{remark}


\begin{proof}
By a rescaling we can assume that $L=1$. The main idea is to represent $\phi$ using the integral of Green's function against the non-homogeneous term $h$ which can be estimated precisely using regularity properties of the Green's function, and a part which solves a homogeneous equation and enjoys better regularity property thanks to the smoothness of the coefficients of the equation. 

For this purpose, we need to study the Green's function for the differential operator $\alpha^2-\partial_v^2+a(v)$, which unfortunately may not be well defined over $(-2,2)$. To fix the issue, we choose for $c\in(0,1)$ a smooth cutoff function $\Psi_c: \R\to [0,1]$ with $\Psi_c\in C_0^{\infty}(-c,c)$, $\Psi_c\equiv1$ on $[-7c/8, 7c/8]$ and $\sup_{\xi\in\R}\big|e^{\langle\xi\rangle^{4/5}}\widehat{\,\Psi_c}(\xi)\big|\lesssim_c1.$ Recall that $\alpha\in[1/100,\infty)$. We can find $c\in(0,1/4)$ such that $$\alpha^2/2-\partial_v^2-2\|a\|_{L^\infty(-2,\,2)}\mathbf{1}_{(-c,\, c)}(v)$$ is a positive operator on $H^1(\R)$. The size of $c$ depends only on $\|a\|_{L^\infty(-2,2)}$.  Fix any $v_0\in\R$ with $[v_0-c,v_0+c]\subseteq [-2,2]$. Define for $v\in\R$,
\begin{equation}\label{GeL0.1}
a_c(v):=a(v)\Psi_c(v-v_0).
\end{equation}
Again for simplicity of notations we suppressed the dependence of $a_c$ on $v_0$ which is fixed.
Then we can define for the interval $I:=[v_0-c, c_0+c]$ the Green's function $g_I(v,\rho)$ which satisfies
\begin{equation}\label{Enh34}
\begin{split}
&\big(\alpha^2-\partial_v^2\big)g_I(v,\rho)+a_c(v)g_I(v,\rho)=\delta(v-\rho),\qquad{\rm for}\,\,v\in\R,\,\rho\in I.
\end{split}
\end{equation}
Standard energy estimates imply that for $\rho\in I$
\begin{equation}\label{Enh35}
|\alpha|\big\|g_I(v,\rho)\big\|_{L^2(\R)}+\big\|\partial_vg_I(v,\rho)\big\|_{L^2(\R)}\lesssim 1. 
\end{equation}
 Set for $v,\rho\in\R$,
\begin{equation}\label{Enh36}
g_I^\ast(v,\rho):=g_I(v,\rho)\Psi_c(\rho-v_0).
\end{equation}

We claim the following bounds.
\begin{claim}\label{claim1}
For sufficiently small $\delta'\in(0,1)$ depending only on $C_2, M_2$, such that we have the bounds for all $\xi,\eta\in\R$,
\begin{equation}\label{Enh37}
\big|\widehat{\,\,g_I^\ast\,}(\xi,\eta)\big|\lesssim \frac{e^{-\delta'\langle \xi+\eta\rangle^{1/2}}}{\alpha^2+\xi^2}.
\end{equation}
\end{claim}

Assume the claim for a moment. We decompose $h$ on the interval $[v_0-7c/8,v_0+7c/8]$  in the following way
\begin{equation}\label{Enh38}
\phi(v)=\phi_1(v)+\phi_2(v):=\int_{\R}g^\ast_I(v,\rho)h(\rho)\,d\rho+\phi_2(v),
\end{equation}
where $\phi_2$ satisfies on $[v_0-3c/4,v_0+3c/4]$
\begin{equation}\label{Enh39}
(\alpha^2-\partial_v^2)\phi_2(v)+a(v)\phi_2(v)=0.
\end{equation}
The desired bounds \eqref{Enh33} then follow from the bounds \eqref{Enh37}, proposition \ref{GevElli}, and a standard partition of unity argument, except for the exponential factor $e^{-\alpha/200}$. To see the exponential factor in \eqref{Enh33}, we note that it is important only for large $\alpha$. In that case the Green's function $g$ can be defined globally on $(-2,2)$ and the desired bounds \eqref{Enh33} follows from proposition \eqref{GevElli}.
\end{proof}

It remains to give
\begin{proof}[Proof of Claim \ref{claim1}]
It follows from \eqref{Enh34} that $g_I^\ast(v,\rho)$ satisfies for $v,\rho\in\R$ the equation
\begin{equation}\label{Enh24}
\begin{split}
&(\alpha^2-\partial_v^2+a_c(v))g_I^\ast(v,\rho)=\delta(v-\rho) \Psi_c(\rho-v_0).
\end{split}
\end{equation}
Define for $v, \rho\in\R$,
\begin{equation}\label{GeL1}
h_I(v,\rho):=g_I^\ast(v+\rho,\rho),
\end{equation}
then by \eqref{Enh24} we have for $v,\rho\in\R$,
\begin{equation}\label{GeL2}
\begin{split}
&(\alpha^2-\partial_v^2+a_c(v+\rho))h_I(v,\rho)=\delta(v) \Psi_c(\rho-v_0)
\end{split}
\end{equation}

It is clear from a change of variable that \eqref{Enh37} follows from the inequality that for $\xi,\eta\in\R$ 
\begin{equation}\label{GeL3}
\big|\widehat{\,\,h_I}(\xi,\eta)\big|\lesssim \frac{e^{-\delta'\langle\eta\rangle^{1/2}}}{\alpha^2+\xi^2}.
\end{equation}


We divide the proof of the bounds \eqref{GeL3} into several steps.

{\bf Step 1: low frequency bounds.}  \eqref{Enh35} implies that 
\begin{equation}\label{Ge5}
\big\|(|\alpha|+|\xi|)\widehat{\,h_I}(\xi,\eta)\big\|_{L^2(\R^2)}\lesssim1.
\end{equation}

{\bf Step 2: high frequency integral bounds.} In this step we use a commutator argument to prove
\begin{equation}\label{Ge6}
\big\|(|\alpha|+|\xi|)e^{\delta'\langle\eta\rangle^{1/2}}\widehat{\,h_I}(\xi,\eta)\big\|_{L^2(\R^2)}\lesssim1,
\end{equation}
for a suitable $\delta'\in(0,1)$ depending only on $C_2,M_2$.
Define the Fourier multiplier operator $A$ as follows. For any $\varphi\in L^2(\R)$, 
\begin{equation}\label{Enh25}
\widehat{\,A\varphi\,}(\eta):=e^{\delta'\langle \eta\rangle^{1/2}}\widehat{\varphi}(\eta), \quad {\rm for}\,\,\eta\in\R.
\end{equation}
Applying $A$ to \eqref{Enh24} (acting on the variable $\rho$) we obtain 
\begin{equation}\label{Enh26}
\begin{split}
&(\alpha^2-\partial_v^2)A\big[h_I(v,\cdot)\big](\rho)+a_c(v+\rho)A\big[\,h_I(v,\cdot)\big](\rho)\\
&=\delta(v)A\big[\Psi(\cdot-v_0)\big](\rho)+a_c(v+\rho)A\big[\,h_I(v,\cdot)\big](\rho)-A\big[a_c(v+\cdot)h_I(v,\cdot)\big](\rho).
\end{split}
\end{equation}
Standard energy estimates (by multiplying $A\big[h_I(v,\cdot)\big](\rho)$ to \eqref{Enh26} and integrating in $v,\rho\in\R$) and Sobolev inequality imply
\begin{equation}\label{Enh27}
\begin{split}
&\left\|(|\alpha|+|\xi|)A(\eta)\widehat{\,h_I}(\xi,\eta)\right\|_{L^2(\R^2)}\\
&\lesssim 1+\Big\|a_c(v+\rho) A\big[\,h_I(v,\cdot)\big](\rho)-A\big[a_c(v+\cdot)h_I(v,\cdot)\big](\rho)\Big\|_{L^2(\R^2)}.
\end{split}
\end{equation}
We have the following commutator estimates using \eqref{Ge5}, for any $\gamma\in(0,1)$ and suitable $C_\gamma\in(0,\infty)$,
\begin{equation}\label{Enh30}
\begin{split}
&\Big\|a_c(v+\rho) A\big[\,h_I(v,\cdot)\big](\rho)-A\big[a_c(v+\cdot)h_I(v,\cdot)\big](\rho)\Big\|_{L^2(\R^2)}\\
&\lesssim \Big\|\int_{\R^2} \widehat{\,a_c}(\beta) \big[A(\eta-\beta)-A(\eta)\big]\widehat{\,h_I}(\xi-\beta,\eta-\beta)\,d\gamma\Big\|_{L^2(\R^2)}\\
&\lesssim \Big\|\langle \eta\rangle^{-1/2}A(\eta)\widehat{\,h_I}(\xi,\eta)\Big\|_{L^2(\R^2)}\lesssim \gamma \big\|A\big[\,h_I(v,\cdot)\big](\rho)\big\|_{L^2(\R^2)}+C_\gamma.
\end{split}
\end{equation}
Combining \eqref{Ge5}, \eqref{Enh27} and \eqref{Enh30}, we obtain the desired bounds \eqref{Ge6}.

{\bf Step 3: the pointwise bound.} The desired bounds \eqref{GeL3} follow from \eqref{Ge6} and equation \eqref{Enh24}.
\end{proof}

 \section{Bounds on the Green's function associated with a long range potential}\label{sec:green}
 In this section we define and study the property of a Green's function associated with a long range potential which is important in proving the ``vortex depletion" phenomenon. Define for $w\in(-\infty,-5]$ the potential 
 \begin{equation}\label{Pot1}
 V_w(v):=\frac{e^{2v}D(v)}{B(v)-B(w)}\Big[\mathbf{1}_{[w+2, \infty)}\ast \Psi^\dagger\Big](v),
 \end{equation}
 where $\Psi^\dagger$ is a nonnegative Gevrey regular cutoff function satisfying 
 $${\rm supp}\,\Psi^\dagger\subseteq [-1,1],\quad \int_\R\Psi^\dagger=1, \quad {\rm and}\quad \sup_{\xi\in\R}\Big|e^{\langle\xi\rangle^{4/5}}\widehat{\,\Psi\,}(\xi)\Big|\lesssim 1.$$
 We note that $V_w\ge0$ on $\R$, which is important in applying the maximum principle below.
We begin with a simple result which will be useful for later applications.
\begin{lemma}\label{POTT1}
Assume that $k\in\mathbb{Z}\backslash\{0\}$ and $A, A'\in\R$ with $A<A'$. For $\rho\in\R$, suppose $g_k(\cdot,\rho)\in H^1(\R)$ is the solution to 
\begin{equation}\label{POTT2}
(k^2-\partial_v^2)g_k(v,\rho)+8{\bf 1}_{[A,\,A']}(v)g_k(v,\rho)=\delta(v-\rho), \qquad {\rm for}\,\,v\in\R.
\end{equation}
Set $m:=\min(\rho,v), M:=\max(\rho,v)$, $\mu_k:=\sqrt{k^2+8}$ and $d:=|[m,M]\cap[A,\,A']|$, then we have the following bounds for $v, \rho\in\R$,
\begin{equation}\label{POTT3}
|g_k(v,\rho)|\lesssim\frac{1}{|k|} e^{-|k||v-\rho|} e^{-(\mu_k-|k|) d}.
\end{equation}
In the above, the implied constant is independent of $k\in\mathbb{Z}\backslash\{0\}, A, A', \rho$ and $v\in\R$. 
\end{lemma}

\begin{proof}
The proof follows from slightly complicated but explicit calculations. By translation symmetry, we can assume that $A<A'=0$. We consider several cases, depending on the range of $\rho$. We can assume $k>0$, without loss of generality.

{\bf Case 1: $\rho\leq A$.}  Direct calculations show that for $v\in\R$,
\begin{equation}\label{Enh6}
\begin{split}
g_k(v,\rho)=&c_1e^{k(v-\rho)}{\bf 1}_{(-\infty,\rho]}(v)+\Big[c_2e^{k(v-\rho)}+c_3e^{-k(v-\rho)}\Big]{\bf 1}_{(\rho,A)}(v)\\
&+\Big[c_4e^{\mu_k v}+c_5e^{-\mu_k v}\Big]{\bf 1}_{[A,0]}(v)+c_6e^{-kv}{\bf 1}_{(0,\infty)}(v),
\end{split}
\end{equation}
where the coefficients $c_j, j\in\Z\cap[1,6]$ are determined by the requirement that $g_k(v,\rho)$ is $C^1$ except when $v=\rho$ where the derivative in $v$ has a jump of unit size.  More precise, we have
\begin{equation}\label{Enh6.1}
\begin{split}
&c_1=c_2+c_3, \quad k c_1-kc_2+kc_3=1, c_2\,e^{k(A-\rho)}+c_3\,e^{-k(A-\rho)}=c_4\,e^{\mu_kA}+c_5\,e^{-\mu_kA},\\
&c_2k\,e^{k(A-\rho)}-c_3k\,e^{-k(A-\rho)}=c_4\mu_k\,e^{\mu_k A}-c_5\mu_k\,e^{-\mu_kA},\\
&c_4+c_5=c_6,\quad c_4\mu_k-c_5\mu_k=-c_6k.
\end{split}
\end{equation}
Routine calculations then show that
\begin{equation}\label{Enh7}
\begin{split}
&|c_1|\lesssim \frac{1}{2k}, \quad c_2\approx -\frac{1}{k^3}e^{-2k(A-\rho)}, \quad c_3=\frac{1}{2k}, \quad c_4=\frac{\mu_k-k}{2\mu_k}c_6\\
&c_5=\frac{\mu_k+k}{2\mu_k}c_6,\quad c_6\approx\frac{1}{k}e^{k\rho}e^{(\mu_k-k)A}.
\end{split}
\end{equation}
The desired bounds \eqref{Enh3} follow from \eqref{Enh5} and \eqref{Enh6}-\eqref{Enh7} in this case.

{\bf Case 2: $\rho\in[A,0]$.}  We have by direct computation that for $v\in\R$,
\begin{equation}\label{Enh8}
\begin{split}
g_k(v,\rho)=&c_1e^{k(v-A)}{\bf 1}_{(-\infty,A]}(v)+\Big[c_2e^{\mu_k(v-\rho)}+c_3e^{-\mu_k(v-\rho)}\Big]{\bf 1}_{(A,\rho)}(v)\\
&+\Big[c_4e^{\mu_k(v-\rho)}+c_5e^{-\mu_k(v-\rho)}\Big]{\bf 1}_{[\rho, 0]}(v)+c_6e^{-kv}{\bf 1}_{(0,\infty)}(v),
\end{split}
\end{equation}
where similar to \eqref{Enh6.1} we have
\begin{equation}\label{Enh8.1}
\begin{split}
&c_1=c_2\,e^{\mu_k(A-\rho)}+c_3\,e^{-\mu_k(A-\rho)},\quad c_1k=c_2\mu_k\,e^{\mu_k(A-\rho)}-c_3\mu_k\,e^{-\mu_k(A-\rho)},\\
&c_2+c_3=c_4+c_5,\quad c_2\mu_k-c_3\mu_k-(c_4\mu_k-c_5\mu_k)=1, \\
&c_4e^{-\mu_k\rho}+c_5e^{\mu_k\rho}=c_6,\quad c_4\mu_ke^{-\mu_k\rho}-c_5\mu_ke^{\mu_k\rho}=-c_6k.
\end{split}
\end{equation}
Routine calculations show
\begin{equation}\label{Enh9}
\begin{split}
&c_1\approx e^{-2\mu_k\rho+\mu_kA}c_6,\quad c_2=\frac{\mu_k+k}{2\mu_k}e^{-\mu_k(A-\rho)}c_1\approx \frac{1}{\mu_k},\quad c_3=\frac{\mu_k-k}{2\mu_k}e^{\mu_k(A-\rho)}c_1,\\
&c_4=\frac{\mu_k-k}{2\mu_k}e^{\mu_k\rho}c_6,\quad c_5=\frac{\mu_k+k}{2\mu_k}e^{-\mu_k\rho}c_6\approx \frac{1}{\mu_k},\quad c_6\approx \frac{1}{k}e^{\mu_k\rho}.
\end{split}
\end{equation}
The desired bounds \eqref{Enh3} follow from \eqref{Enh5} and \eqref{Enh8}-\eqref{Enh9} in this case.

{\bf Case 3: $\rho\ge0$.} \eqref{Enh3} can be proved using similar calculations as in Case 1, or using symmetry of the bounds.

\end{proof}

We are now ready to prove a key estimate on the Green's function associated with the nonlocal potential $V_w$.
\begin{lemma}\label{Enh1}
Assume that $w\leq-5$. For $k\in\mathbb{Z}\backslash\{0\}$, let $\mathcal{G}_k^w(v,\rho)$ be the Green's function to the differential operator $k^2-\partial_v^2+V_w$ on $\mathbb{R}$, i.e., 
\begin{equation}\label{Enh2}
(k^2-\partial_v^2+V_w)\mathcal{G}_k^w(v,\rho)=\delta(v-\rho),
\end{equation}
for $v\in\R, \rho\in\R$. Set $m:=\min(\rho,v), M:=\max(\rho,v)$, $\mu_k:=\sqrt{k^2+8}$ and $d:=|[m,M]\cap[w,0]|$, then we have the following bounds
\begin{equation}\label{Enh3}
|\mathcal{G}_k^w(v,\rho)|\lesssim\frac{1}{|k|} e^{-|k||v-\rho|} e^{-(\mu_k-|k|) d}=\varpi_{k,w}(v,\rho)/|k|.
\end{equation}
In addition, we have the following bounds for the derivatives for $v,\rho\in\R$,
\begin{equation}\label{Enh3.5}
|\partial_v\mathcal{G}_k^w(v,\rho)|\lesssim \varpi_{k,w}(v,\rho),
\end{equation}
\begin{equation}\label{Enh3.51}
|\partial_{v\rho}\mathcal{G}_k^w(v,\rho)-\delta(v-\rho)|\lesssim |k|\varpi_{k,w}(v,\rho).
\end{equation}
\end{lemma}

\begin{remark}\label{Enh4}
The improved decay given by the additional factor $e^{-(\mu_k-|k|)d}$ is at the heart of the vortex depletion phenomenon.
\end{remark}

\begin{proof}
We can assume that $k\ge1$ without loss of generality. Since $k\in\mathbb{Z}\backslash\{0\}$ and $w\leq-5$ are fixed throughout the proof, for the simplicity of notations, we suppress the dependence on $k, w$ of various quantities, when there is no possibility of confusion. 

{\bf Step 1: proof of bounds \eqref{Enh3}.} We can assume that $w\leq-5$ satisfies $|w|\gg1$, since otherwise the desired bounds \eqref{Enh3} follow from the fact that $V_w\ge0$ on $\R$ and the comparison principle.   Fix $A>10$ sufficiently large, to be determined later. Set for $w\leq-2A$,  $V_A^\ast(v):=8\mathbf{1}_{[w+A,-A]}(v)$ for $v\in\R$, and let $G^\ast_A(v,\rho)$ be the Green's function for $k^2-\partial_v^2+V_A^\ast$ on $\R$. Set also $V_A^{\ast\ast}(v):=\frac{e^{2v}D(v)}{B(v)-B(w)}\mathbf{1}_{[w+A,-A]}(v)$ for $v\in\R$, and let $G^{\ast\ast}_A(v,\rho)$ be the Green's function for $k^2-\partial_v^2+V_A^{\ast\ast}$ on $\R$. Then by the comparison principle, for $v, \rho\in\R$, 
\begin{equation}\label{Enh5}
\mathcal{G}_k^w(v,\rho)\leq G_A^{\ast\ast}(v,\rho).
\end{equation}
To prove the desired bounds \eqref{Enh3}, it suffices to prove for $v, \rho\in\R$ and $w\leq-2A$,
 \begin{equation}\label{Enh5.01}
G_A^{\ast\ast}(v,\rho)\lesssim_A\frac{1}{|k|} e^{-|k||v-\rho|} e^{-(\mu_k-|k|) d}.
\end{equation}
The main idea to prove \eqref{Enh5.01} is to compare $G_A^{\ast\ast}$ and $G_A^{\ast}$. Writing for $v,\rho\in\R$,
\begin{equation}\label{Enh5.02}
G_A^{\ast\ast}(v,\rho):=G_A^{\ast}(v,\rho)+g_A(v,\rho).
\end{equation}
Then $g_A$ satisfies for $v,\rho\in\R$,
\begin{equation}\label{Enh5.03}
\big(k^2-\partial_v^2+V_A^\ast(v)\big)g_A(v,\rho)=\big(V_A^\ast(v)-V_A^{\ast\ast}(v)\big)g_A(v,\rho)+\big(V_A^\ast(v)-V_A^{\ast\ast}(v)\big)G_A^{\ast}(v,\rho).
\end{equation}
We use the bounds for $\rho\in \R$, which follows from simple calculations in view of \eqref{LAP100}-\eqref{LAP101},
\begin{equation}\label{Enh5.04}
\big|V_A^\ast(\rho)-V_A^{\ast\ast}(\rho)\big|\lesssim e^{-A/2}\big[e^{-|\rho-w|}+e^{-|\rho|}\big].
\end{equation}
Denote for $v,\rho\in\R$ and $L_A:=\big|\big[\min\{\rho,v\}, \max\{v,\rho\}\big]\cap\big[w+A,-A\big]\big|$, the weight function
\begin{equation}\label{Enh5.05}
\zeta_{A}(v,\rho):=e^{k|v-\rho|}e^{(\mu_k-k)L_A}. 
\end{equation}
Then for $\sigma\in\R$, it is easy to check that
\begin{equation}\label{Enh5.06}
\zeta_A(v,\rho)\leq \zeta_A(v,\sigma)\zeta_A(\sigma,\rho).
\end{equation}
Let 
\begin{equation}\label{Enh5.07}
\alpha:=\sup_{v,\rho\in\R} \big[\zeta_A(v,\rho) |g_A(v,\rho)|\Big], 
\end{equation}
which is finite since $|g_A(v,\rho)|\lesssim (1/k)e^{-k|v-\rho|}$.
By Lemma \ref{POTT1}, we have for $v,\rho\in\R$,
\begin{equation}\label{Enh5.071}
k\zeta_A(v,\rho)G_A^{\ast}(v,\rho)\lesssim1.
\end{equation}
Using \eqref{Enh5.03}, \eqref{Enh5.04} and \eqref{Enh5.071} we can bound
\begin{equation}\label{Enh5.08}
\begin{split}
&k\zeta_A(v,\rho) |g_A(v,\rho)|\lesssim \int_\R k\zeta_A(v,\rho)G_A^{\ast}(v,\sigma)e^{-A/2}\big[e^{-|\sigma-w|/2}+e^{-|\sigma|/2}\big]\big(|g_A(\sigma,\rho)|+G_A^{\ast}(\sigma,\rho)\big)\,d\sigma\\
&\lesssim e^{-A/2}\alpha+e^{-A/2},
\end{split}
\end{equation}
which implies that
\begin{equation}\label{Enh5.09}
\alpha\lesssim e^{-A/2}\alpha+e^{-A/2}.
\end{equation}
Choosing $A$ large we obtain the desired bounds \eqref{Enh5.01} from \eqref{Enh5.02}, \eqref{Enh5.071} and \eqref{Enh5.09}.

{\bf Step 2: proof of the bounds \eqref{Enh3.5}.}
We now turn to the proof of \eqref{Enh3.5}. For $w\leq-10$, $v,\rho\in\R$, we note that the quantity
 $$\frac{1}{k} e^{-k|v-\rho|} e^{-(\mu_k-k) d}$$ varies at most by $O(1)$ over an interval of $v$ of size $O(1/k)$. Denote for fixed $v, \rho\in\R$,
$$I_v:=[v-2/k, v+2/k]\qquad {\rm and} \qquad\gamma:=\frac{1}{k} e^{-k|v-\rho|} e^{-(\mu_k-k) d}.$$
Fix a smooth cutoff function $\varphi\in C_0^\infty(I_v)$ such that $\varphi\equiv 1$ on $[v-1/k, v+1/k]$ and $|\partial_{v'}\varphi(v')|\leq 2k$ for all $v'\in\R$. We notice the bounds
\begin{equation}\label{Enh9.01}
\sup_{v'\in I_v}\mathcal{G}_k^w(v',\rho)\lesssim\gamma.
\end{equation}
Using the equation \eqref{Enh2} and the bounds \eqref{Enh3} we obtain for some $C_0\in(0,\infty)$ that
\begin{equation}\label{Enh9.1}
\begin{split}
&\int_\R k^2|\mathcal{G}_k^w(v',\rho)|^2\varphi^2(v')+|\partial_{v'}\mathcal{G}_k^w(v',\rho)|^2\varphi^2(v')\,dv'\\
&\lesssim 2\int_{\R}|\partial_{v'}\mathcal{G}_k^w(v',\rho)||\mathcal{G}_k^w(v',\rho)||\varphi(v')\partial_{v'}\varphi(v')|\,dv'+\varphi^2(\rho)/k\\
&\leq C_0 |k|\gamma^2+\frac{1}{2}\int_\R |\partial_{v'}\mathcal{G}_k^w(v',\rho)|\varphi^2(v')\,dv',
\end{split}
\end{equation}
which implies that 
\begin{equation}\label{Enh9.11}
\int_{v-1/k}^{v+1/k}\big|\partial_{v'}\mathcal{G}_k^w(v',\rho)\big|^2\,dv'\lesssim k\gamma^2.
\end{equation}
Therefore there exists $v'\in [v-1/|k|, v+1/|k|]$ such that 
\begin{equation}\label{Enh9.2}
|\partial_{v'}\mathcal{G}_k^w(v',\rho)|\lesssim |k|\gamma.
\end{equation}
 The equation \eqref{Enh2} and the bounds \eqref{Enh3} imply that  
 \begin{equation}\label{Enh9.3}
\sup_{v''\in I_v, v''\neq \rho}|\partial^2_{v''}\mathcal{G}_k^w(v'',\rho)|\lesssim k^2\gamma.
\end{equation}
The desired bounds \eqref{Enh3.5} follow from \eqref{Enh9.2}-\eqref{Enh9.3}, together with the fact that at $v''=\rho$ the jump in the derivative of $\partial_v\mathcal{G}_k^w(v,\rho)$ is of unit size.

{\bf Step 3: proof of \eqref{Enh3.51}.} Recall that for $v,\rho\in\R$,
\begin{equation}\label{Green1}
(k^2-\partial_v^2)\mathcal{G}_k^w(v,\rho)+V_w(v)\mathcal{G}_k^w(v,\rho)=\delta(v-\rho).
\end{equation}
Thus for $v,\rho\in\R$,
\begin{equation}\label{Green2}
\begin{split}
&(k^2-\partial_v^2)\partial_v\mathcal{G}_k^w(v,\rho)+V_w(v)\partial_{v}\mathcal{G}_k^w(v,\rho)+\partial_vV_w(v)\mathcal{G}_k^w(v,\rho)=\partial_v\delta(v-\rho),\\
&(k^2-\partial_v^2)\partial_\rho\mathcal{G}_k^w(v,\rho)+V_w(v)\partial_{\rho}\mathcal{G}_k^w(v,\rho)=-\partial_v\delta(v-\rho).
\end{split}
\end{equation}
If we write
\begin{equation}\label{Green3}
\partial_v\mathcal{G}_k^w(v,\rho)=-\partial_\rho\mathcal{G}_k^w(v,\rho)+\mathcal{H}_k^w(v,\rho),
\end{equation}
then 
\begin{equation}\label{Green4}
(k^2-\partial_v^2)\mathcal{H}_k^w(v,\rho)+V_w(v)\mathcal{H}_k^w(v,\rho)=-\partial_vV_w(v)\mathcal{G}_k^w(v,\rho).
\end{equation}
Therefore 
\begin{equation}\label{Green5}
\mathcal{H}_k^w(v,\rho)=-\int_\R \mathcal{G}_k^w(v,\sigma)\partial_\sigma V_w(\sigma) \mathcal{G}_k^w(\sigma,\rho)\,d\sigma,
\end{equation}
and 
\begin{equation}\label{Green6}
\big|\partial_v\mathcal{H}_k^w(v,\rho)\big|=\Big|\int_\R \partial_v\mathcal{G}_k^w(v,\sigma)\partial_\sigma V_w(\sigma) \mathcal{G}_k^w(\sigma,\rho)\,d\sigma\Big|\lesssim k^{-1}e^{-k|v-\rho|-(\mu_k-k)d}.
\end{equation}
The desired bound \eqref{Enh3.51} follows from \eqref{Green1}, \eqref{Green3} and \eqref{Green6}.

\end{proof}



For later applications, we prove the following refined property for the Green's function $\mathcal{G}_k^w(v,\rho)$. 
\begin{proposition}\label{ENHQ1}
Assume that $k\in\Z\backslash\{0\}$ and $\mu_k:=\sqrt{k^2+8}$. Fix $\Psi_1\in C_0^\infty(-8,8)$ with $\Psi_1\equiv1$ on $[-1, 1]$ and $\sup_{\xi\in\R}\big|e^{\langle\xi\rangle^{7/8}}\widehat{\,\,\Psi_1\,\,}(\xi)\big|\lesssim1$. Let $w_{\ast}\in(-\infty,-15]$. Define for $v,\rho\in\R$ and $w\leq-5$,
\begin{equation}\label{ENHQ2}
\mathcal{F}_k(v,\rho,w):=\mathcal{G}_k^w(v+w,\rho+w), 
\end{equation}
and for notational convenience (for this proposition only) define also for $v,\rho,w\in\R$,
\begin{equation}\label{ENH12.1}
\qquad Q_k^{w_\ast}(v,\rho,w):=\mathcal{F}_k(v,\rho,w)\Psi_1(w-w_\ast).
\end{equation}
Let
\begin{equation}\label{ENHQ3}
\widehat{\,\,Q_k^{w_\ast}}(v,\rho,\xi):=\int_\R Q_k^{w_\ast}(v,\rho,w) e^{-iw\xi}\,dw, \qquad{\rm for}\,\,\xi\in\R.
\end{equation}
Then we have the following bounds for all $k\in\Z\backslash\{0\}$, $v,\rho\in\R$, 
\begin{equation}\label{ENHQ4}
\sup_{\xi\in\R}\Big[e^{2\delta_0\langle\xi\rangle^{1/2}}\big|\widehat{\,\,Q_k^{w_\ast}}(v,\rho,\xi)\big|\Big]\lesssim |k|^{-1}\varpi_{\varkappa_k,w_\ast}(v+w_\ast,\rho+w_\ast)
\end{equation}
In addition, we also have the bounds on the derivatives for $v,\rho\in\R$,
\begin{equation}\label{ENHQ5}
\sup_{\xi\in\R}\Big[e^{2\delta_0\langle\xi\rangle^{1/2}}\big|\partial_v\widehat{\,\,Q_k^{w_\ast}}(v,\rho,\xi)\big|\Big]\lesssim \varpi_{\varkappa_k,w_\ast}(v+w_\ast,\rho+w_\ast).
\end{equation}
and
\begin{equation}\label{ENHQ6}
\sup_{\xi\in\R}\Big[e^{2\delta_0\langle\xi\rangle^{1/2}}\Big|\partial_{v\rho}\widehat{\,\,Q_k^{w_\ast}}(v,\rho,\xi)-\delta(v-\rho)\widehat{\,\,\Psi_1}(\xi)\Big|\Big]\lesssim |k|\varpi_{\varkappa_k,w_\ast}(v+w_\ast,\rho+w_\ast).
\end{equation}
\end{proposition}

\begin{proof}
The main idea is to use the pointwise bounds in lemma \ref{Enh1} and the commutator argument. Since the parameter $w_\ast$ is fixed throughout the proof, for the simplicity of notations, we set $Q_k(v,\rho, w):=Q_k^{w_\ast}(v,\rho,w)$. By \eqref{Enh3} we have
\begin{equation}\label{ENHQ6.1}
\big\|Q_k(v,\rho,w)\big\|_{L^2(w\in\R)}\lesssim |k|^{-1}\varpi_{\varkappa_k,w_\ast}(v+w_\ast,\rho+w_\ast).
\end{equation}
Let $A$ be the Fourier multiplier operator such that for any $h\in L^2(\R)$,
\begin{equation}\label{ENHQ7}
\widehat{\,\,Ah\,\,}(\xi):=e^{2\delta_0\langle\xi\rangle^{1/2}}\widehat{\,\,h\,\,}(\xi), \qquad{\rm for}\,\,\xi\in\R.
\end{equation}
It follows from the definition \eqref{Enh3} that for $v,\rho, w\in\R$, 
\begin{equation}\label{ENHQ8}
(k^2-\partial_v^2)Q_k(v,\rho,w)+V_{w}(v+w)Q_k(v,\rho,w)=\delta(v-\rho)\Psi_1(w-w_\ast).
\end{equation}
Applying the operator $A$ (in the variable $w$) to \eqref{ENHQ9}, we obtain that
\begin{equation}\label{ENHQ9}
\begin{split}
&(k^2-\partial_v^2)A\big[Q_k(v,\rho,\cdot)\big](w)+V_{w}(v+w)A\big[Q_k(v,\rho,\cdot)\big](w)\\
&=\delta(v-\rho)A\big[\Psi_1(\cdot-w_\ast)\big](w)+V_{w}(v+w)A\big[ Q_k(v,\rho,\cdot)\big](w)-A\big[V_{\cdot}(v+\cdot)Q_k(v,\rho,\cdot)\big](w).
\end{split}\end{equation}
We can reformulate \eqref{ENHQ9} in the integral form
\begin{equation}\label{ENHQ10}
\begin{split}
&A\big[Q_k(v,\rho,\cdot)\big](w)=\mathcal{G}_k^w(v+w,\rho+w)A\big[\Psi_1(\cdot-w_\ast)\big](w)\\
&+\int_\R \mathcal{G}_k^w(v+w,\sigma+w) \bigg\{V_{w}(\sigma+w)A\big[Q_k(\sigma,\rho,\cdot)\big](w)-A\big[V_{\cdot}(\sigma+\cdot)Q_k(\sigma,\rho,\cdot)\big](w)\bigg\}\,d\sigma.
\end{split}
\end{equation}
Denote
\begin{equation}\label{ENHQ11}
M:=|k|\sup_{v,\rho\in\R}\Big[\left\|A\big[Q_k(v,\rho,\cdot)\big](w)\right\|_{L^2(w\in\R)}/(\varpi_{k,w_\ast}(v+w_\ast,\rho+w_\ast))\Big].
\end{equation}
Fix smooth cutoff functions $\Psi_2\in C_0^\infty(-17/2,17/2)$ with $\Psi_2\equiv1$ on $[-33/4, 33/4]$ and $$\sup_{\xi\in\R}\big|e^{\langle\xi\rangle^{7/8}}\widehat{\Psi_2}(\xi)\big|\lesssim1,$$ and $\Psi_3\in C_0^\infty(-9,9)$ with $\Psi_3\equiv[-35/4,35/4]$ and $\sup_{\xi\in\R}\big|e^{\langle\xi\rangle^{7/8}}\widehat{\,\Psi_3}(\xi)\big|\lesssim1$. We define 
\begin{equation}\label{ENHQ11.1}
V_{\sigma,w_\ast}(w):=V_{w}(\sigma+w)\Psi_3(w-w_\ast), \qquad {\rm if}\,\,\sigma\in\R\backslash[-1,-w_\ast+1]
\end{equation}
and 
\begin{equation}\label{ENHQ11.2}
V_{\sigma,w_\ast}(w):=\big[V_{w}(\sigma+w)-8\big]\Psi_3(w-w_\ast), \qquad {\rm if}\,\,\sigma\in[-1,-w_\ast+1].
\end{equation}
We have the bound (which follows from simple calculations in view of \eqref{LAP100}-\eqref{LAP101}) for $\sigma\in\R$,
\begin{equation}\label{ENHQ13}
\sup_{\xi\in\R}\big|e^{3\delta_0\langle\xi\rangle^{1/2}}\widehat{\,V_{\sigma,w_\ast}}(\xi)\big|\lesssim e^{-|\sigma|}+e^{-|\sigma+w_\ast|}.
\end{equation}
We observe that for $\sigma, \rho, w\in\R$,
\begin{equation}\label{ENHQ13.1}
\begin{split}
&\Psi_2(w-w_\ast)V_{w}(\sigma+w)A\big[Q_k(\sigma,\rho,\cdot)\big](w)-\Psi_2(w-w_\ast)A\big[V_{\cdot}(\sigma+\cdot)Q_k(\sigma,\rho,\cdot)\big](w)\\
&=\Psi_2(w-w_\ast)V_{\sigma,w_\ast}(w)A\big[Q_k(\sigma,\rho,\cdot)\big](w)-\Psi_2(w-w_\ast)A\big[V_{\sigma,w_\ast}(\cdot)Q_k(\sigma,\rho,\cdot)\big](w).
\end{split}
\end{equation}
It follows from \eqref{ENHQ13}-\eqref{ENHQ13.1} and the bound \eqref{ENHQ6.1} that for any $\gamma\in(0,1)$ and suitable $C_\gamma\in(0,\infty)$,
\begin{equation}\label{ENHQ12}
\begin{split}
&\left\|\Psi_2(w-w_\ast)V_{w}(\sigma+w)A\big[Q_k(\sigma,\rho,\cdot)\big](w)-\Psi_2(w-w_\ast)A\big[V_{\cdot}(\sigma+\cdot)Q_k(\sigma,\rho,\cdot)\big](w)\right\|_{L^2(w\in\R)}^2\\
&\lesssim\int_{\R}\bigg[\int_{\R^2}\big|\widehat{\,\,\,V_{\sigma,w_\ast}}(\alpha)\big|\big|\widehat{\,\,\Psi_2}(\beta)\big|\big|A(\xi-\alpha-\beta)-A(\xi-\beta)\big|\big|\widehat{\,\,Q_k}(\sigma,\rho,\xi-\alpha-\beta)\big|\,d\alpha d\beta\bigg]^2d\xi\\
&\lesssim \big[e^{-|\sigma|}+e^{-|\sigma+w_\ast|}\big] \int_{\R}\bigg[\int_{\R^2}e^{-\delta_0\langle\alpha\rangle^{1/2}-\langle\beta\rangle^{4/5}}\langle\xi-\alpha-\beta\rangle^{-1/2}\big|\widehat{\,\,Q_k}(v,\rho,\xi-\alpha-\beta)\big|\,d\alpha d\beta\bigg]^2d\xi \\
&\lesssim \big[e^{-|\sigma|}+e^{-|\sigma+w_\ast|}\big]\Big[C_{\gamma} \left\|Q_k(\sigma,\rho,w)\right\|_{L^2(w\in\R)}^2+\gamma\left\|A\big[Q_k(\sigma,\rho,\cdot)\big](w)\right\|_{L^2(w\in\R)}^2\Big]\\
&\lesssim \big[e^{-|\sigma|}+e^{-|\sigma+w_\ast|}\big]\Big[C_{\gamma} \big(\varpi_{\varkappa_k,w_\ast}(\sigma+w_\ast,\rho+w_\ast)/|k|\big)^2+\gamma\big(M\varpi_{\varkappa_k,w_\ast}(\sigma+w_\ast,\rho+w_\ast)/|k|\big)^2\Big].
\end{split}
\end{equation}
In the above, the parameter $\gamma$ is small and is used to divide the frequency to the cases $|\xi|>\gamma^{-1}$ and $|\xi|\leq\gamma^{-1}$, to obtain the crucial gain of the factor $\gamma$. It follows from \eqref{ENHQ12} that
\begin{equation}\label{ENHQQ1}
\begin{split}
&\Big\|\Psi_2(w-w_\ast)\int_\R \mathcal{G}_k^w(v+w,\sigma+w) \bigg\{V_{w}(\sigma+w)A\big[Q_k(\sigma,\rho,\cdot)\big](w)\\
&\hspace{2.3in}-A\big[V_{w}(\sigma+\cdot)Q_k(\sigma,\rho,\cdot)\big](w)\bigg\}\,d\sigma\Big\|_{L^2(w\in\R)}\\
&\lesssim C_{\gamma} \varpi_{\varkappa_k,w_\ast}(v+w_\ast,\rho+w_\ast)/|k|+\gamma^{1/2}M\varpi_{\varkappa_k,w_\ast}(v+w_\ast,\rho+w_\ast)/|k|.
\end{split}
\end{equation}

 We also have for all $\gamma\in(0,1)$ and suitable $C_\gamma\in(0,\infty)$,
\begin{equation}\label{ENHQ12.1}
\begin{split}
&\left\|A\big[Q_k(v,\rho,\cdot)\big](w)-\Psi_2(w-w_\ast)A\big[Q_k(v,\rho,\cdot)\big](w)\right\|_{L^2(w\in\R)}^2\\
&=\left\|A\big[\Psi_2(\cdot-w_\ast)Q_k(v,\rho,\cdot)\big](w)-\Psi_2(w-w_\ast)A\big[Q_k(v,\rho,\cdot)\big](w)\right\|_{L^2(w\in\R)}^2\\
&\lesssim \int_{\R} \bigg[\int_\R \big|\widehat{\,\,\Psi_2}(\alpha)\big| \big|A(\xi-\alpha)-A(\xi)\big| \big|\widehat{\,\,Q_k}(\sigma,\rho,\xi-\alpha)\big|\,d\alpha\bigg]^2\,d\xi\\
&\lesssim C_{\gamma} \big(\varpi_{\varkappa_k,w_\ast}(v+w_\ast,\rho+w_\ast)/|k|\big)^2+\gamma\left\|A\big[Q_k(v,\rho,\cdot)\big](w)\right\|^2_{L^2(w\in\R)},
\end{split}
\end{equation}
which together with \eqref{ENHQQ1}, upon choosing $\gamma\in(0,1)$ sufficiently small, implies that
\begin{equation}\label{ENHQ12.2}
\left\|A\big[Q_k(v,\rho,\cdot)\big](w)\right\|_{L^2(w\in\R)}\lesssim\left\|\Psi_2(w-w_\ast)A\big[Q_k(v,\rho,\cdot)\big](w)\right\|_{L^2(w\in\R)} +\frac{C_{\gamma}}{|k|} \varpi_{\varkappa_k,w_\ast}(v+w_\ast,\rho+w_\ast).
\end{equation}
 It follows from \eqref{ENHQ10}-\eqref{ENHQ12} and \eqref{ENHQ12.2} that for sufficiently small $\gamma\in(0,1)$ and suitable $C_\gamma\in(0,\infty)$,
\begin{equation}\label{ENHQ14}
\begin{split}
&\left\|A\big[Q_k(v,\rho,\cdot)\big](w)\right\|_{L^2(w\in\R)}\lesssim C_\gamma \varpi_{k,w_\ast}(v+w_\ast,\rho+w_\ast)/|k|+\gamma^{1/2} M \varpi_{\varkappa_k,w_\ast}(v+w_\ast,\rho+w_\ast)/|k|,
\end{split}
\end{equation}
which implies that
\begin{equation}\label{ENHQ15}
M\lesssim C_\gamma+\gamma^{1/2} M.
\end{equation}
Choosing $\gamma\in(0,1)$ sufficiently small, we conclude that 
\begin{equation}\label{ENHQ16}
M\lesssim 1.
\end{equation}
The desired bounds \eqref{ENHQ4} follow from \eqref{ENHQ16} by noting that $Q_k(v,\rho,w)=Q_k(v,\rho,w)\Psi_2(w-w_\ast)$. The bounds \eqref{ENHQ5} follow from taking derivative in \eqref{ENHQ10} and estimating the resulting expression. The bounds \eqref{ENHQ6} then follow from \eqref{Green3}-\eqref{Green5}, the equation \eqref{ENHQ8}, and the bounds \eqref{ENHQ4}. We omit the routine details.

\end{proof}

 \section{Spectrum of the linearized operator}\label{sec:spec}
 In this section we study the spectrum of $L_k$ for $k\in\Z\backslash\{0\}$.
 Our main result is the following characterization of the spectrum of $L_k$ for $k\in\Z\backslash\{0\}$. The result is not new, see for example \cite{Bed2} and references therein for more discussions and other aspects of the operator $L_k$. The simple proof below is based on the argument in \cite{Briggs}. 
\begin{proposition}\label{SPEC}
Assume that $k\in\Z\backslash\{0\}$. Recall the definition \eqref{spec1}-\eqref{sped1.1} for the space $X_k$. The spectrum of $L_k: X_k\to X_k$ is $[0, b(0)]$. For $|k|=1$, $\lambda=0$ is an embedded eigenvalue for $L_k$ with the corresponding eigenfunction $\Omega'(r), r\in\R^+$, and there are no other discrete eigenvalues. For $k\in\Z\backslash\{0\}$ and $|k|\neq1$, there are no discrete eigenvalues.
\end{proposition} 

\begin{remark}
The discrete eigenvalue $\lambda=0$ for $L_k$ with $|k|=1$ is connected to the translation symmetry of the background flow, and can be treated by re-centering the radial vorticity profile, see the orthogonality condition \eqref{B3}.
\end{remark}
 
 \begin{proof}
 We can assume without loss of generality $k\in\Z\cap[1,\infty)$. Suppose $g\in X_k$ with $\|g\|_{X_k}=1$ is an eigenfunction corresponding to the eigenvalue $\lambda\in\R$. Writing for $r\in\R^+$,
 \begin{equation}\label{spec2}
 g(r)=\partial_r^2\varphi(r)+\frac{\partial_r\varphi(r)}{r}-k^2r^{-2}\varphi(r),
 \end{equation}
we obtain that for $r\in\R^+$,
 \begin{equation}\label{spec3}
 \Big(\frac{U(r)}{r}-\lambda\Big)\big(\partial_r^2\varphi+\frac{\partial_r\varphi}{r}-k^2r^{-2}\varphi\big)-\frac{\Omega'(r)}{r}\varphi=0.
 \end{equation}
 Note that 
 \begin{equation}\label{spec4}
 \begin{split}
 \varphi(r)=-\int_0^\infty G_k(r,\rho)g(\rho)\,d\rho=\int_0^\infty\frac{\rho}{2k}\min\Big\{\big(r/\rho\big)^k, (\rho/r)^k\Big\}g(\rho)\,d\rho.
 \end{split}
 \end{equation}
 It follows from \eqref{Back} and \eqref{spec4}, $\|g\|_{X_k}=1$ and Cauchy-Schwarz inequality that for $r\in\R^+$,
 \begin{equation}\label{spec5}
 |\varphi(r)|+|r\partial_r\varphi(r)|=O\Big(\frac{r\langle \log{r}\rangle}{\langle r\rangle^2}\Big).
 \end{equation}
 Define for $r\in\R^+$,
 \begin{equation}\label{spec6}
 h(r):=\frac{\varphi(r)}{r(\lambda-U(r)/r)}.
 \end{equation}

 We distinguish several cases.
 
 {\bf Case 1: $\lambda\not\in[0,b(0)]$.}
  Then from \eqref{spec3} and \eqref{spec5} we see that $h\in C^{\infty}(\R^+)$ and for $r\in\R^+$,
 \begin{equation}\label{spec6.1}
 |h(r)|\lesssim_\lambda \frac{\langle\log{r}\rangle}{\langle r\rangle^2}, \quad |h'(r)|\lesssim_\lambda \frac{\langle\log{r}\rangle}{r\langle r\rangle^2}.
 \end{equation}
 Direct calculations show
 \begin{equation}\label{spec7}
 \begin{split}
 \varphi(r)&=r(\lambda-U(r)/r)h(r)=(\lambda r-U(r))h(r),\\
\frac{\varphi'(r)}{r}&=\frac{\lambda-U'(r)}{r}h(r)+\big(\lambda-\frac{U(r)}{r}\big)h'(r),\\
\varphi''(r)&=-U''(r)h(r)+2(\lambda-U'(r))h'(r)+(\lambda r-U(r))h''(r).
 \end{split}
 \end{equation}
 Thus from \eqref{spec3} we obtain that
 \begin{equation}\label{spec8}
 \begin{split}
& (\lambda r-U(r))h''(r)+\Big[2(\lambda-U'(r))+\lambda-\frac{U(r)}{r}\Big]h'(r)\\
&\qquad+\Big[-U''(r)+\frac{\lambda-U'(r)}{r}-\frac{k^2(\lambda-U(r)/r)}{r}+\Omega'(r)\Big]h(r)=0.
 \end{split}
 \end{equation}
 Using the identity
 \begin{equation}\label{spec9}
 U'(r)+\frac{U(r)}{r}=\Omega(r),
 \end{equation}
 we obtain that
 \begin{equation}\label{spec10}
 U''(r)+\frac{U'(r)}{r}-\frac{U(r)}{r^2}=\Omega'(r).
 \end{equation}
 We can use \eqref{spec10} to simplify \eqref{spec8} as
 \begin{equation}\label{spec11}
 (\lambda r-U(r))h''(r)+\Big[2(\lambda -U'(r))+(\lambda-U(r)/r)\Big]h'(r)+\frac{1-k^2}{r}\big(\lambda-U(r)/r\big)h(r)=0,
 \end{equation}
 or equivalently
 \begin{equation}\label{spec12}
 \begin{split}
& r(\lambda r-U(r))^2h''(r)+\Big[(\lambda r-U(r))^2+2r(\lambda r-U(r))(\lambda-U'(r))\Big]h'(r)\\
&\qquad+(1-k^2) r(\lambda-U(r)/r)^2h(r)=0.
 \end{split}
 \end{equation}
 \eqref{spec12} can be reformulated as
 \begin{equation}\label{spec13}
 \frac{d}{dr}\Big[r(\lambda r-U(r))^2h'(r)\Big]+(1-k^2)r(\lambda-U(r)/r)^2h(r)=0.
 \end{equation}
 Multiplying $\overline{h}$ and integrating over $(0,\infty)$, using also \eqref{spec6.1} to treat the boundary terms, we obtain that
 \begin{equation}\label{spec14}
 \int_0^\infty r(\lambda r-U(r))^2|h'(r)|^2+(k^2-1)r(\lambda-U(r)/r)^2|h(r)|^2\,dr=0.
 \end{equation}
 Therefore, $h\equiv 0$ if $k\neq1$, a contradiction with the assumption that $\|g\|_{X_k}=1$. \eqref{spec14} also implies that $h\equiv C$ for some nonzero constant $C$ if $k=1$,  and thus $\varphi(r)=C(\lambda r-U)$. Consequently from \eqref{spec5} we have $\lambda=0$, a contradiction with $\lambda\not\in[0,b(0)]$. In summary, we conclude that if $\lambda\not\in[0,b(0)]$, then $\lambda$ cannot be a discrete eigenvalue of $L_k$.
 
 {\bf Case 2: $\lambda\in(0,b(0))$.} We can assume that $\lambda=b(r_0)=U(r_0)/r_0$ for some $r_0\in(0,\infty)$. From \eqref{spec3} and \eqref{Back}, it follows that $\varphi(r_0)=0$. Using \eqref{spec5} and \eqref{spec3}, we see also that $\varphi\in C^{\infty}(\R^+)$. It follows that $h\in C^\infty(\R^+)$ and the following bounds hold for $r\in\R^+$,
 \begin{equation}\label{sped15}
 |h(r)|\lesssim_\lambda \frac{\langle\log{r}\rangle}{\langle r\rangle^2}, \quad |h'(r)|\lesssim_\lambda\frac{\langle\log{r}\rangle}{r\langle r\rangle^2}.
 \end{equation}
 We can then use the same argument as in Case 1 to prove that $\lambda\in(0,b(0))$ can not be a discrete eigenvalue for $L_k$.
 
 {\bf Case 3: $\lambda=b(0)$.} In this case, we note from \eqref{spec3}, using \eqref{Back} and \eqref{Int0.1}, that for $r\in\R^+$
 \begin{equation}\label{spec16}
  \big(\partial_r^2\varphi+\frac{\partial_r\varphi}{r}-k^2r^{-2}\varphi\big)-\frac{\Omega'(r)}{r(b(r)-b(0))}\varphi=0,
  \end{equation}
 and 
 \begin{equation}\label{spec17}
 \frac{\Omega'(r)}{r(b(r)-b(0))}>0.
 \end{equation}
 By the maximum principle and the bounds \eqref{spec5}, we must have $\varphi\equiv0$, a contradiction with $\|g\|_{X_k}=1$. Therefore $\lambda=b(0)$ cannot be a discrete eigenvalue.
 
 {\bf Case 4: $\lambda=0$.} In this case, using \eqref{spec3} and \eqref{spec5} we obtain $h\in C^\infty(\R^+)$ and that for $r\in\R^+$,
 \begin{equation}\label{sped18}
 |h(r)|\lesssim_\lambda \langle\log{r}\rangle, \quad |h'(r)|\lesssim_\lambda\frac{\langle\log{r}\rangle}{r}.
 \end{equation}
 We can then repeat the argument in Case 1, and conclude that $k=1$. The corresponding eigenfunction $g$ for $\lambda=0$ can be computed from $\varphi=-U(r)$, using \eqref{Ur} and \eqref{spec2} as
 $$g=-\partial_r^2U-\frac{\partial_rU}{r}+k^2r^{-2}U=-\Omega'(r).$$
 
 Combining cases 1-4, we then complete the proof.
  \end{proof}

\section{The limiting absorption principle}\label{sec:lap}
In this section we study the spectral density functions $\Pi^+_{k,\epsilon}(v,w)$ and $\Pi^-_{k,\epsilon}(v,w), v, w\in\R$ for $k\in\Z\backslash\{0\}$, $\epsilon>0$ and sufficiently small. Recall from \eqref{B13} that $\Pi^\iota_{k,\epsilon}, \iota\in\{+,-\}$ satisfy the equation that for $v,w\in\R$,
\begin{equation}\label{LAP1}
(k^2-\partial_v^2)\Pi^\iota_{k,\epsilon}(v,w)+\frac{e^{2v}D(v)\Pi^\iota_{k,\epsilon}(v,w)}{B(v)-B(w)+i\iota\epsilon}=\frac{e^{2v}f_0^k(v)}{B(v)-B(w)+i\iota\epsilon}.
\end{equation}
To study \eqref{LAP1}, we distinguish two cases: $w\ge-20$ and $w\leq-10$. 

For $w\ge-20$, we can reformulate \eqref{LAP1} as
\begin{equation}\label{LAP1.1}
\Pi^\iota_{k,\epsilon}(v,w)+\frac{1}{2|k|}\int_\R e^{-|k||v-\rho|} \frac{e^{2\rho}D(\rho)\Pi^\iota_{k,\epsilon}(\rho,w)}{B(\rho)-B(w)+i\iota\epsilon}\,d\rho=\frac{1}{2|k|}\int_\R e^{-|k||v-\rho|}\frac{e^{2\rho}f_0^k(\rho)}{B(\rho)-B(w)+i\iota\epsilon}\,d\rho,
\end{equation}
and we can bound $\Pi_{k,\epsilon}^\iota$ by solving the integral equation using the spectral property proved in Proposition \ref{SPEC}.  

For $w\leq-10$, the situation is trickier, and it is important to notice that the potential 
$$\frac{e^{2v}D(v)}{B(v)-B(w)+i\iota\epsilon}\approx 8$$
for $v\in[w,0]$ which is a large interval for $|w|\gg1$. Therefore we need to incorporate part of the potential into the main term $k^2-\partial_v^2$, and, instead of using \eqref{LAP1.1}, we reformulate \eqref{LAP1} as
\begin{equation}\label{LAP1.2}
\begin{split}
&\Pi^\iota_{k,\epsilon}(v,w)+\int_\R \mathcal{G}_{k}^w(v,\rho) \Big[\frac{e^{2\rho}D(\rho)}{B(\rho)-B(w)+i\iota\epsilon}-V_w(\rho)\Big]\Pi_{k,\epsilon}^{\iota}(\rho,w)\,d\rho\\
&=\int_\R  \mathcal{G}_{k}^w(v,\rho)\frac{e^{2\rho}f_0^k(\rho)}{B(\rho)-B(w)+i\iota\epsilon}\,d\rho.
\end{split}
\end{equation}
 In the above we recall the definitions \eqref{Pot1} and \eqref{Enh2} for $V_w$ and $\mathcal{G}_k^w$. \eqref{LAP1.2} can be analyzed using the spectral property of $L_k$, see Proposition \ref{SPEC}, similar to the above, although the calculations are slightly more complicated.
 
\subsection{Limiting absorption principle for $w\ge-20$.} 
Define for $k\in\Z\backslash\{0\}$, and $k_\ast\in\Z$ with $1\leq|k_\ast|\leq|k|$,
\begin{equation}\label{LAP7}
Y_{k,k_\ast}:=\bigg\{h\in L^2(\R):\,\sup_{j\in\Z}\Big[\big\|e^{|k_\ast||v|}h\big\|_{L^2(j,j+2)}+|k|^{-1}\big\|e^{|k_\ast||v|}\partial_vh\big\|_{L^2(j,j+2)}\Big]<\infty\bigg\},
\end{equation}
with the natural norm
\begin{equation}\label{LAP8}
\|h\|_{Y_{k,k_\ast}}:=\sup_{j\in\Z}\Big[\big\|e^{|k_\ast||v|}h\big\|_{L^2(j,j+2)}+|k|^{-1}\big\|e^{|k_\ast||v|}\partial_vh\big\|_{L^2(j,j+2)}\Big].
\end{equation}
For $j\in\Z$, we also introduce for $h\in Y_{k,k_\ast}$, the norm
\begin{equation}\label{LAP8.1}
\|h\|_{Y_{k,k_\ast}^j}:=\big\|e^{|k_\ast||v|}h\big\|_{L^2(j,j+2)}+|k|^{-1}\big\|e^{|k_\ast||v|}\partial_vh\big\|_{L^2(j,j+2)}.
\end{equation}
Define for $k\in\Z\backslash\{0\}$, $k_\ast\in\Z$ with $1\leq|k_\ast|\leq|k|$ and $\epsilon\in[-1/8,1/8]\backslash\{0\}$ the operator $T_{k,\epsilon}^w: Y_{k,k_\ast}\to Y_{k,k_\ast}$ as follows. For any $h\in Y_{k,k_\ast}$,
\begin{equation}\label{LAP9}
T_{k,\epsilon}^wh(v):=\frac{1}{2|k|}\int_{\R} e^{-|k||v-\rho|}\frac{e^{2\rho}D(\rho)h(\rho)}{B(\rho)-B(w)+i\epsilon}\,d\rho.
\end{equation}

\begin{lemma}\label{LAP9.1}
Assume that $w\ge-20$, $k\in\Z\backslash\{0\}$ and $k_\ast\in\Z$ with $1\leq|k_\ast|\leq|k|$. There exists $C\in(0,\infty)$ which can be chosen independent of $w,k, k_\ast$, such that for $\epsilon\in[-e^{-2|w|},e^{-2|w|}]\backslash\{0\}$ we have the following bounds for all $h\in Y_{k,k_\ast}$,
\begin{equation}\label{LAP11.0}
\left\|T_{k,\epsilon}^wh\right\|_{Y_{k,k_\ast}}\leq C|k|^{-1/5}\sum_{l\in\Z}e^{-|\ell|}\|h\|_{Y_{k,k_\ast}^\ell}\leq C|k|^{-1/5}\|h\|_{Y_{k,k_\ast}};
\end{equation}
similarly, 
\begin{equation}\label{LAP11.0'}
\begin{split}
&\sup_{j\in\Z}\Big[\big\|e^{|k_\ast||v-w|}T_{k,\epsilon}^wh(v)\big\|_{L^2(j,j+2)}+|k|^{-1}\big\|e^{|k_\ast||v-w|}\partial_vT_{k,\epsilon}^wh(v)\big\|_{L^2(j,j+2)}\Big]\\
&\leq C|k|^{-1/5}\sum_{j\in\Z}e^{-|j|}\Big[\big\|e^{|k_\ast||v-w|}h\big\|_{L^2(j,j+2)}+|k|^{-1}\big\|e^{|k_\ast||v-w|}\partial_vh\big\|_{L^2(j,j+2)}\Big],
\end{split}
\end{equation}
and in addition, for $w\ge0$, $v_\ast\in[0,w]$, we have
\begin{equation}\label{LAP11.0''}
\begin{split}
&\big\|e^{|k_\ast||v-w|}T_{k,\epsilon}^wh(v)\big\|_{L^2(v_\ast,v_\ast+2)}+|k|^{-1}\big\|e^{|k_\ast||v-w|}\partial_vT_{k,\epsilon}^wh(v)\big\|_{L^2(v_\ast,v_\ast+2)}\Big]\\
&\leq C|k|^{-1/5}e^{-|v_\ast|}\sum_{j\in\Z}e^{-|j|}\Big[\big\|e^{|k_\ast||v-w|}h\big\|_{L^2(j,j+2)}+|k|^{-1}\big\|e^{|k_\ast||v-w|}\partial_vh\big\|_{L^2(j,j+2)}\Big].
\end{split}
\end{equation}

\end{lemma}

\begin{proof}
We focus on \eqref{LAP11.0}, as the proof of \eqref{LAP11.0'}-\eqref{LAP11.0''} is similar and we will indicate at the end of the proof the additional details that are needed. 
We can assume that $k\ge1$ and $1\leq k_\ast\leq k$, without loss of generality. Choose a smooth cutoff function $\varphi\in C_0^\infty(-2,2)$ with $\varphi\equiv1$ on $[-1,1]$. Set $\varphi_k(v):=\varphi(\sqrt{k}\,v), v\in\R$. We can write 
\begin{equation}\label{LA1}
\begin{split}
T_{k,\epsilon}^wh(v)=\mathcal{J}_1h(v)+\mathcal{J}_2h(v)&:=\frac{1}{2k}\int_\R e^{-k|v-\rho|}\varphi_k(\rho-w)\frac{e^{2\rho}D(\rho)}{B(\rho)-B(w)+i\epsilon}h(\rho)\,d\rho\\
&+\frac{1}{2k}\int_\R e^{-k|v-\rho|}(1-\varphi_k(\rho-w))\frac{e^{2\rho}D(\rho)}{B(\rho)-B(w)+i\epsilon}h(\rho)\,d\rho
\end{split}
\end{equation}
In the above we have suppressed the dependence of $\mathcal{J}_1$ and $\mathcal{J}_2$ on $k, \epsilon$, for the simplicity of notations. We shall use this convention often, when there is no possibility of confusion. Setting $h^\ast(v)=e^{k_\ast|v|}h(v), v\in\R$, then for $j\in\Z$,
\begin{equation}\label{LA2}
\|h^\ast\|_{L^2(j,j+2)}+k^{-1}\|\partial_vh^\ast\|_{L^2(j,j+2)}\lesssim\|h\|_{Y_{k,k_\ast}^j}.
\end{equation}
Using 
\begin{equation}\label{LA2.1}
\begin{split}
&|B(\rho)-B(w)|\gtrsim \frac{1}{\sqrt{k}}\frac{1}{1+e^{2\rho}}, \,\,{\rm for}\,\,\rho<w-\frac{1}{\sqrt{k}}, \\
& |B(\rho)-B(w)|\gtrsim e^{-2w}/\sqrt{k},\,\,{\rm for}\,\,\rho>w+\frac{1}{\sqrt{k}},
\end{split}
\end{equation}
we can bound 
\begin{equation}\label{LA3}
\begin{split}
e^{|k_\ast||v|}\big|\mathcal{J}_2h(v)\big|&=\frac{1}{2k}\Big|\int_\R e^{-k|v-\rho|}e^{k_\ast|v|}(1-\varphi_k(\rho-w))\frac{e^{2\rho}D(\rho)}{B(\rho)-B(w)+i\epsilon}h(\rho)\,d\rho\Big|\\
&\lesssim\frac{1}{\sqrt{k}}\int_{-\infty}^{w-\frac{1}{\sqrt{k}}}\frac{e^{2\rho}}{1+e^{6\rho}}h^\ast(\rho)\,d\rho+\frac{1}{\sqrt{k}}\int_{w+\frac{1}{\sqrt{k}}}^{\infty}\frac{e^{2(\rho+w)}}{1+e^{8\rho}}h^\ast(\rho)\,d\rho.\end{split}
\end{equation}
Similarly we can bound 
\begin{equation}\label{LA3.5}
k^{-1}e^{|k_\ast||v|}\big|\partial_v\big[\mathcal{J}_2h(v)\big]\big|\lesssim \frac{1}{\sqrt{k}}\int_{-\infty}^{w-\frac{1}{\sqrt{k}}}\frac{e^{2\rho}}{1+e^{6\rho}}h^\ast(\rho)\,d\rho+\frac{1}{\sqrt{k}}\int_{w+\frac{1}{\sqrt{k}}}^{\infty}\frac{e^{2(\rho+w)}}{1+e^{8\rho}}h^\ast(\rho)\,d\rho.
\end{equation}
It follows from \eqref{LA3} and \eqref{LA3.5} that
\begin{equation}\label{LA3.6}
\big\|\mathcal{J}_2h\big\|_{Y_{k,k_\ast}}\lesssim \frac{1}{\sqrt{k}}\sum_{j\in\Z}e^{-|j|}\|h\|_{Y_{k,k_\ast}^j}.
\end{equation}
We can bound $\mathcal{J}_1h(v)$ as follows,
\begin{equation}\label{LA4}
\begin{split}
&e^{k_\ast|v|}\big|\mathcal{J}_1h(v)\big|=\frac{1}{2k}\Big|\int_\R e^{-k|v-\rho|}e^{k_\ast|v|}\varphi_k(\rho-w)\frac{e^{2\rho}D(\rho)}{B(\rho)-B(w)+i\epsilon}h(\rho)\,d\rho\Big|\\
&\lesssim \frac{1}{2k}\Big|\int_\R e^{-k|v-\rho|}e^{k_\ast|v|}\varphi_k(\rho-w)\frac{e^{2\rho}D(\rho)}{B'(\rho)}h(\rho)\partial_\rho \log{\frac{B(\rho)-B(w)+i\epsilon}{B'(w)}}\,d\rho\Big|\\
&\lesssim \frac{1}{2k}\Big|\int_\R \partial_\rho\Big[e^{-k|v-\rho|}e^{k_\ast|v|}\varphi_k(\rho-w)\frac{e^{2\rho}D(\rho)}{B'(\rho)}\Big]h(\rho)\log{\frac{B(\rho)-B(w)+i\epsilon}{B'(w)}}\,d\rho\Big|\\
&\quad+\frac{1}{2k}\Big|\int_\R e^{-k|v-\rho|}e^{k_\ast|v|}\varphi_k(\rho-w)\frac{e^{2\rho}D(\rho)}{B'(\rho)}h'(\rho)\log{\frac{B(\rho)-B(w)+i\epsilon}{B'(w)}}\,d\rho\Big|\\
&\lesssim\int_{|\rho-w|<2/\sqrt{k}}\frac{e^{4\rho}}{1+e^{8\rho}}\Big[|h^\ast(\rho)|+e^{k_\ast|\rho|}|h'(\rho)|/k \Big]\,\big\langle\log{|\rho-w|}\big\rangle\,d\rho.
\end{split}
\end{equation}
Similarly, using integration by parts, we have
\begin{equation}\label{LA7}
\begin{split}
&k^{-1}e^{k_\ast|v|}\big|\partial_v\big[\mathcal{J}_1h(v)\big]\big|=\frac{1}{2k^2}\Big|\int_\R \partial_ve^{-k|v-\rho|}e^{k_\ast|v|}\varphi_k(\rho-w)\frac{e^{2\rho}D(\rho)}{B(\rho)-B(w)+i\epsilon}h(\rho)\,d\rho\Big|\\
&\lesssim  \frac{1}{2k^2}\Big|\int_\R \partial_ve^{-k|v-\rho|}e^{k_\ast|v|}\varphi_k(\rho-w)\frac{e^{2\rho}D(\rho)}{B'(\rho)}h(\rho)\partial_\rho \log{\frac{B(\rho)-B(w)+i\epsilon}{B'(w)}}\,d\rho\Big|\\
&\lesssim k^{-1}\big|D(v)e^{4v}\varphi_k(v-w)h^\ast(v)\big|\big\langle\log{|v-w|}\big\rangle+\int_{|\rho-w|<2/\sqrt{k}}\frac{e^{4\rho}}{1+e^{8\rho}}|h^\ast(\rho)|\big\langle\log|\rho-w|\big\rangle\,d\rho\\
&\qquad+k^{-1}\int_{|\rho-w|<2/\sqrt{k}}\frac{e^{4\rho}}{1+e^{8\rho}}e^{k_\ast|\rho|}|\partial_\rho h(\rho)|\big\langle\log|\rho-w|\big\rangle\,d\rho.
\end{split}
\end{equation}
It follows from \eqref{LA4} and \eqref{LA7}, and Cauchy-Schwarz inequality that 
\begin{equation}\label{LA8}
\big\|\mathcal{J}_1h\big\|_{Y_{k,k_\ast}}\lesssim k^{-1/5} \sum_{j\in\Z}e^{-|j|}\|h\|_{Y_{k,k_\ast}^j}.
\end{equation}
Combining \eqref{LA3.6} and \eqref{LA8}, the proof of \eqref{LAP11.0} is then complete.

The proof of \eqref{LAP11.0'} is similar and we only need the point-wise inequalities for $v,\rho\in\R$ and $w\in[-20,\infty)$,
\begin{equation}\label{LA7.01}
e^{k_\ast|v-w|}e^{-k|v-\rho|}\leq e^{k_\ast|\rho-w|}.
\end{equation}
For \eqref{LAP11.0''} we only need the inequality for $w\ge0, v\in[0,w]$ and $\rho\in\R$,
\begin{equation}\label{LA7.02}
e^{k_\ast|v-w|}e^{-k|v-\rho|}\lesssim e^{-|v|}e^{|\rho|} e^{k_\ast|\rho-w|}.
\end{equation}
The proof is now complete.
\end{proof}

The following limiting absorption for $w\ge-20$ plays an essential role in the study of the spectral density functions.
\begin{lemma}\label{LAP10}
Assume that $w\ge-20$, $k\in\Z\backslash\{0\}$ and $k_\ast\in\Z$ with $1\leq|k_\ast|\leq|k|$. There exist $\epsilon_\ast>0$ sufficiently small, and $\kappa\in(0,\infty)$ which can be chosen independent of $w, k, k_\ast$, such that the following statement holds. Suppose that $\epsilon\in\R$ with $0<|\epsilon|<\epsilon_\ast e^{-2|w|}$, and $h\in Y_{k,k_\ast}$. If $|k|=1$, we assume in addition that $h$ satisfies
\begin{equation}\label{LAP11.1}
\lim_{j\to\infty}\Big[\big\|e^{v}h\big\|_{L^2(j,j+1)}+\big\|e^{v}\partial_vh\big\|_{L^2(j,j+1)}\Big]=0.
\end{equation}
Then we have
\begin{equation}\label{LAP11.5}
\left\|h+T_{k,\epsilon}^wh\right\|_{Y_{k,k_\ast}}\ge \kappa\|h\|_{Y_{k,k_\ast}},
\end{equation}
and similarly
\begin{equation}\label{LAP11.5'}
\begin{split}
&\sup_{j\in\Z}\Big[\big\|e^{|k_\ast||v-w|}\big(h+T_{k,\epsilon}^wh\big)(v)\big\|_{L^2(j,j+2)}+|k|^{-1}\big\|e^{|k_\ast||v-w|}\partial_v\big(h+T_{k,\epsilon}^wh\big)(v)\big\|_{L^2(j,j+2)}\Big]\\
&\ge \kappa \,\sup_{j\in\Z}\Big[\big\|e^{|k_\ast||v-w|}h\big\|_{L^2(j,j+2)}+|k|^{-1}\big\|e^{|k_\ast||v-w|}\partial_vh\big\|_{L^2(j,j+2)}\Big].
\end{split}
\end{equation}
\end{lemma}

\begin{proof}
We first give the proof of \eqref{LAP11.5}. We can assume that $k\ge1$, $1\leq k_\ast\leq k$ and $\|h\|_{Y_{k,k_\ast}}=1$. 
 By \eqref{LAP11.0} we can assume that $1\leq k\leq L$ for some $L\ge1$, as the other case follows directly from \eqref{LAP11.0}. Suppose the bound does not hold, then there exist $k_0\in\Z\cap[1,L]$, $k_\ast\in\Z\cap[1,k_0]$, $h_j\in Y_{k_0, k_\ast}$, $\|h_j\|_{Y_{k_0, k_\ast}}=1$, $w_j\ge-20$, and $0\neq\epsilon_j\to0$, for $j\ge1$, such that 
\begin{equation}\label{LA9}
h_j(v)+\frac{1}{2k_0}\int_\R e^{-k_0|v-\rho|}\frac{e^{2\rho}D(\rho)}{B(\rho)-B(w_j)+i\epsilon_j}h_j(\rho)\,d\rho\to0,\end{equation}
in $Y_{k_0, k_\ast}$, as $j\to\infty$. We first note that for sufficiently large $K>1$,
\begin{equation}\label{LA10}
\limsup_{j\to\infty}\|h_j\|_{H^1(-K,K)}>0,
\end{equation}
which follows from \eqref{LAP11.0} and \eqref{LA9}. Using 
\begin{equation}\label{LA11}
\begin{split}
\partial_v\big[T_{k_0,\epsilon}^wh_j(v)\big]=&\frac{1}{2k_0}\int_\R\partial_ve^{-k_0|v-\rho|}\big(1-\varphi_{k_0}(\rho-w)\big)\frac{e^{2\rho}D(\rho)}{B(\rho)-B(w_j)+i\epsilon_j}h_j(\rho)\,d\rho\\
&+ \frac{1}{2k_0}\int_\R \partial_ve^{-k_0|v-\rho|}\varphi_{k_0}(\rho-w_j)\frac{e^{2\rho}D(\rho)}{B(\rho)-B(w_j)+i\epsilon_j}h_j(\rho)\,d\rho\\
:=&\mathcal{D}_1h_{j}(v)+\mathcal{D}_2h_{j}(v).
\end{split}
\end{equation}
For $\mathcal{D}_1h_{j}(v)$, we have
\begin{equation}\label{LA12}
\begin{split}
\partial_v \big[\mathcal{D}_1h_{j}(v)\big]=&-\big(1-\varphi_{k_0}(v-w_j)\big)\frac{e^{2v}D(v)h_j(v)}{B(v)-B(w_j)+i\epsilon}\\
&+\frac{k_0}{2}\int_\R e^{-k_0|v-\rho|}(1-\varphi_{k_0}(\rho-w_j))\frac{e^{2\rho}D(\rho)}{B(\rho)-B(w_j)+i\epsilon_j}h_j(\rho)\,d\rho,
\end{split}
\end{equation}
which can be bounded in $L^2(-K,K)$ uniformly in $j\ge1$, for any $K>1$. For the term $\mathcal{D}_2h_{j}(v)$, we have
\begin{equation}\label{LA13}
\begin{split}
\partial_v \big[&\mathcal{D}_2h_{j}(v)\big]= \frac{1}{2k_0}\partial_v\int_\R \partial_ve^{-k_0|v-\rho|}\varphi_{k_0}(\rho-w_j)\frac{e^{2\rho}D(\rho)}{B'(\rho)}h_j(\rho)\partial_\rho \log{\frac{B(\rho)-B(w_j)+i\epsilon_j}{B'(w_j)}}\,d\rho\\
&=-\varphi_{k_0}(v-w_j)\frac{e^{2v}D(v)}{B'(v)}h_j(v) \partial_v\Big[\log{\frac{B(v)-B(w_j)+i\epsilon_j}{B'(w_j)}}\Big]\\
&\quad-\frac{k_0}{2}\int_\R \partial_\rho\Big[e^{-k_0|v-\rho|}\varphi_{k_0}(\rho-w_j)\frac{e^{2\rho}D(\rho)}{B'(\rho)}h_j(\rho)\Big]  \log{\frac{B(\rho)-B(w_j)+i\epsilon_j}{B'(w_j)}}\,d\rho,
\end{split}
\end{equation}
and as a consequence 
$$\partial_v \Big[\mathcal{D}_2h_{j}(v)+\varphi_{k_0}(v-w_j)\frac{e^{2v}D(v)}{B'(v)}h_j(v) \log{\frac{B(v)-B(w_j)+i\epsilon_j}{B'(w_j)}}\Big]$$ 
can be bounded  $L^2(-K,K)$ uniformly in $j\ge1$, for any $K>1$. Therefore, we can pass to a subsequence and assume that $h_j \to h$ in $H^1_{{\rm loc}}(\R)$ for some $h\in Y_{k_0, k_\ast}(\R)$ with $h\not\equiv0$.

To reach a contradiction, we consider two cases.

{\bf Case 1: $w_j\to w_0\in[-20,\infty)$.} Letting $j\to\infty$, we obtain from \eqref{LA9} that
\begin{equation}\label{LA14}
h(v)+\frac{1}{2k_0}\lim_{j\to\infty}\int_\R e^{-k_0|v-\rho|}\frac{e^{2\rho}D(\rho)}{B(\rho)-B(w_0)+i\epsilon_j}h(\rho)\,d\rho=0,
\end{equation}
and hence
\begin{equation}\label{LA15}
(k_0^2-\partial_v^2)h+\lim_{j\to\infty}\frac{e^{2v}D(v)h(v)}{B(v)-B(w_0)+i\epsilon_j}=0.
\end{equation}
Multiplying $\overline{h}$ and integrating over $\R$, and taking the imaginary part, we get that $h(w_0)=0$. Therefore
\begin{equation}\label{LA16}
(k_0^2-\partial_v^2)h+\frac{e^{2v}D(v)h(v)}{B(v)-B(w_0)}=0.
\end{equation}
Set
\begin{equation}\label{LA17}
\phi(r)=h(v), \quad r=e^v \quad r_0=e^w, \qquad{\rm for}\,\,v, w\in\R. 
\end{equation}
\eqref{LA17} can be reformulated as 
\begin{equation}\label{LA18}
\big(-\partial_r^2-\frac{\partial_r}{r}+k_0^2\,r^{-2}\big)\phi+\frac{d(r)}{b(r)-b(r_0)}\phi=0, \quad{\rm for}\,\,r\in(0,\infty).
\end{equation}
By $\phi(r_0)=0$, standard elliptic regularity theory shows that $\phi\in C^\infty(0,\infty)$. Denote for $r\in(0,\infty)$
\begin{equation}\label{LA19}
g(r):=\big(-\partial_r^2-\frac{\partial_r}{r}+k_0^2\,r^{-2}\big)\phi(r)=-\frac{d(r)}{b(r)-b(r_0)}\phi(r).
\end{equation}
it follows from $h\in Y_{k_0,k_\ast}$ and $h(w_0)=0$ that 
\begin{equation}\label{LA20}
\int_0^\infty \frac{r^2}{|\Omega'(r)|}g^2(r)\,dr=\int_0^\infty \frac{r}{|d(r)|}g^2(r)\,dr=\int_0^\infty \frac{r |d(r)|}{|b(r)-b(r_0)|^2}\phi^2(r)\,dr<\infty.
\end{equation}
Therefore $g$ is an eigenfunction of $L_{k_0}$ with eigenvalue $\lambda=b(r_0)$, a contradiction.

{\bf Case 2: $w_j\to\infty$.} Letting $j\to\infty$, we obtain that for $v\in\R$,
\begin{equation}\label{LA21}
h(v)+\frac{1}{2k_0}\int_\R e^{-k_0|v-\rho|}\frac{D(\rho)}{B(\rho)}e^{2\rho}h(\rho)\,d\rho=0,
\end{equation}
or equivalently for $v\in\R$,
\begin{equation}\label{LA22}
(k_0^2-\partial_v^2)h(v)+\frac{e^{2v}D(v)h(v)}{B(v)}=0.
\end{equation}
Setting for $r=e^v, v\in\R$,
\begin{equation}\label{LA23}
\phi(r):=h(v),\quad g(r):=\big(-\partial_r^2-\frac{\partial_r}{r}+k_0^2r^{-2}\big)\phi(r)=-\frac{d(r)}{b(r)}\phi(r),
\end{equation}
then for $r\in\R^+$
\begin{equation}\label{LA24}
b(r)g(r)+d(r)\int_0^\infty G_{k_0}(r,\rho)g(\rho)\,d\rho=0.
\end{equation}
It follows from $h\in Y_{k_0, k_\ast}$ that $g\in X_{k_0}$. Therefore $g$ is an eigenfunction of $L_{k_0}$ corresponding to the eigenvalue $\lambda=0$. By Proposition \ref{SPEC} $k_0=1$ and $g=-\Omega'(r)$, which implies that $\phi(r)=U(r)\approx r^{-1}$ as $r\to\infty$. So $h(v)\approx e^{-v}$ as $v\to\infty$, a contradiction with \eqref{LAP11.1}.  The proof of \eqref{LAP11.5} is then complete.

We now turn to the proof of \eqref{LAP11.5'}. The general idea is the same as in the proof of \eqref{LAP11.5}, using a contradiction argument. However, we need more refined estimates to extract a nontrivial limit since the weight used in this case requires us to renormalize the solution by multiplying a large constant. We can assume that $k\ge1$, $1\leq k_\ast\leq k$ and 
$$\sup_{l\in\Z}\Big[\big\|e^{|k_\ast||v-w|}h\big\|_{L^2(l,l+2)}+|k|^{-1}\big\|e^{|k_\ast||v-w|}\partial_vh\big\|_{L^2(l,l+2)}\Big]=1.$$
 By \eqref{LAP11.0'} we can assume that $1\leq k\leq L$ for some $L\ge1$, as the other case follows directly from \eqref{LAP11.0'}. Suppose  \eqref{LAP11.5'} does not hold, then there exist $k_0\in\Z\cap[1,L]$, $k_\ast\in\Z\cap[1,k_0]$, and for $j\ge1$, $0\neq\epsilon_j\to0$, $w_j\ge-20$, and $h_j\in L^2(\R)$ satisfying
 $$\sup_{l\in\Z}\Big[\big\|e^{|k_\ast||v-w_j|}h_j\big\|_{L^2(l,l+2)}+|k_0|^{-1}\big\|e^{|k_\ast||v-w_j|}\partial_vh_j\big\|_{L^2(l,l+2)}\Big]=1,$$ 
 such that 
\begin{equation}\label{LAPL1}
\begin{split}
&\sup_{l\in\Z}\Big[\big\|e^{|k_\ast||v-w_j|}\big(h_j+T_{k_0,\epsilon_j}^{w_j}h_j\big)(v)\big\|_{L^2(l,l+2)}+|k_0|^{-1}\big\|e^{|k_\ast||v-w_j|}\partial_v\big(h_j+T_{k_0,\epsilon_j}^{w_j}h_j\big)(v)\big\|_{L^2(l,l+2)}\Big]\\
&\to 0+,\qquad {\rm as}\,\,j\to\infty.
\end{split}
\end{equation}
We first note, as a consequence of \eqref{LAP11.0'}, that for some sufficiently large $K>1$ and all large $j\gg1$,
\begin{equation}\label{LAPL2}
\limsup_{j\to\infty}\big\|e^{k_\ast|v-w_j|}h_j(v)\big\|_{H^1(-K,K)}\gtrsim1.
\end{equation}
It suffices to consider the case $w_j\to\infty$, as the other case follows from the same argument as in the proof of \eqref{LAP11.5}, case 1.
We write for $j\ge1$
\begin{equation}\label{LAPL3}
h_j=-T_{k_0,\epsilon_j}^{w_j}h_j+r_j,
\end{equation}
where the functions $r_j(v), v\in\R$ satisfy as $j\to\infty$
\begin{equation}\label{LAPL3.3}
\sup_{l\in\Z}\Big[\big\|e^{|k_\ast||v-w_j|}r_j(v)\big\|_{L^2(l,l+2)}+|k_0|^{-1}\big\|e^{|k_\ast||v-w_j|}\partial_vr_j(v)\big\|_{L^2(l,l+2)}\Big]\to0+.
\end{equation}
By \eqref{LAP11.0''}, we have for $v_\ast\in[0,w_j]$ that
\begin{equation}\label{LAPL4}
\Big[\big\|e^{|k_\ast||v-w_j|}T_{k_0,\epsilon_j}^{w_j}h_j(v)\big\|_{L^2(v_\ast,v_\ast+2)}+|k_0|^{-1}\big\|e^{|k_\ast||v-w_j|}\partial_vT_{k_0,\epsilon_j}^{w_j}h_j(v)\big\|_{L^2(v_\ast,v_\ast+2)}\Big]\lesssim e^{-|v_\ast|}.
\end{equation}
Therefore, defining for $v\in\R$,
\begin{equation}\label{LAPL5}
g_j(v):=e^{|k_\ast||w_j|}h_j(v),
\end{equation}
we can conclude that $g_j$ are bounded in $H^1_{\rm loc}(\R)$ uniformly in $j$. Using calculations similar to \eqref{LA11}-\eqref{LA13} to obtain compactness in $H^1_{\rm loc}$, by passing to a subsequence, we can assume $g_j\to g$ in $H^1_{\rm loc}$ as $j\to\infty$, for some 
$g\in H^1_{\rm loc}$ with $g\not\equiv0$ and for $v_\ast\in\R$
\begin{equation}\label{LAP6}
\Big[\big\|g\big\|_{L^2(v_\ast,v_\ast+2)}+|k_0|^{-1}\big\|\partial_vg\big\|_{L^2(v_\ast,v_\ast+2)}\Big]\lesssim e^{-|k_\ast||v_\ast|}\mathbf{1}_{v_\ast<0}+e^{(|k_\ast|-1)|v_\ast|}\mathbf{1}_{v_\ast>0}.
\end{equation}
It follows from \eqref{LAPL3}-\eqref{LAPL3.3} that for $v\in\R$,
\begin{equation}\label{LAP7}
(k_0^2-\partial_v^2)g(v)+\frac{e^{2v}D(v)}{B(v)}g(v)=0.
\end{equation}
By \eqref{LAP6} and the comparison principle, noting that $\frac{e^{2v}D(v)}{B(v)}\to0$ as $v\to\infty$, we conclude that if we choose $R>1$ sufficiently large, then for all $\sigma>0$ and $v\ge R$,
\begin{equation}\label{LAP7.01}
|g(v)|\leq \sigma\, e^{(|k_0|-1/2)v} +|g(R)|e^{-(|k_0|-1/2)(v-R)},
\end{equation}
which together with the bounds \eqref{LAP6} (for $v_\ast<0$) and the equation \eqref{LAP7} imply that for $v\in\R$,
\begin{equation}\label{LAP8}
\big|g(v)\big|+\big|\partial_vg(v)\big|\lesssim_{k_0} e^{-|k_0||v|},
\end{equation}
and we can obtain a contradiction as in the proof \eqref{LAP11.5}, Case 2. The proof of \eqref{LAP11.5'} is now complete.

\end{proof}

\begin{remark}\label{LA25}
We briefly explain the motivation behind the assumption \eqref{LAP11.1} for $|k|=1$.
Setting for $r,r_0\in(0,\infty)$,
\begin{equation}\label{LA26}
g_{k,\epsilon}^{\iota}(r,r_0):=\big(-\partial_r^2-\frac{\partial_r}{r}+k^2r^{-2}\big)\psi_{k,\epsilon}^{\iota}(r,r_0),
\end{equation}
then from \eqref{F7} we have for $r, r_0\in(0,\infty)$,
\begin{equation}\label{LA27}
\big[L_kg_{k,\epsilon}^{\iota}(\cdot,r_0)\big](r)+(i\iota\epsilon-b(r_0))g_{k,\epsilon}^{\iota}(r,r_0)=\omega_0^k(r).
\end{equation}
By the orthogonality of $\omega_0^k$ and $\Omega'(r)$ with respect to the $L^2\big(\R^+, \frac{r^2}{|\Omega'(r)|}dr\big)$ metric, see \eqref{B3},  and the fact that $L_k$ is self adjoint in this metric, we see that $g_{k,\epsilon}^{\iota}(r,r_0)$ is orthogonal to $\Omega'(r)$ in $X_k$, that is
\begin{equation}\label{LA28}
\int_0^\infty g_{k,\epsilon}^{\iota}(r,r_0)r^2\,dr=0,
\end{equation}
which implies that
\begin{equation}\label{LA29}
\psi_{k,\epsilon}^{\iota}(r,r_0)=o\big(1/r\big),\quad{\rm as}\,\,r\to\infty.
\end{equation}
Recall the relation $\Pi_{k,\epsilon}^\iota(v,w)=\psi_{k,\epsilon}^{\iota}(r,r_0)$, we see that 
\begin{equation}\label{30}
\lim_{v\to\infty}\big[e^v\Pi_{k,\epsilon}^\iota(v,w)\big]=0.\end{equation}
Hence we can apply Lemma \ref{LAP10} to obtain bounds on $\Pi_{k,\epsilon}^\iota(v,w)$ for $|k|=1$. 
\end{remark}

\subsection{The limiting absorption principle for $w\leq-10$.} This case is more involved since the lower order term is ``long range" and a non-local contribution has to be extracted in order to consider it perturbative. Define for $w\leq-10$, $k\in\Z\backslash\{0\}$, $k_\ast\in\Z$ with $1\leq|k_\ast|\leq|k|$, and $\epsilon\in[-e^{2w},e^{2w}]\backslash\{0\}$ the space
\begin{equation}\label{LAP22}
Y_{k,k_\ast,w}:=\Big\{h\in L^2(\R):\,\sup_{j\in\Z}\Big[\big\|\varrho_{k_\ast,w}(v) h(v)\big\|_{L^2(j,j+2)}+|k|^{-1}\big\|\varrho_{k_\ast,w}(v)\partial_vh(v)\big\|_{L^2(j,j+2)}\Big]<\infty\Big\},
\end{equation}
and the operator $S_{k,\epsilon}^w:Y_{k,k_\ast,w}\to Y_{k,k_\ast,w}$, for any $h\in Y_{k,k_\ast,w}$,
\begin{equation}\label{LAP23}
S_{k,\epsilon}^wh(v):=\int_\R \mathcal{G}_k^w(v,\rho)\Big[\frac{e^{2\rho} D(\rho)}{B(\rho)-B(w)+i\epsilon}-V_w(\rho)\Big]h(\rho)\,d\rho.
\end{equation} 
In the above we recall that
\begin{equation}\label{LAP24}
\varrho_{k_\ast,w}(v):=\zeta_{k_\ast,w}(v,0)=\frac{1}{\varpi_{k_\ast,w}(v,0)}.
\end{equation}
For $w\leq-10$, $k\in\Z\backslash\{0\}$ and $k_\ast\in\Z$ with $1\leq|k_\ast|\leq|k|$, we also define for $j\in\Z$ and $h\in Y_{k,k_\ast,w}$ the norm 
\begin{equation}\label{LAP24.11}
\|h\|_{Y_{k,k_\ast,w}^j}:=\big\|\varrho_{k_\ast,w}(v) h(v)\big\|_{L^2(j,j+2)}+|k|^{-1}\big\|\varrho_{k_\ast,w}(v)\partial_vh(v)\big\|_{L^2(j,j+2)}.
\end{equation}

For applications below we first establish the following bounds.
\begin{lemma}\label{LAO1}
 Suppose $w\leq-10$, $k\in\Z\backslash\{0\}$, $k_\ast\in\Z$ with $1\leq|k_\ast|\leq|k|$, and $\epsilon\in[-e^{2w},e^{2w}]\backslash\{0\}$. For $\rho, v\in\R$ with $|\rho-w|>1/\sqrt{|k|}$, we have the point-wise inequalities
\begin{equation}\label{LAP107}
\varrho_{k_\ast,w}(v) \mathcal{G}_k^w(v,\rho)\Big|\frac{e^{2\rho}D(\rho)}{B(\rho)-B(w)+i\epsilon}-V_w(\rho)\Big|\lesssim |k|^{-1/2}e^{-2|\rho-w|}\varrho_{k_\ast,w}(\rho),
\end{equation}\begin{equation}\label{LAP107.1}
\varrho_{k_\ast,w}(v) \big|\partial_v\mathcal{G}_k^w(v,\rho)\big|\Big|\frac{e^{2\rho}D(\rho)}{B(\rho)-B(w)+i\epsilon}-V_w(\rho)\Big|\lesssim |k|^{1/2}e^{-2|\rho-w|}\varrho_{k_\ast,w}(\rho),
\end{equation}
and for $\rho, v\in\R$ we have
\begin{equation}\label{LAP107.5}
\varrho_{k_\ast,w}(v)\mathcal{G}_k^w(v,\rho)\lesssim \varrho_{k_\ast,w}(\rho)/|k|, \qquad \varrho_{k_\ast,w}(v)\big|\partial_v\mathcal{G}_k^w(v,\rho)\big|\lesssim \varrho_{k_\ast,w}(\rho).
\end{equation}
We also record the useful estimate for $\rho>w+5$,
\begin{equation}\label{LAPL104}
\Big|\frac{e^{2\rho}D(\rho)}{B(\rho)-B(w)+i\epsilon}-V_w(\rho)\Big|\lesssim \frac{e^{-2\rho}\epsilon}{(1+e^\rho)^4}.
\end{equation}
\end{lemma}

\begin{proof}
Recall the bounds \eqref{LAP100}-\eqref{LAP101}. Therefore, we have for $\rho<w-1/\sqrt{k}$,  
\begin{equation}\label{LAP102}
\Big|\frac{e^{2\rho}D(\rho)}{B(\rho)-B(w)+i\epsilon}-V_w(\rho)\Big|\lesssim \sqrt{k}\,e^{-2|w-\rho|};
\end{equation}
for $\rho\in[w+1/\sqrt{k},w+5]$,
\begin{equation}\label{LAP103}
\Big|\frac{e^{2\rho}D(\rho)}{B(\rho)-B(w)+i\epsilon}-V_w(\rho)\Big|\lesssim\sqrt{k};
\end{equation}
for $\rho>w+5$,
\begin{equation}\label{LAP104}
\Big|\frac{e^{2\rho}D(\rho)}{B(\rho)-B(w)+i\epsilon}-V_w(\rho)\Big|\lesssim \frac{e^{-2\rho}\epsilon}{(1+e^\rho)^4}\lesssim \frac{e^{-2|\rho-w|}}{(1+e^\rho)^4}.
\end{equation}
 \eqref{LAP107} then follows from \eqref{LAP102}-\eqref{LAP104} and \eqref{LAP107.5}, which follows from simple calculations, the bounds \eqref{Enh3}-\eqref{Enh3.5} on the Green's function $\mathcal{G}_k^w(v,\rho)$ and the definition \eqref{LAP24} of $\varrho_{k_\ast,w}(v)$ for $v,\rho\in\R$.

\end{proof}

Now we can prove bounds on $S_{k,\epsilon}^w$.
\begin{lemma}\label{LAP25}
Assume that $w\leq -10$, $k\in\Z\backslash\{0\}$, $k_\ast\in\Z$ with $1\leq|k_\ast|\leq|k|$, and $\epsilon\in[-e^{2w}, e^{2w}]\backslash\{0\}$. Then there exist $C\in(0,\infty)$ which can be chosen independent of $w, k, k_\ast$ and $\epsilon$, such that the following statement holds. For $h\in Y_{k,w}$, we have the bounds
\begin{equation}\label{LAP26.0}
\big\|S_{k,\epsilon}^wh\big\|_{Y_{k,k_\ast,w}}\leq C|k|^{-1/5}\sum_{j\in\Z}e^{-|j-w|}\|h\|_{Y^j_{k,k_\ast,w}},
\end{equation}
and similarly
\begin{equation}\label{LAP26.0'}
\begin{split}
&\sup_{j\in\Z}\Big[\big\|\zeta_{k_\ast,w}(v,w) S_{k,\epsilon}^wh(v)\big\|_{L^2(j,j+2)}+|k|^{-1}\big\|\zeta_{k_\ast,w}(v,w)\partial_vS_{k,\epsilon}^wh(v)\big\|_{L^2(j,j+2)}\Big]\\
&\leq C|k|^{-1/5}\sum_{j\in\Z}e^{-|j-w|}\Big[\big\|\zeta_{k_\ast,w}(v,w) h(v)\big\|_{L^2(j,j+2)}+|k|^{-1}\big\|\zeta_{k_\ast,w}(v,w)\partial_vh(v)\big\|_{L^2(j,j+2)}\Big].
\end{split}
\end{equation}
\end{lemma}

\begin{proof}
We focus on the proof of \eqref{LAP26.0}, as the proof of \eqref{LAP26.0'} is similar and we will indicate the required changes at the end of the proof. We assume without loss of generality that $k\ge1$ and $1\leq k_\ast\leq k$. Choose smooth cutoff function $\varphi\in C_0^\infty(-2,2)$ with $\varphi\equiv1$ on $[-1,1]$, and set $\varphi_k(v)=\varphi(\sqrt{k}\,v)$ for $v\in\R$. Let 
\begin{equation}\label{LAP106}
\begin{split}
\mathcal{J}_1h(v)&=\int_\R \mathcal{G}_k^w(v,\rho)\Big[\frac{e^{2\rho}D(\rho)}{B(\rho)-B(w)+i\epsilon}-V_w(\rho)\Big](1-\varphi_k(\rho-w))h(\rho)\,d\rho,\\
\mathcal{J}_2h(v)&=\int_\R \mathcal{G}_k^w(v,\rho)\Big[\frac{e^{2\rho}D(\rho)}{B(\rho)-B(w)+i\epsilon}-V_w(\rho)\Big]\varphi_k(\rho-w)h(\rho)\,d\rho.
\end{split}
\end{equation}
In the above, as usual, we suppressed the dependence of $\mathcal{J}_1, \mathcal{J}_2$ on $k,w,\epsilon$, for the simplicity of notations.

Since $S_{k,\epsilon}^wh=\mathcal{J}_1h+\mathcal{J}_2h$, it suffices to bound $\mathcal{J}_1h$ and $\mathcal{J}_2h$.
Using \eqref{LAP107}-\eqref{LAP107.1}, we can bound
\begin{equation}\label{LAP106.1}
\varrho_{k_\ast,w}(v)\big|\mathcal{J}_1h(v)\big|+k^{-1}\varrho_{k_\ast,w}(v)\big|\partial_v\mathcal{J}_1h(v)\big|\lesssim \frac{1}{\sqrt{k}} \int_\R e^{-|\rho-w|}\varrho_{k_\ast,w}(\rho)|h(\rho)|\,d\rho,
\end{equation}
which implies that 
\begin{equation}\label{LAP106.2}
\big\|\mathcal{J}_1h\big\|_{Y_{k,k_\ast,w}}\lesssim Ck^{-1/2}\sum_{j\in\Z}e^{-|j-w|}\|h\|_{Y^j_{k,k_\ast,w}}.
\end{equation}
It suffices to bound $\mathcal{J}_2h(v)$. We write
\begin{equation}\label{LAP106.3}
\begin{split}
\mathcal{J}_2h(v)=\mathcal{J}_{21}h(v)+\mathcal{J}_{22}h(v):=&\int_\R \mathcal{G}_k^w(v,\rho)\big[-V_w(\rho)\big]\varphi_k(\rho-w)h(\rho)\,d\rho\\
&+\int_\R \mathcal{G}_k^w(v,\rho)\Big[\frac{e^{2\rho}D(\rho)}{B(\rho)-B(w)+i\epsilon}\Big]\varphi_k(\rho-w)h(\rho)\,d\rho.
\end{split}
\end{equation}
Using the bound \eqref{LAP107.5}, we can bound 
\begin{equation}\label{LAP117}
\begin{split}
\varrho_{k_\ast,w}(v)\big|\mathcal{J}_{21}h(v)\big|+k^{-1}\varrho_{k_\ast,w}(v)\big|\mathcal{J}_{21}h(v)\big|\lesssim k^{-1}\int_{|\rho-w|\leq2/\sqrt{k}}\varrho_{k_\ast,w}(\rho)|h(\rho)|\,d\rho,
\end{split}
\end{equation}
which implies that 
\begin{equation}\label{LAP106.4}
\big\|\mathcal{J}_{21}h\big\|_{Y_{k,k_\ast,w}}\lesssim C|k|^{-1}\sum_{j\in\Z}e^{-|j-w|}\|h\|_{Y^j_{k,k_\ast,w}}.
\end{equation}
It remains to bound the term $\mathcal{J}_{22}h$. Using \eqref{LAP107}-\eqref{LAP107.1} and \eqref{LAP100}, we obtain
\begin{equation}\label{LAP118}
\begin{split}
&\varrho_{k_\ast,w}(v)\Big|\int_\R \mathcal{G}_k^w(v,\rho)\frac{e^{2\rho}D(\rho)}{B(\rho)-B(w)+i\epsilon}\varphi_k(\rho-w)h(\rho)\,d\rho\Big|\\
&\lesssim\varrho_{k_\ast,w}(v)\Big|\int_\R \partial_\rho\Big[\mathcal{G}_k^w(v,\rho)\frac{e^{2\rho}D(\rho)}{\partial_\rho B(\rho)}\varphi_k(\rho-w)\Big]h(\rho)\log\frac{B(\rho)-B(w)+i\epsilon}{\partial_wB(w)}\,d\rho\Big|\\
&\quad+\varrho_{k_\ast,w}(v)\Big|\int_\R \mathcal{G}_k^w(v,\rho)\frac{e^{2\rho}D(\rho)}{\partial_\rho B(\rho)}\varphi_k(\rho-w)\partial_{\rho}h(\rho)\log\frac{B(\rho)-B(w)+i\epsilon}{\partial_wB(w)}\,d\rho\Big|\\
&\lesssim \int_{|\rho-w|<2/\sqrt{k}}\varrho_{k_\ast,w}(\rho)|h(\rho)|\langle\log{|\rho-w|}\rangle\,d\rho+\frac{1}{k}\int_{|\rho-w|<2/\sqrt{k}}\varrho_{k_\ast,w}(\rho)|\partial_\rho h(\rho)|\langle\log{|\rho-w|}\rangle\,d\rho.
\end{split}
\end{equation}
In addition, in view of \eqref{Enh3.51},
\begin{equation}\label{LAP119}
\begin{split}
&k^{-1}\varrho_{k_\ast,w}(v)\Big|\int_\R \partial_v\mathcal{G}_k^w(v,\rho)\frac{e^{2\rho}D(\rho)}{B(\rho)-B(w)+i\epsilon}\varphi_k(\rho-w)h(\rho)\,d\rho\Big|\\
&\lesssim k^{-1}\varrho_{k_\ast,w}(v)\Big|\int_\R \partial_\rho\Big[\partial_v\mathcal{G}_k^w(v,\rho)\frac{e^{2\rho}D(\rho)}{\partial_\rho B(\rho)}\varphi_k(\rho-w)\Big]h(\rho)\log\frac{B(\rho)-B(w)+i\epsilon}{\partial_wB(w)}\,d\rho\Big|\\
&\quad+k^{-1}\varrho_{k_\ast,w}(v)\Big|\int_\R \partial_v\mathcal{G}_k^w(v,\rho)\frac{e^{2\rho}D(\rho)}{\partial_\rho B(\rho)}\varphi_k(\rho-w)\partial_{\rho} h(\rho)\log\frac{B(\rho)-B(w)+i\epsilon}{\partial_wB(w)}\,d\rho\Big|\\
&\lesssim k^{-1}\varrho_{k_\ast,w}(v)\varphi_k(v-w)|h(v)|\langle\log{|v-w|}\rangle+\int_{|\rho-w|<2/\sqrt{k}}\varrho_{k_\ast,w}(\rho)|h(\rho)|\langle\log{|\rho-w|}\rangle\,d\rho\\
&\quad +k^{-1}\int_{|\rho-w|<2/\sqrt{k}}\varrho_{k_\ast,w}(\rho)|\partial_\rho h(\rho)|\langle\log{|\rho-w|}\rangle\,d\rho.\\
\end{split}
\end{equation}

It follows from \eqref{LAP118} and \eqref{LAP119} that 
\begin{equation}\label{LAP119.1}
\big\|\mathcal{J}_{22}h\big\|_{Y_{k,k_\ast,w}}\lesssim C|k|^{-1/5}\sum_{j\in\Z}e^{-|j-w|}\|h\|_{Y^j_{k,k_\ast,w}}.
\end{equation}
Combining \eqref{LAP106.2}, \eqref{LAP106.4} and \eqref{LAP119.1}, the proof of \eqref{LAP26.0} is then complete.

The proof of \eqref{LAP26.0'} is similar and we only need to use
\begin{equation}\label{LAP119.2}
\zeta_{k,w}(v,w)\mathcal{G}_k^w(v,\rho)\lesssim \zeta_{k,w}(\rho,w)/|k|,\qquad \zeta_{k,w}(v,w)\big|\partial_v\mathcal{G}_k^w(v,\rho)\big|\lesssim \zeta_{k,w}(\rho,w).
\end{equation}
The lemma is now proved.

\end{proof}

We have the following limiting absorption principle for $w\leq-10$.
\begin{lemma}\label{LAP25}
Assume that $w\leq-10$, $k\in\Z\backslash\{0\}$ and $k_\ast\in\Z$ with $1\leq|k_\ast|\leq |k|$. Then there exist $\kappa\in(0,\infty)$ and a sufficiently small $\beta_{\ast}\in(0,\infty)$, which can be chosen independent of $w, k$ and $k_\ast$, such that the following statement holds. For $\epsilon\in[-\beta_\ast e^{2w},\beta_\ast e^{2w}]\backslash\{0\}$ and $h\in Y_{k,k_\ast,w}$, we have
\begin{equation}\label{LAP26.5}
\big\|h+S_{k,\epsilon}^wh\big\|_{Y_{k,k_\ast,w}}\ge \kappa \|h\|_{Y_{k,k_\ast,w}},
\end{equation}
and similarly
\begin{equation}\label{LAP26.5'}
\begin{split}
&\sup_{j\in\Z}\Big[\big\|\zeta_{k_\ast,w}(v,w) (h+S_{k,\epsilon}^wh)(v)\big\|_{L^2(j,j+2)}+|k|^{-1}\big\|\zeta_{k_\ast,w}(v,w)\partial_v(h+S_{k,\epsilon}^wh)(v)\big\|_{L^2(j,j+2)}\Big]\\
&\ge \kappa\sup_{j\in\Z}\Big[\big\|\zeta_{k_\ast,w}(v,w) h(v)\big\|_{L^2(j,j+2)}+|k|^{-1}\big\|\zeta_{k_\ast,w}(v,w)\partial_vh(v)\big\|_{L^2(j,j+2)}\Big].
\end{split}
\end{equation}

\end{lemma}

\begin{proof}
We focus on \eqref{LAP26.5}, as the proof of \eqref{LAP26.5'} is easier. We can assume $k\ge1$ and $1\leq k_\ast\leq k$. Suppose \eqref{LAP26.5} does not hold. Then by Lemma \ref{LAP25}, there exist $k_0\ge1$ and $k_\ast\in\Z$ with $1\leq k_\ast\leq k_0$, $w_j\leq-10$, $\beta_j\neq0, \beta_j\to0$ as $j\to\infty$, $h_j\in Y_{k_0,k_\ast,w_j}$, $\|h_j\|_{Y_{k_0,k_\ast,w_j}}=1$, such that with $\epsilon_j:=\beta_j e^{2w_j}$, as $j\to\infty$ we have
\begin{equation}\label{LAP120}
h_j+S_{k_0,\epsilon_j}^{w_j}h_j\to 0, \qquad {\rm in}\,\, Y_{k_0, k_\ast, w_j}.
\end{equation}
By Lemma \ref{LAP25}, we can find some $\delta>0, K>1$, such that
\begin{equation}\label{LAP121}
\big\|\varrho_{k_\ast,w}(v)h_j(v)\big\|_{L^2(w_j-K,w_j+K)}+\big\|\varrho_{k_\ast,w}(v)\partial_vh_j(v)\big\|_{L^2(w_j-K,w_j+K)}>\delta,
\end{equation}
for sufficiently large $j\ge1$. If $w_j$ remains bounded for a subsequence, then we can obtain a contradiction as in the proof of Lemma \ref{LAP10}, case 1. It suffices to assume that $w_j\to-\infty$. To pass to a nontrivial limit, we need to obtain more refined bounds on $h_j$, $j\ge0$. 

Write for $j\ge1$,
\begin{equation}\label{LAP121.1}
h_j=-S_{k_0,\epsilon_j}^{w_j}h_j+r_j, \qquad{\rm where}\,\,r_j\in Y_{k_0,k_\ast,w_j} \,\,{\rm and}\,\,\lim_{j\to+\infty}\|r_j\|_{Y_{k_0,k_\ast,w_j}}=0.
\end{equation}
Fix $\Phi_{\leq6}\in C_0^\infty(-\infty, 8)$ with $\Phi_{\leq6}\equiv1$ on $(-\infty,6]$. We define
\begin{equation}\label{LAPQ1}
\begin{split}
\mathcal{J}_3h_j(v)&:=\int_\R \mathcal{G}_k^{w_j}(v,\rho) \Big[\frac{e^{2\rho}D(\rho)}{B(\rho)-B(w_j)+i\epsilon_j}-V_{w_j}(\rho)\Big]\big(1-\Phi_{\leq 6}(\rho-w_j)\big) h_j(\rho)\,d\rho,\\
\mathcal{J}_4h_j(v)&:=\int_\R \mathcal{G}_k^{w_j}(v,\rho) \Big[\frac{e^{2\rho}D(\rho)}{B(\rho)-B(w_j)+i\epsilon_j}-V_{w_j}(\rho)\Big]\Phi_{\leq 6}(\rho-w_j) h_j(\rho)\,d\rho,
\end{split}
\end{equation}
and it follows that 
\begin{equation}\label{LAP121.2}
S_{k_0,\epsilon_j}^{w_j}h_j=\mathcal{J}_3h_j(v)+\mathcal{J}_4h_j(v).
\end{equation}
Using the bounds \eqref{LAPL104}, we obtain that
\begin{equation}\label{LAP121.4}
\big\|\mathcal{J}_3h_j(v)\big\|_{Y_{k_0,k_\ast,w_j}}\lesssim_{k_0} |\epsilon_j|e^{-2w_j}\big\|h_j\big\|_{Y_{k_0,k_\ast,w_j}}=|\beta_j|\to0+, \quad{\rm as}\,\,j\to\infty.
\end{equation}
To obtain more accurate bounds on $\mathcal{J}_4h_j$ than those from Lemma \ref{LAP26.0}, that are needed for extracting nontrivial limiting profile from the sequence $h_j$, we fix $\varphi\in C_0^\infty(-20,20)$ with $\varphi\equiv1$ on $[-10,10]$ and note that 
\begin{equation}\label{LAP121.5}
\begin{split}
&\big|\mathcal{J}_4h_j(v)\big|+\big|\partial_v\mathcal{J}_4h_j(v)\big|\lesssim_{k_0}\int_\R \big(\big|\mathcal{G}_{k_0}^{w_j}(v,\rho)\big|+\big|\partial_v\mathcal{G}_{k_0}^{w_j}(v,\rho)\big|\big)\big|V_{w_j}(\rho)\Phi_{\leq6}(\rho-w_j) h_j(\rho)\big|\,d\rho\\
&+\left|\int_\R \partial_{\rho}\Big[\frac{e^{2\rho}D(\rho)}{\partial_\rho B(\rho)}\mathcal{G}_{k_0}^{w_j}(v,\rho)\Phi_{\leq6}(\rho-w_j)\varphi(\rho-w_j) h_j(\rho)\Big]\log{\frac{B(\rho)-B(w_j)+i\epsilon_j}{\partial_wB(w_j)}}\,d\rho\right|\\
&+\left|\int_\R \partial_{\rho}\Big[\frac{e^{2\rho}D(\rho)}{\partial_\rho B(\rho)}\partial_v\mathcal{G}_{k_0}^{w_j}(v,\rho)\Phi_{\leq6}(\rho-w_j) \varphi(\rho-w_j)h_j(\rho)\Big]\log{\frac{B(\rho)-B(w_j)+i\epsilon_j}{\partial_wB(w_j)}}\,d\rho\right|\\
&+\int_{-\infty}^{w_j-1} \big(\big|\mathcal{G}_{k_0}^{w_j}(v,\rho)\big|+\big|\partial_v\mathcal{G}_{k_0}^{w_j}(v,\rho)\big|\big)|1-\varphi(\rho-w_j)|e^{-2|\rho-w_j|} \big|h_j(\rho)\big|\,d\rho.
\end{split}
\end{equation}
We obtain from \eqref{LAP121.5} that for $v_\ast>w_j+9$,
\begin{equation}\label{LAP121.7}
\begin{split}
&\big\|\big|\mathcal{J}_4h_j(v)\big|+\big|\partial_v\mathcal{J}_4h_j(v)\big|\big\|_{L^2(v_\ast,v_\ast+2)}\lesssim \frac{\varpi_{k_\ast,w_j}(v_\ast,w_j)}{\varrho_{k_\ast,w_j}(w_j)}.
\end{split}
\end{equation}
Setting  
\begin{equation}\label{LAP122}
g_j(v)=e^{\mu_{k_\ast}|w_j|}h_j(w_j+v),\qquad v\in\R,
\end{equation}
then combining \eqref{LAP121}-\eqref{LAP121.1}, \eqref{LAP121.4} and \eqref{LAP121.7} we obtain that
\begin{equation}\label{LAP122.1}
\|g_j\|_{H^1(-K,K)}>\delta,\qquad {\rm and}\,\,g_j\to g\,\,{\rm in}\,\, H^1_{{\rm loc}}(-L,L), \,\,{\rm for\,\,all\,\,}L>1.
\end{equation}
We note that the compactness in $H^1_{{\rm loc}}(\R)$ of $g_j$ follows from similar argument as in the case of $w>-20$.
In addition, $g$ satisfies 
\begin{equation}\label{LAP122.2}
\sup_{j\in\Z}\Big[\big\|e^{|v|}g\big\|_{L^2(j,j+2)}+\big\|e^{|v|}\partial_vg\big\|_{L^2(j,j+2)}\Big]<\infty.
\end{equation}
Since in the sense of distributions for $v\in\R$,
\begin{equation}\label{LAP123}
(k_0^2-\partial_v^2+V_{w_j}(w_j+v))g_j(v)+\Big[\frac{e^{2v+2w_j}D(w_j+v)}{B(w_j+v)-B(w_j)+i\epsilon_j}-V_{w_j}(w_j+v)\Big]g_j(v)=0,
\end{equation}
sending $j\to\infty$ and using \eqref{LAP100}-\eqref{LAP101}, we obtain that for $v\in\R$,
\begin{equation}\label{LAP124}
(k^2-\partial_v^2)g(v)+\lim_{j\to\infty}\frac{8e^{2v}}{e^{2v}-1+8i\beta_j}g(v)=0,
\end{equation}
in the sense of distributions. Multiplying $\overline{g}$ to \eqref{LAP124}, integrating over $\R$ and taking the real part, we obtain that 
$g(0)=0$. It follows that for $v\in\R$,
\begin{equation}\label{LAP125}
(k^2-\partial_v^2)g(v)+\frac{8e^{2v}}{e^{2v}-1}g(v)=0,
\end{equation}
in the sense of distributions. Noting that $\frac{8e^{2v}}{e^{2v}-1}>0$ for $v>0$ and that $g(0)=0$, by the maximum principle, we conclude that $g\equiv0$ for $v\ge0$. Then by simple ODE argument and Gronwall type inequalities, we conclude that $g\equiv0$ on $\R$, a contradiction with \eqref{LAP122.1}. The proof of \eqref{LAP26.5} is then complete. The proof of \eqref{LAP26.5'} follows the same line of argument, but is simpler since we no longer need to renormalize the sequence as in \eqref{LAP122}.
\end{proof}

 \section{The explicitly solvable case: $|k|=1$}\label{sec:k=1} We first consider the special case when $|k|=1$, which is, remarkably, explicitly solvable as observed in \cite{Bed2}. Since we focus on the spectral density functions rather than the explicit formula for $f_k(t,v)$ studied in \cite{Bed2}, we provide some details of the calculations. Theorem \ref{MTH2} for the case $|k|=1$ follows from the following explicit formula.
 \begin{proposition}\label{BSD24}
Assume that $k\in\{1,-1\}$. For $v, w\in\R$, we have 
\begin{equation}\label{BSD25}
\begin{split}
\Gamma_k(v,w)&=2  \lim_{\epsilon\to0+}\Im\,\Gamma^+_{k,\epsilon}(v,w)=2\lim_{\epsilon\to0+}\Im\, \Pi_{k,\epsilon}^+(v,w)\\
&=2\pi \frac{B(v)-B(w)}{(\partial_wB(w))^2}e^{v+w}\mathbf{1}_{v<w}\bigg\{f_0^k(w)-e^{-w}\frac{D(w)}{\partial_wB(w)}\int_{-\infty}^wf_0^k(\rho)e^{3\rho}\,d\rho\bigg\}.
\end{split}
\end{equation}
\end{proposition}

The rest of the section is devoted to the proof of proposition \ref{BSD24}.

We note first that for $\epsilon\in(0,1/8)$, the function 
\begin{equation}\label{BSD12}
h_{w,\epsilon}(v):=\big(B(v)-B(w)+i\epsilon\big)e^v
\end{equation}
for $v\in\R$ solves 
\begin{equation}\label{BSD13}
(1-\partial_v^2)h_{w,\epsilon}(v)+\frac{e^{2v}D(v)}{B(v)-B(w)+i\epsilon}h_{w,\epsilon}(v)=0, \qquad{\rm on}\,\,\R.
\end{equation}
By the general theory of ODEs, we can find another solution to \eqref{BSD13}
\begin{equation}\label{BSD14}
g_{w,\epsilon}(v):=h_{w,\epsilon}(v)\int_v^\infty \frac{d\sigma}{h^2_{w,\epsilon}(\sigma)}.
\end{equation}
Define for $\epsilon\in(0,1/4), w\in\R, v, \rho\in\R$, 
\begin{equation}\label{BSD15}
G_{w,\epsilon}(v,\rho):=h_{w,\epsilon}(v)g_{w,\epsilon}(\rho)\mathbf{1}_{(-\infty,\rho)}(v)+g_{w,\epsilon}(v)h_{w,\epsilon}(\rho)\mathbf{1}_{[\rho,\infty)}(v).
\end{equation}
Direct calculations show that $G_{w,\epsilon}(v,\rho)$ is the fundamental solution to the differential operator 
\begin{equation}\label{BSD16}
1-\partial_v^2+\frac{e^{2v}D(v)}{B(v)-B(w)+i\epsilon}, \qquad {\rm on}\,\,\R.
\end{equation}
Therefore for $\epsilon\in(0,1/4), w, v, \rho\in\R$, we have
\begin{equation}\label{BSD17}
\Pi_{k,\epsilon}^+(v,w)=\int_\R G_{w,\epsilon}(v,\rho)\frac{e^{2\rho}f_0^k(\rho)}{B(\rho)-B(w)+i\epsilon}\,d\rho.
\end{equation}


Define $\epsilon\in(0,1/4), w\in\R, v\in\R$,
\begin{equation}\label{BSD19}
W_{w,\epsilon}(v):=\int_v^\infty \frac{d\sigma}{h^2_{w,\epsilon}(\sigma)}.
\end{equation}


Simple computation shows
\begin{equation}\label{BSD19.1}
\sup_{\rho\in[w+1,\infty)}\big|W_{w,\epsilon}(\rho)\big|\lesssim_w1, \quad \sup_{\rho\in[w+1,\infty)}\big|\Im W_{w,\epsilon}(\rho)\big|\lesssim \epsilon.
\end{equation}

For $\rho\in(-\infty,w+1]$, we have
\begin{equation}\label{BSD25.1}
\begin{split}
 W_{w,\epsilon}(\rho)&= \int_{\rho}^{\infty} \frac{e^{-2\sigma}}{(B(\sigma)-B(w)+i\epsilon)^2}d\sigma\\
&= \int_{\rho}^{w+1} -\frac{e^{-2\sigma}}{\partial_\sigma B(\sigma)}\partial_\sigma\Big[\frac{1}{B(\sigma)-B(w)+i\epsilon}\Big]d\sigma+ \int_{w+1}^{\infty} \frac{e^{-2\sigma}}{(B(\sigma)-B(w)+i\epsilon)^2}d\sigma\\
&=\frac{e^{-2\rho}}{\partial_\rho B(\rho) (B(\rho)-B(w)+i\epsilon)}-\frac{e^{-2(w+1)}}{\partial_wB(w+1) (B(w+1)-B(w)+i\epsilon)}\\
&\quad+ \int_{\rho}^{w+1} \partial_\sigma\Big[\frac{e^{-2\sigma}}{\partial_\sigma B(\sigma)}\Big]\frac{1}{B(\sigma)-B(w)+i\epsilon}d\sigma+ \int_{w+1}^{\infty} \frac{e^{-2\sigma}}{(B(\sigma)-B(w)+i\epsilon)^2}d\sigma,
\end{split}
\end{equation}
It follows from \eqref{BSD25.1} that
\begin{equation}\label{BSD25.2}
|\epsilon\, W_{w,\epsilon}(\rho)|\lesssim_w e^{4|\rho|},\quad{\rm and}\quad \lim_{\epsilon\to0+} \epsilon\, W_{w,\epsilon}(\rho)=0, \,\,{\rm for}\,\,\rho\in(-\infty,w+1]\backslash\{w\},
\end{equation}
and in the sense of distributions for $\rho\in(-\infty,w+1)$,
\begin{equation}\label{BSD25.3}
\lim_{\epsilon\to0+}\Im\,W_{w,\epsilon}(\rho)= \pi \frac{e^{-2w}}{(\partial_wB(w))^2} \delta(\rho-w)+\pi \frac{1}{\partial_wB(w)}\partial_w\Big[\frac{e^{-2w}}{\partial_wB(w)}\Big]{\bf 1}_{\rho<w}.
\end{equation}
Applying the limiting absorption principle to the equations \eqref{LAP1.1} and \eqref{LAP1.2}, see Lemmas \ref{LAP10} and \ref{LAP25}, we can obtain for $\Gamma_{1,\epsilon}^\iota(v,w)$ for  $\iota\in\{+,-\}, w\in\R$ and $\epsilon\in\R$ with $0<|\epsilon|<\epsilon_\ast e^{-2|w|}$,  the bounds
\begin{equation}\label{BSD11}
\sup_{j\in\Z}\Big[\|e^{|v|}\Gamma_{k,\epsilon}^\iota(v,w)\|_{L^2(j,j+2)}+\|e^{|v|}\partial_v\Gamma_{k,\epsilon}^\iota(v,w)\|_{L^2(j,j+2)}\Big]\lesssim_w1.
\end{equation}
The bounds \eqref{BSD11} imply that to calculate the limit \eqref{BSD25} it suffices to consider the case $v\neq w$.



We can calculate, using \eqref{BSD17}, that for $v, w\in\R$ with $v\neq w$, 
\begin{equation}\label{BSD26}
\begin{split}
  \lim_{\epsilon\to0+}\Im\,\Gamma^+_{k,\epsilon}(v,w)&=\lim_{\epsilon\to0+}\Im\,\Pi^+_{k,\epsilon}(v,w)\\
&=\lim_{\epsilon\to0+}\Im  \int_{v}^{\infty}(B(v)-B(w)-i\epsilon)e^{v+\rho}W_{w,\epsilon}(\rho)f_0^k(\rho)e^{2\rho}\,d\rho \\
&\qquad+\lim_{\epsilon\to0+}\Im  \int_{-\infty}^{v}(B(v)-B(w)+i\epsilon)e^{v+\rho}W_{w,\epsilon}(v)f_0^k(\rho)e^{2\rho}\,d\rho\\
&=\lim_{\epsilon\to0+}\Im  \int_{v}^{\infty}(B(v)-B(w))e^{v+\rho}W_{w,\epsilon}(\rho)f_0^k(\rho)e^{2\rho}\,d\rho\\
&\qquad+\lim_{\epsilon\to0+}\Im  \int_{-\infty}^{v}(B(v)-B(w))e^{v+\rho}W_{w,\epsilon}(v)f_0^k(\rho)e^{2\rho}\,d\rho.
\end{split}
\end{equation}

Therefore by \eqref{BSD19.1} and \eqref{BSD25.2}-\eqref{BSD25.3}, we obtain that 
\begin{equation}
\begin{split}
&  \lim_{\epsilon\to0+}\Im\,\Gamma^+_{k,\epsilon}(v,w)\\
&=\pi \frac{B(v)-B(w)}{(\partial_wB(w))^2}e^{v+w}\mathbf{1}_{v<w}f_0^k(w)+\pi\mathbf{1}_{v<w}\int_{v}^{w} \frac{e^v (B(v)-B(w))}{\partial_wB(w)}\partial_w\Big[\frac{e^{-2w}}{\partial_wB(w)}\Big]f_0^k(\rho)e^{3\rho}\,d\rho\\
&\quad+\pi\mathbf{1}_{v<w}\int_{-\infty}^{v} \frac{e^v (B(v)-B(w))}{\partial_wB(w)}\partial_w\Big[\frac{e^{-2w}}{\partial_wB(w)}\Big]f_0^k(\rho)e^{3\rho}\,d\rho.
\end{split}
\end{equation}
The desired formula \eqref{BSD25} follows from the identity for $w\in\R$,
\begin{equation}\label{BSD27}
2\partial_wB(w)+\partial_w^2B(w)=e^{2w}D(w).
\end{equation}


\section{Bounds on the spectral density function I: preliminary bounds}\label{sec:sdf1}
In this section we obtain important bounds on the spectral density functions $\Pi_{k,\epsilon}^{\iota}(v,w)$ with $k\in\Z\backslash\{0\}, \iota\in\{+,-\}, \epsilon\in[-1/8, 1/8]\backslash\{0\}$ and $v,w\in\R$, with a focus on capturing the optimal decay property. We first establish bounds in low regularity Sobolev spaces, which are then used as a stepping stone to obtain stronger bounds in Gevrey spaces. The main tools are the limiting absorption principle proved in section \ref{sec:lap}, and commutator type arguments.  

For applications below, we fix smooth cutoff functions $\Psi, \Psi^\ast, \Psi^{\ast\ast}:\R\to[0,1]$ with $\Psi\in C^\infty(-2,2)$ satisfying $\Psi\equiv1$ on $[-1,1]$,  $\Psi^\ast\in C^\infty_0(-4,4)$ satisfying $\Psi^\ast\equiv1$ on $[-3,3]$, and $\Psi^{\ast\ast}\in C^\infty_0(-5,5)$ satisfying $\Psi^{\ast\ast}\equiv1$ on $[-9/2,9/2]$. Arrange in addition for $h\in\{\Psi,\Psi^\ast,\Psi^{\ast\ast}\}$, that
$\sup_{\xi\in\R}\big|e^{\langle\xi\rangle^{5/6}}\widehat{\,h\,}(\xi)\big|\lesssim 1$.

\subsection{The bounds for $\Pi_{k,\epsilon}^{\iota}(v,w)$ for $|k|\ge2$}
Recall that the spectral density function $\Pi_{k,\epsilon}^{\iota}(v,w)$ with $k\in\Z\backslash\{0\}, \iota\in\{+,-\}, \epsilon\in[-1/8, 1/8]\backslash\{0\}$ and $v,w\in\R$, satisfy for $v\in\R, w\in\R$,
\begin{equation}\label{BSD1}
(k^2-\partial_v^2)\Pi_{k,\epsilon}^{\iota}(v,w)+\frac{e^{2v}D(v)}{B(v)-B(w)+i\iota\epsilon}\Pi_{k,\epsilon}^{\iota}(v,w)=\frac{e^{2v}f_0^k(v)}{B(v)-B(w)+i\iota \epsilon}.
\end{equation}
Our goal is to use the limiting absorption principle to obtain bounds on the spectral density function $\Pi_{k,\epsilon}^{\iota}(v,w)$. Recall that $\Phi_0(v)\in C_0^\infty(-\infty, -1)$ with $
\Phi_0\equiv 1$ on $(-\infty, -2]$ and $\sup_{\xi\in\R}\big|e^{\langle \xi\rangle^{8/9}}\widehat{\partial_v\Phi_0}(\xi)\big|\lesssim1$, and 
\begin{equation}\label{BSD2}
\Pi_{k,\epsilon}^\iota(v,w)=(\sigma_k/c_\ast)e^{|k|v}\Phi_{0}(v)+\Gamma_{k,\epsilon}^{\iota}(v,w).
\end{equation}
It follows from \eqref{BSD1} that $\Gamma_{k,\epsilon}^{\iota}(v,w)$ satisfies the equation
\begin{equation}\label{BSD3}
\begin{split}
&(k^2-\partial_v^2)\Gamma_{k,\epsilon}^{\iota}(v,w)+\frac{e^{2v}D(v)}{B(v)-B(w)+i\iota\epsilon}\Gamma_{k,\epsilon}^{\iota}(v,w)\\
&=\frac{e^{2v}F_{0k}(v)}{B(v)-B(w)+i\iota \epsilon}+(\sigma_k/c_\ast)\big(2|k|e^{|k|v}\partial_v\Phi_{0}(v)+e^{|k|v}\partial_v^2\Phi_{0}(v)\big).
\end{split}
\end{equation}
\begin{lemma}\label{BSD4}
There exists $\epsilon_\ast\in(0,1)$ sufficiently small such that the following statement holds. Assume that $k\in\Z\backslash\{0\}$ with $|k|\ge2$, $\iota\in\{+,-\}$, $w_\ast\in\R$, and $\epsilon\in\R$ with $0<|\epsilon|<\epsilon_\ast e^{-2|w_\ast|}$. Recall the definition \eqref{LAP24} for $\varrho_{k,w}(v), v\in\R$. Then the spectral density function $\Gamma_{k,\epsilon}^\iota(v,w)$ satisfies for $w_\ast\leq-15$
\begin{equation}\label{BSD5}
\begin{split}
&\Big\|\sup_{j\in\Z}\Big[\big\|\varrho_{\varkappa_k,w_\ast}(v)\Gamma_{k,\epsilon}^\iota(v,w)\Psi(w-w_\ast)\big\|_{L^2(v\in[j,j+2])}\\
&\qquad+|k|^{-1}\big\|\varrho_{\varkappa_k,w_\ast}(v)\partial_v\Gamma_{k,\epsilon}^\iota(v,w)\Psi(w-w_\ast)\big\|_{L^2(v\in[j,j+2])}\Big]\Big\|_{L^2(w\in\R)}\lesssim (M_k^{\dagger}+|\sigma_k|)/|k|;
\end{split}
\end{equation}
and for $w_\ast\ge-15$,
\begin{equation}\label{BSD6}
\begin{split}
&\Big\|\sup_{j\in\Z}\Big[\big\|e^{\varkappa_k|v|}\Gamma_{k,\epsilon}^\iota(v,w)\Psi(w-w_\ast)\big\|_{L^2(v\in[j,j+2])}\\
&\qquad+|k|^{-1}\big\|e^{\varkappa_k|v|}\partial_v\Gamma_{k,\epsilon}^\iota(v,w)\Psi(w-w_\ast)\big\|_{L^2(v\in[j,j+2])}\Big]\Big\|_{L^2(\R)}\lesssim (M_k^{\dagger}+|\sigma_k|)/|k|.
\end{split}
\end{equation}
\end{lemma}

\begin{proof}
We can assume $k\ge2$ and $M_k^{\dagger}+|\sigma_k|=1$ without loss of generality. For $w_\ast\leq-15$, we can reformulate the equation \ref{BSD3} as 
\begin{equation}\label{BSD7}
\begin{split}
&\Gamma_{k,\epsilon}^{\iota}(v,w)+\int_\R \mathcal{G}_k^w(v,\rho)\Big[\frac{e^{2\rho}D(\rho)}{B(\rho)-B(w)+i\iota\epsilon}-V_w(\rho)\Big]\Gamma_{k,\epsilon}^{\iota}(\rho,w)\,d\rho\\
&=\mathcal{R}_1(v,w):=\int_\R \mathcal{G}_k^w(v,\rho)\bigg[\frac{e^{2\rho}F_{0k}(\rho)}{B(\rho)-B(w)+i\iota \epsilon}+\frac{\sigma_k}{c_\ast}\big(ke^{k\rho}\partial_\rho\Phi_{0}(\rho)+e^{k\rho}\partial_\rho^2\Phi_{0}(\rho)\big)\bigg]\,d\rho.
\end{split}
\end{equation}
We have the bounds  (recalling the definitions \eqref{LAP22}-\eqref{LAP23} for $Y_{k,k_\ast,w}$)
\begin{equation}\label{BSD8}
\big\|\|\mathcal{R}_1(v,w)\Psi(w-w_\ast)\|_{Y_{k,\varkappa_k,w_\ast}}\big\|_{L^2(w\in\R)}\lesssim1/|k|.
\end{equation}
The proof of \eqref{BSD8} is standard, using the boundedness of Hilbert transforms, and is also subsumed in the proof of the stronger bounds  \eqref{GVS21.3}-\eqref{GVS21.4} below. We omit the repeated details and refer to the proof of \eqref{GVS21.3}-\eqref{GVS21.4} for details. The desired bounds \eqref{BSD5} then follow from the limiting absorption principle, see Lemma \ref{LAP25}, applied for each $w$ with $|w-w_\ast|\leq 2$ and then square integrate in $w$. 

For $w\ge-15$, we can reformulate the equation  \ref{BSD3} as 
\begin{equation}\label{BSD9}
\begin{split}
&\Gamma_{k,\epsilon}^{\iota}(v,w)+\frac{1}{2|k|}\int_\R e^{-|k(v-\rho)|}\frac{e^{2\rho}D(\rho)}{B(\rho)-B(w)+i\iota\epsilon}\Gamma_{k,\epsilon}^{\iota}(\rho,w)\,d\rho\\
&=\mathcal{R}_2(v,w):=\frac{1}{2|k|}\int_\R e^{-|k(v-\rho)|}\bigg[\frac{e^{2\rho}F_{0k}(\rho)}{B(\rho)-B(w)+i\iota \epsilon}+\frac{\sigma_k}{c_\ast}\big(2ke^{k\rho}\partial_\rho\Phi_0(\rho)+e^{k\rho}\partial_\rho^2\Phi_0(\rho)\big)\bigg]\,d\rho.
\end{split}
\end{equation}
We have the bounds (see \eqref{LAP7}-\eqref{LAP8} for the definition of the space $Y_{k,\varkappa_k}$)
\begin{equation}\label{BSD10}
\Big\|\big\|\mathcal{R}_2(v,w)\Psi(w-w_\ast)\big\|_{Y_{k,\varkappa_k}}\Big\|_{L^2(w\in\R)}\lesssim1/|k|.
\end{equation}
The proof of \eqref{BSD10} follows from standard calculations and the boundedness property of Hilbert transforms, and is subsumed in the proof of the stronger bounds \eqref{GVS10}-\eqref{GVS12'}. We omit the repeated details and refer to the proof of \eqref{GVS10}-\eqref{GVS12'}.

The desired bounds \eqref{BSD6} then follow from the limiting absorption principle, see Lemma \ref{LAP10}, applied for each $w$ with $|w-w_\ast|\leq 2$ and square integrate in $w$. 
\end{proof}

 Assume that $k\in\Z\backslash\{0\}$ with $|k|\ge2$, $\iota\in\{+,-\}$, $w_\ast\in\R$, and $\epsilon\in\R$ with $0<\epsilon<\epsilon_\ast \,e^{-2|w_\ast|}$ for sufficiently small $\epsilon_\ast>0$ from Lemma \ref{BSD4}. We now turn to the Gevrey estimates of $\Gamma_{k,\epsilon}^{\iota}$.

\subsubsection{The case $w\ge-15$} We first consider the the case $w\ge-15$. For fixed $w_\ast\in[-10,\infty)$, we define for $v,w\in\R$,
\begin{equation}\label{GVS1}
\Theta_{k,\epsilon}^{\iota, w_\ast}(v, w):=\Psi(w-w_\ast)\Gamma_{k,\epsilon}^{\iota}(v+w,w).
\end{equation}
 Define the Fourier multiplier operator $A_k$ as
\begin{equation}\label{GVS4}
\widehat{A_kh}(\xi):=e^{\delta_1\langle k,\xi\rangle^{1/2}}\widehat{\,\,h\,\,}(\xi),\qquad {\rm for\,\,any\,\,}\xi\in\R, h\in L^2(\R),
\end{equation}
and the norm for functions $h:\R^2\to \mathbb{C}$ with $h\in L^2(\R^2)$,
\begin{equation}\label{GVSM1}
\begin{split}
\|h\|_Z:=\Big\|\sup_{j\in\Z} \Big[&\big\|e^{\varkappa_k|v+w_\ast|}A_k\big[\Theta_{k,\epsilon}^{\iota, w_\ast}(v, \cdot)\big](w)\big\|_{L^2(v\in[j,j+2])}\\
&+|k|^{-1}\big\|e^{\varkappa_k|v+w_\ast|}\partial_vA_k\big[\Theta_{k,\epsilon}^{\iota, w_\ast}(v, \cdot)\big](w)\big\|_{L^2(v\in[j,j+2])}\Big]\Big\|_{L^2(\R)}.
\end{split}
\end{equation}
For notational conveniences, we also introduce the shifted $Y_k$ norm for $h:\R^2\to\mathbb{C}$,
\begin{equation}\label{GVSM1.1}
\begin{split}
\|h\|_{Y_{\varkappa_k}^{+w_\ast}}:=\Big\|\sup_{j\in\Z} \Big[&\big\|e^{\varkappa_k|v+w_\ast|}h(v,w)\big\|_{L^2(v\in[j,j+2])}\\
&+|k|^{-1}\big\|e^{\varkappa_k|v+w_\ast|}\partial_vh(v,w)\big\|_{L^2(v\in[j,j+2])}\Big]\Big\|_{L^2(w\in\R)}.
\end{split}
\end{equation}
Then
\begin{equation}\label{GVSM1.2}
\big\|h(v,w)\big\|_Z=\big\|A_k\big[h(v,\cdot)\big](w)\big\|_{Y_{\varkappa_k}^{+w_\ast}}.
\end{equation}
We have the following result.
\begin{proposition}\label{GVSM2}
 Assume that $k\in\Z\backslash\{0\}$ with $|k|\ge2$,  $w_\ast\in[-15,\infty)$, $\iota\in\{+,-\}$ and $\epsilon\in\R$ with $0<\epsilon<e^{-2|w_\ast|}\epsilon_\ast$. We have the bounds
\begin{equation}\label{GVSM3}
\begin{split}
&\sup_{j\in\Z}\big\|e^{\varkappa_k|v+w_\ast|}\big[ \big|A_k\big[\Theta_{k,\epsilon}^{\iota, w_\ast}(v, \cdot)\big](w)\big|+|k|^{-1}\big|\partial_vA_k\big[\Theta_{k,\epsilon}^{\iota, w_\ast}(v, \cdot)\big](w)\big|\big]\big\|_{L^2(v\in[j,j+2],w\in\R)}\\
&\lesssim (M_k^{\dagger}+|\sigma_k|)/|k|. 
\end{split}
\end{equation} 
\end{proposition}

\begin{proof}
We can assume, without loss of generality, $k\ge2$ and $M_k^{\dagger}+|\sigma_k|=1$. By lemma \ref{BSD4} we have
\begin{equation}\label{GVSM2.01}
\big\|\Theta_{k,\epsilon}^{\iota, w_\ast}(v, w)\big\|_{Y_{\varkappa_k}^{+w_\ast}}\lesssim 1/|k|,
\end{equation} 
which we will use to control the low frequency part of the desired bounds \eqref{GVSM3}.

Denote 
$$M:=\sup_{j\in\Z}\big\|e^{\varkappa_k|v+w_\ast|}\big[ \big|A_k\big[\Theta_{k,\epsilon}^{\iota, w_\ast}(v, \cdot)\big](w)\big|+|k|^{-1}\big|\partial_vA_k\big[\Theta_{k,\epsilon}^{\iota, w_\ast}(v, \cdot)\big](w)\big|\big]\big\|_{L^2(v\in[j,j+2],w\in\R)}.$$ We need to prove that $M\lesssim1/|k|$. From the equation \eqref{BSD9}, we obtain that for $v,w\in\R$,
\begin{equation}\label{GVS2}
\begin{split}
&\Theta_{k,\epsilon}^{\iota, w_\ast}(v, w)+\frac{1}{2|k|}\int_\R e^{-|k||v-\rho|}\frac{e^{2\rho+2w}D(\rho+w)}{B(\rho+w)-B(w)+i\iota\epsilon}\Theta_{k,\epsilon}^{\iota, w_\ast}(\rho, w)\,d\rho\\
&=\frac{1}{2|k|}\int_\R e^{-|k||v-\rho|}\frac{e^{2\rho+2w}}{B(\rho+w)-B(w)+i\iota\epsilon}h_{0k}^1(\rho,w)\,d\rho+\frac{1}{2|k|}\int_\R e^{-|k||v-\rho|} h_{0k}^2(\rho,w)\,d\rho\\
&:=H_{1k}(v,w)+H_{2k}(v,w).
\end{split}
\end{equation}
where for $\rho, w\in\R$, (recall that $\sigma_k=0$ for $|k|\ge k^\dagger$)
\begin{equation}\label{GVS3}
\begin{split}
&h_{0k}^1(\rho,w)=F_{0k}(\rho+w)\Psi(w-w_\ast),\\
&h_{0k}^2(\rho,w)=\frac{\sigma_k}{c_\ast}e^{|k|(\rho+w)}\big(2|k|\partial_\rho \Phi_{0}(\rho+w)+\partial_\rho^2\Phi_{0}(\rho+w)\big)\Psi(w-w_\ast). 
\end{split}
\end{equation}
Applying the Fourier multiplier operator $A_k$ (in the variable $w$) to \eqref{GVS2} and multiplying the cutoff function $\Psi^\ast(w-w_\ast)$, we obtain
\begin{equation}\label{GVS5}
\begin{split}
&\Psi^\ast(w-w_\ast)A_k\big[\Theta_{k,\epsilon}^{\iota, w_\ast}(v, \cdot)\big](w)+\int_\R \frac{e^{-|k||v-\rho|}e^{2\rho+2w}D(\rho+w)\Psi^\ast(w-w_\ast)}{2|k|(B(\rho+w)-B(w)+i\iota\epsilon)}A_k\big[\Theta_{k,\epsilon}^{\iota, w_\ast}(\rho, \cdot)\big](w)\,d\rho\\
&=\int_\R \frac{e^{-|k||v-\rho|}}{2|k|}e^{2\rho+2w_\ast}\frac{\Psi^\ast(w-w_\ast)}{B'(w_\ast)}\Big\{K(\rho,w)A_k\big[\Theta_{k,\epsilon}^{\iota, w_\ast}(\rho, \cdot)\big](w)-A_k \big[K(\rho,\cdot)\Theta_{k,\epsilon}^{\iota, w_\ast}(\rho, \cdot)\big](w)\Big\}\,d\rho\\
&\quad+\Psi^\ast(w-w_\ast)A_k\Big[H_{1k}(v,\cdot)\Big](w)+\Psi^\ast(w-w_\ast)A_k\Big[H_{2k}(v,\cdot)\Big](w)\\
&:=\mathcal{C}_k(v,w)+\Psi^\ast(w-w_\ast)A_k\Big[H_{1k}(v,\cdot)\Big](w)+\Psi^\ast(w-w_\ast)A_k\Big[H_{2k}(v,\cdot)\Big](w).
\end{split}
\end{equation}
where we set for $\rho,w\in\R$,
\begin{equation}\label{GVS6}
K(\rho,w):= \frac{e^{2(w-w_\ast)}D(\rho+w)B'(w_\ast)}{B(\rho+w)-B(w)+i\iota\epsilon}\Psi^{\ast\ast}(w-w_\ast),\qquad K^\ast(\rho,w):=\Psi^\ast(\rho)K(\rho,w).
\end{equation}
By the definitions \eqref{GVS3}, \eqref{GVS6} and the assumption \ref{MARr1}, (denoting $\widehat{\,\varphi\,}(\rho,\xi)$ as the Fourier transform of $\varphi(\rho,w)$ in $w$), in view of lemma \ref{gbo3}, we have the following bounds for $\rho\in\R$ and $w_\ast\in[-15,\infty)$,
\begin{equation}\label{GVS7}
\begin{split}
&\big\|e^{\delta_1\langle k,\xi\rangle^{1/2}}\widehat{\,h_{0k}^1}(\rho,\xi)\big\|_{L^2(\xi\in\R)}\lesssim \frac{e^{\mu^\ast_{\varkappa_k}(\rho+w_\ast)}}{1+e^{(\mu^\ast_{\varkappa_k}+\varkappa_k+8)(\rho+w_\ast)}},\\
&\big\|e^{\langle k,\xi\rangle^{4/5}}\widehat{\,h_{0k}^2}(\rho,\xi)\big\|_{L^2(\xi\in\R)}\lesssim \mathbf{1}_{[-w_\ast-8,-w_\ast+4]}(\rho),\\
&\big\|e^{\delta_0\langle \xi\rangle^{1/2}}\widehat{\,K\,}(\rho,\xi)\big\|_{L^2(\xi\in\R)}\lesssim \frac{1}{1+e^{8(\rho+w_\ast)}}\Big|\frac{B'(w_\ast)}{B(\rho+w_\ast)-B(w_\ast)+i\epsilon}\Big|, \quad{\rm for}\,\,|\rho|\ge1/2,\\
&\sup_{\eta\in\R}\big\|e^{\delta_0\langle \xi\rangle^{1/2}}\widehat{\,K^\ast}(\eta,\xi)\big\|_{L^2(\xi\in\R)}\lesssim \frac{1}{1+e^{8w_\ast}}.
\end{split}
\end{equation}
To apply the limiting absorption principle, see lemma \ref{LAP10}, we need to bound the terms in the last line of \eqref{GVS5} and prove that for $h\in\big\{\mathcal{C}_k(v,w), A_k\big[H_{1k}(v,\cdot)\big](w), A_k\big[H_{2k}(v,\cdot)\big](w)\big\}$,
\begin{equation}\label{GVS7.2}
\big\|h\big\|_{Y_k^{+w_\ast}}\lesssim (C_\gamma+\gamma M)/|k|,
\end{equation}
for sufficiently small $\gamma>0$ and suitable $C_\gamma\in(0,\infty)$.
 We first bound the term $A_k\big[H_{1k}(v,\cdot)\big](w)$. Set for $v\in\R$,
\begin{equation}\label{GVS8.9}
F_{0k}^{w_\ast}(v):=\Psi^{\ast\ast}(v-w_\ast)F_{0k}(v),
\end{equation}
and
\begin{equation}\label{GVS9}
\begin{split}
h_{0k}^{1\ast}(\rho,w)&:=\Psi(\rho)\frac{e^{2w-2w_\ast}B'(w_\ast)}{B(\rho+w)-B(w)+i\iota\epsilon}h_{0k}^{1}(\rho,w)\Psi^\ast(w-w_\ast)\\
&=\Psi(\rho)F_{0k}^{w_\ast}(\rho+w)\frac{e^{2w-2w_\ast}B'(w_\ast)}{B(\rho+w)-B(w)+i\iota\epsilon}\Psi^\ast(w-w_\ast).
\end{split}
\end{equation}
it follows from \eqref{GVS3} and lemma \ref{gbo3} that
\begin{equation}\label{GVS9.0'}
\begin{split}
&\Big\|\int_\R e^{-(\delta_0/2)\langle\alpha\rangle^{1/2}} \big|A_k\big[h_{0k}^1(\rho,\cdot)e^{i\alpha\cdot}\big](w)\big|\,d\alpha\Big\|_{L^2(w\in\R)}\\
&\lesssim \frac{e^{\mu^\ast_{\varkappa_k}(\rho+w_\ast)}}{1+e^{(\mu^\ast_{\varkappa_k}+\varkappa_k+8)(\rho+w_\ast)}},
\end{split}
\end{equation}
and for $\eta\in\R$,
\begin{equation}\label{GVS9.0}
\begin{split}
&\Big\|\int_\R A_k\big[h_{0k}^{1\ast}(\rho,\cdot)\big](w) e^{-i\rho \eta}\,d\rho\Big\|_{L^2(w\in\R)}\\
&\lesssim \Big\|\int_{\R}e^{-\delta_0\langle\alpha\rangle^{1/2}}e^{\delta_1\langle k,\xi\rangle^{1/2}}\big|\widehat{\,F_{0k}^{w_\ast}}(\xi-\alpha)\big|\,d\alpha\Big\|_{L^2(\xi\in\R)}\lesssim  \frac{e^{\mu^\ast_{\varkappa_k}w_\ast}}{1+e^{(\mu^\ast_{\varkappa_k}+\varkappa_k+8)w_\ast}}.
\end{split}
\end{equation}
We can then bound $A_k\big[H_{1k}(v,\cdot)\big](w)$ as follows. For $w\in\R$, by lemma \ref{gbo3}, the bounds \eqref{GVS9.0} and Parseval's identity, we have
\begin{equation}\label{GVS10}
\begin{split}
&\sup_{j\in\Z}\Big[e^{\varkappa_k|w_\ast+j|}\big\|\Psi^\ast(v-j)A_k\big[H_{1k}(v,\cdot)\big](w)\big\|_{L^2(v\in\R)}\Big]\\
&\lesssim \sup_{j\in\Z}\Big\|\int_{\R^2}\Psi^\ast(v-j)e^{\varkappa_k|w_\ast+\rho|}\frac{e^{2\rho+2w_\ast}(1-\Psi(\rho)) e^{-\delta_0\langle\alpha\rangle^{1/2}}}{|k|\big|B(\rho+w_\ast)-B(w_\ast)+i\iota\epsilon\big|}\big| A_k\big[h^{1}_{0k}(\rho, \cdot)e^{i\alpha \cdot}\big](w)\big|\,d\rho d\alpha\Big\|_{L^2(v\in\R)}\\
&+\sup_{j\in\Z}\Big\|e^{\varkappa_k|w_\ast+j|}\Psi^\ast(v-j)\frac{1}{2k}\int_{\R}e^{-k|v-\rho|}\frac{e^{2\rho+2w_\ast}\Psi^\ast(\rho)}{B'(w_\ast)}A_k\big[h_{0k}^{1\ast}(\rho,\cdot)\big](w)\,d\rho\Big\|_{L^2(v\in\R)}\\
&\lesssim \int_{\R^2}e^{\varkappa_k|w_\ast+\rho|}\frac{e^{2\rho+2w_\ast}(1-\Psi(\rho)) e^{-\delta_0\langle\alpha\rangle^{1/2}}}{|k|\big|B(\rho+w_\ast)-B(w_\ast)+i\iota\epsilon\big|}\big| A_k\big[h^{1}_{0k}(\rho, \cdot)e^{i\alpha \cdot}\big](w)\big|\,d\rho d\alpha\\
&\qquad+e^{(\varkappa_k+4)w_\ast}\Big\|\int_{\R}\langle k,\eta\rangle^{-2}e^{-\delta_0\langle\eta+\beta\rangle^{3/4}}\Big|\int_{\R}e^{i\beta\rho} A_k\big[h_{0k}^{1\ast}(\rho,\cdot)\big](w)\,d\rho\Big| d\beta\Big\|_{L^2(\eta\in\R)}.
\end{split}
\end{equation}
In view of \eqref{GVS9.0'}-\eqref{GVS9.0}, square integrating \eqref{GVS10} in $w\in\R$, we obtain that
\begin{equation}\label{GVS10.01}
\Big\|\sup_{j\in\Z}\big\|e^{|k||v+w_\ast|}A_k\big[H_{1k}(v,\cdot)\big](w)\big\|_{L^2(v\in[j,j+2])}\Big\|_{L^2(w\in\R)}\lesssim |k|^{-1}.
\end{equation}
In the above, we have used the inequalities for $\alpha,\beta\in\R$,
\begin{equation}\label{GVS13.1}
\begin{split}
\frac{1}{|k|}\Big|\int_{\R^2}e^{-k|v-\rho|}\Psi^\ast(v-j)e^{2\rho}\Psi^\ast(\rho)e^{-iv\alpha-i\rho\beta}\,dv d\rho\Big|\lesssim e^{-\varkappa_k|j|}\langle k,\alpha\rangle^{-2}e^{-\langle\alpha+\beta\rangle^{3/4}}.
\end{split}
\end{equation}
Similarly we have
\begin{equation}\label{GVS11}
|k|^{-1}\Big\|\sup_{j\in\Z}\big\|e^{|k||v+w_\ast|}\partial_vA_k\big[H_{1k}(v,\cdot)\big](w)\big\|_{L^2(v\in[j,j+2])}\Big\|_{L^2(w\in\R)}\lesssim |k|^{-1}.
\end{equation}
\eqref{GVS10.01}-\eqref{GVS11} imply the desired bounds \eqref{GVS7.2} for $h=A_k\big[H_{1k}(v,\cdot)\big](w)$.

The bounds on $A_k\big[H_{2k}(v,\cdot)\big](w)$ is simpler, and using \eqref{GVS7} we have
\begin{equation}\label{GVS12'}
\big\|A_k\big[H_{2k}(v,\cdot)\big](w)\big\|_{Y_{\varkappa_k}^{+w_\ast}}\lesssim |k|^{-1}.
\end{equation}
\eqref{GVS12'} imply \eqref{GVS7.2} for $h=A_k\big[H_{2k}(v,\cdot)\big](w)$.

We now bound the main commutator term $\mathcal{C}_k(v,w)$. Defining for $\rho,w\in\R$,
\begin{equation}\label{GVS12.1}
N(\rho,w):=A_k \big[K(\rho,\cdot)\Theta_{k,\epsilon}^{\iota, w_\ast}(\rho, \cdot)\big](w)-K(\rho,w)A_k\big[\Theta_{k,\epsilon}^{\iota, w_\ast}(\rho, \cdot)\big](w),\quad N^\ast(\rho,w):=\Psi^\ast(\rho)N(\rho,w).
\end{equation}
For the simplicity of notation below we also denote for $\rho\in\R$,
\begin{equation}\label{GVS12.2}
M(\rho):=\Big\|\langle k,\partial_w\rangle^{-1/2}A_k\big[\Theta_{k,\epsilon}^{\iota,w_\ast}(\rho,\cdot)\big](w)\Big\|_{L^2(w\in\R)}+\frac{1}{|k|}\Big\|\langle k,\partial_w\rangle^{-1/2}\partial_\rho A_k\big[\Theta_{k,\epsilon}^{\iota,w_\ast}(\rho,\cdot)\big](w)\Big\|_{L^2(w\in\R)},
\end{equation}
then it follows by dividing the frequency of the variable $w$ into low frequency and high frequency parts, that for any $\gamma\in(0,1)$,
\begin{equation}\label{GVS12.3}
\sup_{j\in\Z}\big\|e^{\varkappa_k|\rho+w_\ast|}M(\rho)\big\|_{L^2(\rho\in[j,j+2])}\lesssim C_\gamma +\gamma M.
\end{equation}
Notice that using \eqref{GVS7}, we have for $|\rho|\ge 1/2$,
\begin{equation}\label{GVS13}
\begin{split}
&\big\|N(\rho,w)\big\|_{L^2(w\in\R)}\lesssim\Big\|\int_\R\widehat{\,K\,}(\rho,\alpha)\Big[e^{\delta_1\langle k,\xi\rangle^{1/2}}-e^{\delta_1\langle k,\xi-\alpha\rangle^{1/2}}\Big]\widehat{\,\,\Theta_{k,\epsilon}^{\iota,w_\ast}}(\rho, \xi-\alpha)\,d\alpha\Big\|_{L^2(\xi\in\R)}\\
&\lesssim \Big\|\int_\R e^{(\delta_0/4)\langle \alpha\rangle^{1/2}}\big|\widehat{\,K\,}(\rho,\alpha)\big|e^{\delta_1\langle k,\xi-\alpha\rangle^{1/2}}\langle k,\xi-\alpha\rangle^{-1/2}\big|\widehat{\,\,\Theta_{k,\epsilon}^{\iota,w_\ast}}(\rho, \xi-\alpha)\big|\,d\alpha\Big\|_{L^2(\xi\in\R)}\\
&\lesssim M(\rho)\frac{1}{1+e^{8(\rho+w_\ast)}}\Big|\frac{B'(w_\ast)}{B(\rho+w_\ast)-B(w_\ast)+i\epsilon}\Big|.
\end{split}
\end{equation}
Denoting $Q(\rho,w):=\Psi^{\ast\ast}(\rho)\Theta_{k,\epsilon}^{\iota, w_\ast}(\rho, w)$ for $\rho,w\in\R$, for the brevity of notations, we have
\begin{equation}\label{GVS13.1}
\begin{split}
&\int_\R \big\|\langle k,\xi\rangle^{-1/2} e^{\delta_1\langle k,\xi\rangle^{1/2}}\widehat{\,Q\,}(\gamma,\xi)\big\|_{L^2(\xi\in\R)}\,d\gamma\\
&\lesssim |k|^{1/2}\big\|\langle\gamma/k\rangle \langle k,\xi\rangle^{-1/2}e^{\delta_1\langle k,\xi\rangle^{1/2}}\widehat{\,Q\,}(\gamma,\xi)\big\|_{L^2(\gamma,\xi\in\R)}\lesssim |k|^{1/2}\big\|\Psi^\ast(\rho)M(\rho)\big\|_{L^2(\rho\in\R)}.
\end{split}
\end{equation}
Therefore, in view of by \eqref{GVS7}, \eqref{GVS12.1} and \eqref{GVS13.1}, we have for $\eta\in\R$,
\begin{equation}\label{GVS13.5}
\begin{split}
&\Big\|\int_\R N^\ast(\rho,w)e^{-i\rho\eta}\,d\rho\Big\|_{L^2(w\in\R)}\\
&\lesssim\Big\|\int_{\R^2} \big|\widehat{\,\,K^\ast}(\gamma,\alpha)\big|\Big|e^{\delta_1\langle k,\xi\rangle^{1/2}}-e^{\delta_1\langle k,\xi-\alpha\rangle^{1/2}}\Big|\big|\widehat{\,Q\,}(\eta-\gamma, \xi-\alpha)\big|\,d\alpha d\gamma\Big\|_{L^2(\xi\in\R)}\\
&\lesssim\Big\|\int_{\R^2} e^{(\delta_0/4)\langle\alpha\rangle^{1/2}}\big|\widehat{\,\,K^\ast}(\gamma,\alpha)\big|e^{\delta_1\langle k,\xi-\alpha\rangle^{1/2}}\langle k,\xi-\alpha\rangle^{-1/2}\big|\widehat{\,Q\,}(\eta-\gamma, \xi-\alpha)\big|\,d\alpha  d\gamma\Big\|_{L^2(\xi\in\R)}\\
&\lesssim |k|^{1/2}\frac{1}{1+e^{8w_\ast}}\Big\|\Psi^{\ast\ast}(\rho)M(\rho)\Big\|_{L^2(\rho\in\R)}.
\end{split}
\end{equation}
Using the definition \eqref{GVS5}, the bounds \eqref{GVS13} for $|\rho|\ge1$, together with \eqref{GVS13.5} for $|\rho|\lesssim 1$, and \eqref{GVS12.3}, we can then bound for $w\in\R$ and any $\gamma\in(0,1)$ and suitable $C_\gamma\in(0,\infty)$,
\begin{equation}\label{GVS13.6}
\begin{split}
&\Big\|\sup_{j\in\Z}\Big[e^{\varkappa_k|j+w_\ast|}\big\|\Psi(v-j)\mathcal{C}_k(v,w)\big\|_{L^2(v\in\R)}\Big]\Big\|_{L^2(w\in\R)}\\
&\lesssim \frac{1}{2|k|}\bigg\|\sup_{j\in\Z}\bigg[e^{\varkappa_k|j+w_\ast|}\Big\|\int_\R e^{-|k||v-\rho|}\Psi(v-j)\frac{e^{2\rho+2w_\ast}}{B'(w_\ast)}\Psi^\ast(w-w_\ast)\\
&\hspace{1.5in} \times\Big[(1-\Psi(\rho))N(\rho,w)+\Psi(\rho)N^\ast(\rho,w)\Big]\,d\rho\Big\|_{L^2(v\in\R)}\bigg]\bigg\|_{L^2(w\in\R)}\\
&\lesssim \frac{1}{2|k|}\int_\R e^{\varkappa_k|\rho+w_\ast|}\frac{e^{2\rho+2w_\ast}}{1+e^{8(\rho+w_\ast)}}\frac{(1-\Psi(\rho))M(\rho)}{|B(\rho+w_\ast)-B(w_\ast)|} \,d\rho\\
& \qquad+e^{(\varkappa_k+4)w_\ast}\Big\|\int_{\R} \langle k,\eta\rangle^{-2}e^{-\langle \eta+\beta\rangle^{3/4}}\Big|\int_{\R}N^\ast(\rho,w)e^{i\beta\rho}\,d\rho\Big| d\beta \Big\|_{L^2(\eta,w\in\R)}\lesssim  |k|^{-1}(C_\gamma+\gamma M).
\end{split}
\end{equation}
Similarly, we have for any $\gamma\in(0,1)$,
\begin{equation}\label{GVS13.7}
\begin{split}
|k|^{-1}\Big\|\sup_{j\in\Z} \big\|e^{\varkappa_k|v+w_\ast|}\partial_v\mathcal{C}_k(v,w)\big\|_{L^2(v\in[j,j+2])}\Big\|_{L^2(w\in\R)}\lesssim  |k|^{-1}(C_\gamma+\gamma M).\end{split}
\end{equation}
\eqref{GVS13.6}-\eqref{GVS13.7}, together with \eqref{GVS10.01}, \eqref{GVS11}-\eqref{GVS12'} complete the proof of \eqref{GVS7.2}.
Therefore, by the limiting absorption principle, see lemma \ref{LAP10} with a translation in $v$ by $w$, we conclude that for $\gamma\in(0,1)$ and suitable $C_\gamma\in(0,\infty)$,
\begin{equation}\label{GVS15}
\begin{split}
&\sup_{j\in\Z}\big\|\Psi^\ast(w-w_\ast)e^{\varkappa_k|v+w_\ast|}\big[\big|A_k\big[\Theta_{k,\epsilon}^{\iota,w_\ast}(v,\cdot)\big](w)\big|+|k|^{-1}\big|\partial_vA_k\big[\Theta_{k,\epsilon}^{\iota,w_\ast}(v,\cdot)\big](w)\big|\big]\big\|_{L^2(v,\in[j,j+2],w\in\R)}\\
&\lesssim (C_\gamma +\gamma M)/|k|.
\end{split}\end{equation}
Another simple commutator argument comparing $\Psi^\ast(w-w_\ast)A_k\big[\Theta_{k,\epsilon}^{\iota,w_\ast}(v,\cdot)\big](w)$ with the nonlocalized $A_k\big[\Theta_{k,\epsilon}^{\iota,w_\ast}(v,\cdot)\big](w)$ shows that for $\gamma\in(0,1)$,
\begin{equation}\label{GVS14.2}
M \lesssim (C_\gamma +\gamma M)/|k|,
\end{equation}
which imply the desired bounds once we choose $\gamma\in(0,1)$ small enough so that the second term on the right hand side can be absorbed by the left hand side.
\end{proof}

\subsubsection{The case $w\leq-15$} We now turn to the case $w\leq-15$. Assume that $w_\ast\leq-15$. Define the norm for functions $h:\R^2\to \mathbb{C}$ with $h\in L^2(\R^2)$,
\begin{equation}\label{GVS16}
\begin{split}
\|h\|_{Z^\ast}:=\Big\|\sup_{j\in\Z} \Big[&\big\|\varrho_{\varkappa_k,w_\ast}(v+w_\ast)A_k\big[h(v,\cdot)\big](w)\big\|_{L^2(v\in[j,j+2])}\\
&+|k|^{-1}\big\|\varrho_{\varkappa_k,w_\ast}(v+w_\ast)\partial_vA_k\big[h(v,\cdot)\big](w)\big\|_{L^2(v\in[j,j+2])}\Big]\Big\|_{L^2(w\in\R)}.
\end{split}
\end{equation}
For notational convenience, we also introduce the translated $Y_{k,w_\ast}$ norm for $h(v,w):\R^2\to\mathbb{C}$,
\begin{equation}\label{GVS16.1}
\begin{split}
\|h(v,w)\|_{Y_{\varkappa_k,w_\ast}^{+w_\ast}}:=\Big\|\sup_{j\in\Z} \Big[&\big\|\varrho_{\varkappa_k,w_\ast}(v+w_\ast)h(v,w)\big\|_{L^2(v\in[j,j+2])}\\
&+|k|^{-1}\big\|\varrho_{\varkappa_k,w_\ast}(v+w_\ast)\partial_vh(v,w)\big\|_{L^2(v\in[j,j+2])}\Big]\Big\|_{L^2(w\in\R)}.
\end{split}
\end{equation}
Then
\begin{equation}\label{GVS16.2}
\big\|h\big\|_{Z^\ast}=\big\|A_k\big[h(v,\cdot)\big](w)\big\|_{Y_{k,w_\ast}^{+w_\ast}}.
\end{equation}
Our main result in this case is the following.
\begin{proposition}\label{GVS17}
Assume that $k\in\Z\backslash\{0\}$ with $|k|\ge2$, $w_\ast\in(-\infty,-15]$, $\iota\in\{+,-\}$, $\epsilon\in\R$ with $0<\epsilon<e^{2w_\ast}\epsilon_{\ast}$. Then we have the bounds
\begin{equation}\label{GVS18}
\begin{split}
&\sup_{j\in\Z}\big\|\varrho_{\varkappa_k,w_\ast}(v+w_\ast)\big[\big|A_k\big[\Theta_{k,\epsilon}^{\iota, w_\ast}(v, \cdot)\big](w)\big|+|k|^{-1}\big|\partial_vA_k\big[\Theta_{k,\epsilon}^{\iota, w_\ast}(v, \cdot)\big](w)\big|\big]\big\|_{L^2(v\in[j,j+2], w\in\R)}\\
&\lesssim (M_k^{\dagger}+|\sigma_k|)/|k|.
\end{split}
\end{equation}
\end{proposition}

\begin{proof}
We assume, without loss of generality, $k\ge2$ and $M_k^{\dagger}+|\sigma_k|=1$. Denote 
$$M:=\sup_{j\in\Z}\big\|\varrho_{\varkappa_k,w_\ast}(v+w_\ast)\big[\big|A_k\big[\Theta_{k,\epsilon}^{\iota, w_\ast}(v, \cdot)\big](w)\big|+|k|^{-1}\big|\partial_vA_k\big[\Theta_{k,\epsilon}^{\iota, w_\ast}(v, \cdot)\big](w)\big|\big]\big\|_{L^2(v\in[j,j+2], w\in\R)}.$$ We need to prove that $M\lesssim1/|k|$. By lemma \ref{BSD4}, we have 
\begin{equation}\label{GVS18.1}
\big\|\Theta_{k,\epsilon}^{\iota, w_\ast}(v, w)\big\|_{Y_{\varkappa_k,w_\ast}^{+w_\ast}}\lesssim 1/|k|,
\end{equation}
which will be used to control the low frequency part of the desired bounds \eqref{GVS18}.

 Using the equation \eqref{BSD7} and recalling the definitions \eqref{GVS3} for $h_{0k}^1, h_{0k}^2$, we obtain that
\begin{equation}\label{GVS19}
\begin{split}
&\Theta_{k,\epsilon}^{\iota, w_\ast}(v, w)+\int_\R \mathcal{G}_k^w(v+w,\rho+w)\Big[\frac{e^{2\rho+2w}D(\rho+w)}{B(\rho+w)-B(w)+i\iota\epsilon}-V_{w}(\rho+w)\Big]\Theta_{k,\epsilon}^{\iota, w_\ast}(\rho, w)\,d\rho\\
&=\int_\R \mathcal{G}_k^w(v+w,\rho+w)\frac{e^{2\rho+2w}h_{0k}^1(\rho,w)}{B(\rho+w)-B(w)+i\iota\epsilon}\,d\rho+\int_\R \mathcal{G}_k^w(v+w,\rho+w) h_{0k}^2(\rho,w)\,d\rho\\
&:=H_{1k}^\ast(v,w)+H_{2k}^\ast(v,w).
\end{split}
\end{equation}
Recall the definition \eqref{ENHQ2} and define for $w_\ast\in\R$,
\begin{equation}\label{GVS19.1}
\mathcal{F}_k^{w_\ast}(v,\rho,w)= \mathcal{G}_k^w(v+w,\rho+w)\Psi^{\ast\ast}(w-w_\ast).
\end{equation}
We shall use the following bounds on  $\mathcal{F}_k^{w_\ast}(v,\rho,w)$, which follow from proposition \ref{ENHQ1},
\begin{equation}\label{GVS19.01}
\begin{split}
&\sup_{\xi\in\R}\big|e^{\delta_0\langle\xi\rangle^{1/2}}\widehat{\,\,\mathcal{F}_k^{w_\ast}}(v,\rho,\xi)\big|\lesssim|k|^{-1}\varpi_{\varkappa_k,w_\ast}(v+w_\ast,\rho+w_\ast),\quad {\rm for}\,\,v,\rho\in\R,\\
&\Big|\int_\R\Psi^\ast(v-j)\Psi^\ast(\rho)\mathcal{F}_k^{w_\ast}(v,\rho,w)e^{-iv\alpha-i\rho\beta-iw\xi}\,dv d\rho dw\Big|\\
&\lesssim \varpi_{\varkappa_k,w_\ast}(j+w_\ast,w_\ast)e^{-\delta_0\langle\xi\rangle^{1/2}}\langle k,\alpha\rangle^{-2} e^{-\delta_0\langle\alpha+\beta\rangle^{1/2}},\quad {\rm for}\,\,j\in\Z, \alpha,\beta,\xi\in\R.
\end{split}
\end{equation}

 Define for $\rho,w\in\R$,
\begin{equation}\label{GVS19.2}
L(\rho,w):=\Big[\frac{e^{2\rho+2w}D(\rho+w)}{B(\rho+w)-B(w)+i\iota\epsilon}-V_{w}(\rho+w)\Big]\Psi^{\ast\ast}(w-w_\ast),\quad L^\ast(\rho,w):=\Psi^\ast(\rho)L(\rho,w).
\end{equation}
By lemma \ref{gbo3}, we have the following bounds for $L$,
\begin{equation}\label{GVS19.3}
\begin{split}
&\sup_{\xi\in\R}e^{\delta_0\langle\xi\rangle^{1/2}}\big|\widehat{\,L\,}(\rho,\xi)\big|\lesssim e^{-2|\rho|},\,\,{\rm for}\,\,|\rho|\ge1,\quad {\rm and}\quad\sup_{\eta,\xi\in\R}e^{\delta_0\langle\xi\rangle^{1/2}}\big|\widehat{\,L^\ast}(\eta,\xi)\big|\lesssim1.
\end{split}
\end{equation}

Applying the Fourier multiplier operator $A_k$ (in the variable $w$) to \eqref{GVS19} and multiplying $\Psi^\ast(w-w_\ast)$, we obtain that
\begin{equation}\label{GVS20}
\begin{split}
&\Psi^\ast(w-w_\ast)A_k\big[\Theta_{k,\epsilon}^{\iota, w_\ast}(v, \cdot)\big](w)\\
&+\int_\R \mathcal{F}_k(v,\rho,w)\Big[\frac{e^{2\rho+2w}D(\rho+w)}{B(\rho+w)-B(w)+i\iota\epsilon}-V_{w}(\rho+w)\Big]\Psi^\ast(w-w_\ast) A_k\Big[\Theta_{k,\epsilon}^{\iota, w_\ast}(\rho, \cdot)\Big](w)\,d\rho\\
&=\Psi^\ast(w-w_\ast)A_k\Big[H_{1k}^\ast(v,\cdot)\Big](w)+\Psi^\ast(w-w_\ast)A_k\big[H_{2k}^\ast(v,\cdot)\Big](w)+\Psi^\ast(w-w_\ast)\mathcal{C}_k^\ast(v,w),
\end{split}
\end{equation}
where 
\begin{equation}\label{GVS21}
\begin{split}
\mathcal{C}_k^\ast(v,w):=&\int_\R  \mathcal{F}_k^{w_\ast}(v,\rho,w)L(\rho,w)A_k\big[\Theta_{k,\epsilon}^{\iota, w_\ast}(\rho, \cdot)\big](w)\,d\rho\\
&-A_k\bigg\{\int_\R  \mathcal{F}_k^{w_\ast}(v,\rho,\cdot)L(\rho,\cdot)\Theta_{k,\epsilon}^{\iota, w_\ast}(\rho, \cdot)\,d\rho\bigg\}(w).
\end{split}
\end{equation}

To apply the limiting absorption principle, see lemma \ref{LAP25}, we need to bound the terms on the right hand side of the last line of \eqref{GVS20}, and prove for $h\in\big\{A_kH_{1k}^\ast(v,w), A_kH_{2k}^\ast(v,w), \mathcal{C}_k^\ast(v,w)\big\}$,
\begin{equation}\label{GVS21.1}
\big\|h(v,w)\big\|_{Y_{k,w_\ast}^{+w_\ast}}\lesssim (C_\gamma+\gamma M)/|k|,
\end{equation}
for sufficiently small $\gamma>0$ and suitable $C_\gamma\in(0,\infty)$.

For $h=A_k\big[H_{1k}^\ast(v,\cdot)\big](w)$,  using \eqref{ENHQ4}-\eqref{ENHQ6}, \eqref{GVS9.0'}-\eqref{GVS9.0} and lemma \ref{gbo3}, we can bound, similar to \eqref{GVS10} that for any $w\in\R$,
\begin{equation}\label{GVS21.2}
\begin{split}
&\sup_{j\in\Z}\Big[\varrho_{\varkappa_k,w_\ast}(j+w_\ast)\big\|\Psi^\ast(v-j)A_k\big[H^{\ast}_{1k}(v,\cdot)\big](w)\big\|_{L^2(v\in\R)}\Big]\\
&\lesssim \int_{\R^2}\varrho_{\varkappa_k,w_\ast}(\rho+w_\ast)\frac{e^{2\rho+2w_\ast}(1-\Psi(\rho)) e^{-\delta_0\langle\alpha\rangle^{1/2}}}{|k|\big|B(\rho+w_\ast)-B(w_\ast)+i\iota\epsilon\big|}\big|A_k\big[h^{1}_{0k}(\rho, \cdot)e^{i\alpha\cdot}\big](w)\big|\,d\rho d\alpha \\
&\qquad+\varrho_{\varkappa_k,w_\ast}(w_\ast)\Big\|\int_{\R}\langle k,\eta\rangle^{-2}e^{-\delta_0\langle\eta+\alpha\rangle^{1/2}}\Big|\int_\R e^{i\alpha \rho}A_k\big[h_{0k}^{1\ast}(\rho,w)\big](w)\,d\rho\Big| d\alpha \Big\|_{L^2(\eta\in\R)}.
\end{split}
\end{equation}
In view of \eqref{GVS9.0'}-\eqref{GVS9.0}, from \eqref{GVS21.2} we conclude that
\begin{equation}\label{GVS21.3}
\Big\|\sup_{j\in\Z}\big\|\varrho_{k,w_\ast}(v+w_\ast)A_k\big[H^{\ast}_{1k}(v,\cdot)\big](w)\big\|_{L^2(v\in[j,j+2])}\Big\|_{L^2(w\in\R)}\lesssim |k|^{-1}.
\end{equation}
Similarly we have
\begin{equation}\label{GVS21.4}
\Big\|\sup_{j\in\Z}|k|^{-1}\big\|\varrho_{k,w_\ast}(v+w_\ast)\partial_vA_k\big[H^{\ast}_{1k}(v,\cdot)\big](w)\big\|_{L^2(v\in[j,j+2])}\Big\|_{L^2(w\in\R)}\lesssim |k|^{-1}.
\end{equation}
\eqref{GVS21.3}-\eqref{GVS21.4} imply the desired bounds \eqref{GVS21.1} for $h=A_k\big[H^\ast_{1k}(v,\cdot)\big](w)$.

The bounds on $A_kH^\ast_{2k}$ is simpler, and we have
\begin{equation}\label{GVS12}
\Big\|A\big[H_{2k}^\ast(v,\cdot)\big](w)\Big\|_{Y_{k,w_\ast}^{+w_\ast}}\lesssim |k|^{-1}.
\end{equation}
\eqref{GVS12} imply \eqref{GVS21.1} for $h=A_k\big[H^\ast_{2k}(v,\cdot)\big](w)$.

We now prove \eqref{GVS21.1} for the main commutator term $h=\mathcal{C}_k^\ast(v,w)$. We can decompose
\begin{equation}\label{GVS21.9}
\begin{split}
&\mathcal{C}_k^\ast(v,w)=\mathcal{C}_{k1}^\ast(v,w)+\mathcal{C}_{k2}^\ast(v,w):=\int_\R \mathcal{F}_k^{w_\ast}(v,\rho,w)\Phi_k(\rho)L^\ast(\rho,w)A\big[\Theta_{k,\epsilon}^{\iota, w_\ast}(\rho, \cdot)\big](w)\,d\rho\\
&-A_k\bigg\{\int_\R \mathcal{F}_k^{w_\ast}(v,\rho,\cdot)\Phi_k(\rho)L^\ast(\rho,w)\Theta_{k,\epsilon}^{\iota, w_\ast}(\rho, \cdot)\,d\rho\bigg\}(w)+\mathcal{C}_{k2}^\ast(v,w).
\end{split}
\end{equation}
We first bound $\mathcal{C}_{k2}^\ast(v,w)$. Note that in this case the singularity in $\rho$ is removed thanks to the function $1-\Phi_k(\rho)$. 
Define for $\rho\in\R$,
\begin{equation}\label{GVS21.92}
M^\ast(\rho):=\Big\|\langle k,\xi\rangle^{-1/2}e^{\delta_1\langle k,\xi\rangle^{1/2}}\Big[\big|\widehat{\,\Theta_{k,\epsilon}^{\iota, w_\ast}}(\rho,\xi)\big|+|k|^{-1}\big|\partial_\rho\widehat{\,\Theta_{k,\epsilon}^{\iota, w_\ast}}(\rho,\xi)\big|\Big]\Big\|_{L^2(\xi\in\R)}.
\end{equation}
It follows that for any $\gamma\in(0,1)$ and suitable $C_\gamma\in(0,\infty)$,
\begin{equation}\label{GVS21.93}
\sup_{j\in\Z}\big\|\varrho_{\varkappa_k,w_\ast}(\rho+w_\ast)M^\ast(\rho)\big\|_{L^2(\rho\in[j,j+2])}\lesssim C_\gamma+\gamma M.
\end{equation}
Define for $\rho, w, \alpha,\beta\in\R$
\begin{equation}\label{GVSM2.00}
\begin{split}
P_{ \alpha,\beta}(\rho,w)&:=A_k\big[\Theta_{k,\epsilon}^{\iota, w_\ast}(\rho,\cdot)\big](w)e^{i(\alpha+\beta)w}-A_k\big[\Theta_{k,\epsilon}^{\iota, w_\ast}(\rho,\cdot)e^{i(\alpha+\beta)\cdot}\big](w),\\
P^\ast_{\alpha,\beta}(\rho,w)&:=\Psi^\ast(\rho)P_{\alpha,\beta}(\rho,w).
\end{split}
\end{equation}
We have the bound
\begin{equation}\label{GVSM2.1}
\begin{split}
&\int_{\R^2}e^{-(\delta_0/2)\langle\alpha\rangle^{1/2}-(\delta_0/2)\langle\beta\rangle^{1/2}}\big\|P_{\alpha,\beta}(\rho,w)\big\|_{L^2(w\in\R)} d\alpha d\beta\\
&\lesssim \Big[\int_{\R^2}e^{-(\delta_0/2)(\langle\alpha\rangle^{1/2}+\langle\beta\rangle^{1/2})} \big|e^{\delta_1\langle k,\xi-\alpha-\beta\rangle^{1/2}}-e^{\delta_1\langle k,\xi\rangle^{1/2}}\big|^2\big|\widehat{\,\,\Theta_{k,\epsilon}^{\iota, w_\ast}}(\rho,\xi-\alpha-\beta)\big|^2\,d\xi d\alpha d\beta\Big]^{1/2}\\
&\lesssim M^\ast(\rho),
\end{split}
\end{equation}
and the bound,
\begin{equation}\label{GVSM2.2}
\begin{split}
&\int_{\R^2}e^{-(\delta_0/2)\langle\alpha\rangle^{1/2}-(\delta_0/2)\langle\beta\rangle^{1/2}}\Big\|\widehat{\,\,P^\ast_{\alpha,\beta}}(\gamma,w)\Big\|_{L^2(w\in\R)}\,d\alpha d\beta d\gamma\\
&\lesssim |k|^{1/2}\bigg[\int_{\R^2}e^{-(\delta_0/2)(\langle\alpha\rangle^{1/2}+\langle\beta\rangle^{1/2})}\langle\gamma/k\rangle^2\big|\widehat{\,\,P^\ast_{\alpha,\beta}}(\gamma,w)\big|^2\,d\alpha d\beta d\gamma dw\bigg]^{1/2}\\
&\lesssim |k|^{1/2}(C_\gamma+\gamma M)\varpi_{\varkappa_k,w_\ast}(w_\ast,0).
\end{split}
\end{equation}
Using \eqref{GVS19.01} and \eqref{GVS19.3}, we can obtain from \eqref{GVS21.93} that,
\begin{equation}\label{GVS24}
\begin{split}
&\Big\|\sup_{j\in\Z}\Big[\varrho_{\varkappa_k,w_\ast}(j+w_\ast)\big\|\Psi^\ast(v-j)\mathcal{C}_{k2}^\ast(v,w)\big\|_{L^2(\R)}\Big]\Big\|_{L^2(w\in\R)}\\
&\lesssim |k|^{-1}\Big\|\int_{\R^2} \varrho_{\varkappa_k,w_\ast}(\rho+w_\ast)(1-\Psi(\rho))e^{-\delta_0\langle\alpha\rangle^{1/2}-\delta_0\langle\beta\rangle^{1/2}}e^{-2|\rho|}\big|P_{\alpha,\beta}(\rho,w)\big|\,d\alpha d\beta d\rho\Big\|_{L^2(w\in\R)}\\
&\lesssim|k|^{-1}\int_{\R^2} \varrho_{\varkappa_k,w_\ast}(\rho+w_\ast)(1-\Psi(\rho))e^{-2|\rho|}M^\ast(\rho)\, d\rho\lesssim \Big[C_\gamma+\gamma M\Big]/|k|.
\end{split}
\end{equation}
Similarly,
\begin{equation}\label{GVS24.1}
|k|^{-1}\Big\|\sup_{j\in\Z}\Big[\varrho_{\varkappa_k,w_\ast}(j+w_\ast)\big\|\Psi^\ast(v-j)\partial_v\mathcal{C}_{k2}^\ast(v,w)\big\|_{L^2(\R)}\Big]\Big\|_{L^2(w\in\R)}\lesssim \Big[C_\gamma+\gamma M\Big]/|k|.
\end{equation}
\eqref{GVS24}-\eqref{GVS24.1} imply that 
\begin{equation}\label{GVS25}
\big\|\mathcal{C}_{k2}^\ast(v,w)\big\|_{Y_{\varkappa_k,w_\ast}^{+w_\ast}}\lesssim (C_\gamma+\gamma M)/|k|.
\end{equation}
In view of the the bounds \eqref{GVS19.01}, the definition \eqref{GVSM2.00} and the bound \eqref{GVSM2.2}, we obtain that
\begin{equation}\label{GVS26}
\begin{split}
&\Big\|\sup_{j\in\Z}\Big[\varrho_{\varkappa_k,w_\ast}(j+w_\ast)\big\|\Psi^\ast(v-j)\mathcal{C}_{k1}^{\ast}(v,w)\big\|_{L^2(v\in\R)}\Big]\Big\|_{L^2(w\in\R)}\\
&\lesssim \varrho_{\varkappa_k,w_\ast}(w_\ast)\Big\|\int_{\R^4}\frac{ e^{-\delta_0(\langle\eta+\alpha_1\rangle^{-1/2}+\langle\beta_1\rangle^{1/2}+\langle\beta_2\rangle^{1/2})} }{\langle k,\eta\rangle^2}\big|\widehat{\,\,P^\ast_{\beta_1,\beta_2}}(\alpha_1+\alpha_2,w)\big|\,d\alpha d\beta\Big\|_{L^2(\eta, w\in\R)}\\
&\lesssim |k|^{-1} \big[C_\gamma+\gamma M\big],
\end{split}
\end{equation}
where we used the notations that $d\alpha=d\alpha_1 d\alpha_2$ and $d\beta=d\beta_1 d\beta_2$.
Similarly,
\begin{equation}\label{GVS26'}
|k|^{-1}\Big\|\sup_{j\in\Z}\Big[\varrho_{\varkappa_k,w_\ast}(j+w_\ast)\big\|\Psi^\ast(v-j)\partial_v\mathcal{C}_{k1}^{\ast}(v,w)\big\|_{L^2(v,w\in\R)}\Big]\Big\|_{L^2(w\in\R)}\lesssim |k|^{-1} \big[C_\gamma+\gamma M\big].\end{equation}
The proof of \eqref{GVS21.1} is then complete. The desired bounds then follow from the limiting absorption principle as in the case $w\ge-15$. We omit the routine details.
\end{proof}

\section{Bounds on the spectral density function II: refined bounds}\label{sec:sdf2}

Assume that $k\in\mathbb{Z}\backslash\{0\}$ and $|k|\ge2$. By the bounds \eqref{BSD5}-\eqref{BSD6}, \eqref{GVSM3} and \eqref{GVS18}, the limit 
\begin{equation}\label{GVSL0}
\Gamma_k(v,w):=(-i)\lim_{\epsilon\to0+}\big[\Gamma_{k,\epsilon}^{+}(v,w)-\Gamma_{k,\epsilon}^{-}(v,w)\big]=2\lim_{\epsilon\to0+}\Im \,\Gamma_{k,\epsilon}^{+}(v,w),
\end{equation}
exists at least along a sequence of $\epsilon\to0+$ in the space $L_{\rm loc}(\R^2)$, and it follows from \eqref{BSD3} that the limit $\Gamma_k(v,w)$ satisfies the equation for $v,w\in\R$,
\begin{equation}\label{GVSL1}
(k^2-\partial_v^2)\Gamma_k(v,w)+{\rm P.V.} \frac{e^{2v}D(v)\Gamma_k(v,w)}{B(v)-B(w)}=-2\pi\frac{e^{2w}\big(D(w)\digamma_k(w)-F_{0k}(w)\big)}{B'(w)}\delta(v-w).
\end{equation}
In the above for $w\in\R$,
\begin{equation}\label{GVSL1.1}
\digamma_k(w):=\lim_{\epsilon\to0+} \Im \,\Gamma^+_{k,\epsilon}(w,w).
\end{equation}
For $w_\ast\in\R$ and denoting $\digamma_k^{w_\ast}(w):=\digamma(w)\Psi^\ast(w-w_\ast)$, in view of \eqref{GVSM3} and \eqref{GVS18}, we have the bounds
\begin{equation}\label{GVS1.2}
\big\|e^{\delta_1\langle k,\xi\rangle^{1/2}}\widehat{\,\,\digamma_k^{w_\ast}}(\xi)\big\|_{L^2(\xi\in\R)}\lesssim \Big[e^{-\mu_{\varkappa_k}|w_\ast|}{\bf 1}_{w_\ast\leq0}+e^{-\varkappa_k|w_\ast|}{\bf 1}_{w_\ast\ge0}\Big] (M_k^{\dagger}+|\sigma_k|).
\end{equation}

We briefly discuss the existence in the limits \eqref{GVSL0} and \eqref{GVSL1.1}. The regularity property of $\Gamma^+_{k,\epsilon}(v,w)$ for sufficiently small $\epsilon$ ensures that we can at least take a sequential limit in $L^2_{\rm loc}(\R^2)$. The fact that the limit is unique follows from general representation theory of self adjoint operators, as a consequence of the existence and uniqueness properties of spectral measures corresponding to the self adjoint operator $L_k$ (with a smoothing in our case since we passed from the vorticity to the stream function). Alternatively, one can prove more refined bounds on $\partial_\epsilon \Gamma^+_{k,\epsilon}(v,w)$ from which the existence of the limit of $\Gamma^+_{k,\epsilon}(v,w)$ as $\epsilon\to0+$ follows easily, as in \cite{JiaL}. In our case, since the existence of a sequential limit is sufficient and our main focus is on the quantitative estimates, we will not go into the details. 

From \eqref{GVSL1} it is clear, heuristically at least, that $\Gamma_k(v,w)$ with $v,w\in\R$ is generated by the source term $-2\pi\frac{e^{2w}\big(D(w)\digamma_k(w)-F_{0k}(w)\big)}{B'(w)}\delta(v-w)$ and should decay when $v$ is away from $w$. (The case $|k|=1$ is again special due to the presence of embedded eigenvalue $\lambda=0$ for $L_k$.) To establish more quantitative bounds in this direction, we use the limiting absorption principle to obtain low regularity bounds, see \eqref{LAP11.5'} and \eqref{LAP26.5'}, and the commutator argument for Gevrey regularity estimates.

Denote for $w\in\R$,
\begin{equation}\label{GVSL12.1}
\digamma_k^\ast(w):=-2\pi\frac{e^{2w}\big(D(w)\digamma_k(w)-F_{0k}(w)\big)}{B'(w)}.
\end{equation}
By the assumption \eqref{MARr1} and the bound \eqref{GVS1.2}, we have for $w_\ast\in\R$ and $\digamma_k^{w_\ast}(w):=\digamma_k^\ast(w)\Psi^\ast(w-w_\ast)$,
\begin{equation}\label{GVSLL1}
\big\|e^{\delta_1\langle k,\xi\rangle^{1/2}}\widehat{\,\,\digamma_k^{w_\ast}}(\xi)\big\|_{L^2(\xi\in\R)}\lesssim \big[e^{-\mu_{\varkappa_k}|w_\ast|}{\bf 1}_{(-\infty,0]}(w_\ast)+e^{-(\varkappa_k+4)|w_\ast|}{\bf 1}_{(0,\infty)}(w_\ast)\big](M_k^{\dagger}+|\sigma_k|).
\end{equation}
We can reformulate \eqref{GVSL1} for small $\epsilon>0$ and $w\ge-15$ as
\begin{equation}\label{GVSL2}
\begin{split}
&\Gamma_k(v,w)+\frac{1}{2|k|}\int_\R e^{-|k||v-\rho|}\frac{e^{2\rho}D(\rho)\Gamma_k(\rho,w)}{B(\rho)-B(w)+i\epsilon}\,d\rho\\
&=\digamma_k^\ast(w)\frac{e^{-|k||v-w|}}{2|k|}+\frac{1}{2|k|}\int_\R e^{-|k||v-\rho|}\Big[\frac{e^{2\rho}D(\rho)\Gamma_k(\rho,w)}{B(\rho)-B(w)+i\epsilon}-{\rm P.V.} \frac{e^{2\rho}D(\rho)\Gamma_k(\rho,w)}{B(\rho)-B(w)}\Big]\,d\rho,
\end{split}
\end{equation}
and for $w\leq-15$ as
\begin{equation}\label{GVSL3}
\begin{split}
&\Gamma_k(v,w)+\int_\R \mathcal{G}_k^w(v,\rho)\Big[\frac{e^{2\rho}D(\rho)}{B(\rho)-B(w)+i\epsilon}-V_{w}(\rho)\Big]\Gamma_k(\rho,w)\,d\rho\\
&=\digamma_k^\ast(w)\mathcal{G}_k^w(v,w)+\int_\R \mathcal{G}_k^w\Big[\frac{e^{2\rho}D(\rho)\Gamma_k(\rho,w)}{B(\rho)-B(w)+i\epsilon}-{\rm P.V.} \frac{e^{2\rho}D(\rho)\Gamma_k(\rho,w)}{B(\rho)-B(w)}\Big]\,d\rho.
\end{split}
\end{equation}

\begin{proposition}\label{GVSL4}
Assume that $k\in\Z\backslash\{0\}$, $|k|\ge2$, and $\Gamma_k(v,w)$ are defined as in \eqref{GVSL0}. For $w_\ast\in\R$, define for $v,w\in\R$ (for the simplicity of notations)
\begin{equation}\label{GVSL5}
Q_k^{w_\ast}(v,w):=\Gamma_k(v+w,w)\Psi(w-w_\ast).
\end{equation}
 The we have the following bounds
 
(i) {\it The case of $w_\ast\leq-15$.} \begin{equation}\label{GVSL6}
\begin{split}
&\sup_{j\in\Z}\,\zeta_{\varkappa_k,w_\ast}(j+w_\ast,w_\ast)\Big\|\big|A_k\big[Q_k^{w_\ast}(v,\cdot)\big](w)\big|+|k|^{-1}\big|\partial_vA_k\big[Q_k^{w_\ast}(v,\cdot)\big](w)\big|\Big\|_{L^2(v\in[j,j+2], w\in\R)}\\
&\lesssim e^{-\mu_{\varkappa_k}|w_\ast|} (M_k^{\dagger}+|\sigma_k|)/|k|;
\end{split}
\end{equation}

(i) {\it The case of $w_\ast\ge-15$.} \begin{equation}\label{GVSL7}\begin{split}
&\sup_{j\in\Z}\,e^{\varkappa_k|j|}\Big\|\big|A_k\big[Q_k^{w_\ast}(v,\cdot)\big](w)\big|+|k|^{-1}\big|\partial_vA_k\big[Q_k^{w_\ast}(v,\cdot)\big](w)\big|\Big\|_{L^2(v\in[j,j+2], w\in\R)}\\
&\lesssim e^{-\varkappa_k|w_\ast|} (M_k^{\dagger}+|\sigma_k|)/|k|.
\end{split}
\end{equation}\end{proposition}

\begin{proof}
The proof is similar to the proof of propositions \ref{GVSM2} and \ref{GVS17} using the limiting absorption principle, see \eqref{LAP11.5'} for $w\ge-15$ and \eqref{LAP26.5'} for $w\leq-15$, together with the commutator argument, and we will be somewhat brief. Assume that $k\ge2$ and $M_k^{\dagger}+|\sigma_k|=1$, without loss of generality.  

We consider the case $w_\ast\ge-15$ first. Denote
$$M:=\sup_{j\in\Z}\,e^{\varkappa_k|j-w_\ast|}\Big\|\big|A_k\big[Q_k^{w_\ast}(v,\cdot)\big](w)\big|+|k|^{-1}\big|\partial_vA_k\big[Q_k^{w_\ast}(v,\cdot)\big](w)\big|\Big\|_{L^2(v\in[j,j+2], w\in\R)}.$$
We need to show that $M\lesssim1/|k|$. From \eqref{GVSL2}, it follows that $Q_k^{w_\ast}$ satisfies for $v,w\in\R$ and $0<\epsilon\ll e^{-2|w_\ast|}$,
\begin{equation}\label{GVSL8}
\begin{split}
&Q_k^{w_\ast}(v,w)+\frac{1}{2|k|}\int_\R e^{-|k||v-\rho|} \frac{e^{2\rho+2w}D(\rho+w)}{B(\rho+w)-B(w)+i\epsilon}Q_k^{w_\ast}(\rho,w)\,d\rho\\
&=\digamma_k^\ast(w)\frac{e^{-|k||v|}}{2|k|}+J_{1\epsilon}(v,w)+J_{2\epsilon}(v,w),
\end{split}
\end{equation}
where we have denoted
\begin{equation}\label{GVSL9}
\begin{split}
&J_{1,\epsilon}(v,w):=\frac{1}{2|k|}\int_\R e^{-|k||v-\rho|}\big(1-\Psi(\rho)\big)\Psi^\ast(w-w_\ast)\Big[\frac{e^{2\rho+2w}D(\rho+w)Q_k^{w_\ast}(\rho,w)}{B(\rho+w)-B(w)+i\epsilon}\\
&\hspace{3in}-{\rm P.V.} \frac{e^{2\rho+2w}D(\rho+w)Q_k^{w_\ast}(\rho,w)}{B(\rho+w)-B(w)}\Big]\,d\rho,\\
&J_{2\epsilon}(v,w):=\frac{1}{2|k|}\int_\R e^{-|k||v-\rho|}\Psi(\rho)\Psi^\ast(w-w_\ast)\Big[\frac{e^{2\rho+2w}D(\rho+w)Q_k^{w_\ast}(\rho,w)}{B(\rho+w)-B(w)+i\epsilon}\\
&\hspace{3in}-{\rm P.V.} \frac{e^{2\rho}D(\rho)Q_k^{w_\ast}(\rho,w)}{B(\rho+w)-B(w)}\Big]\,d\rho.
\end{split}
\end{equation}
In the above we suppressed the dependence of $J_{1\epsilon}, J_{2\epsilon}$ on $k, w_\ast$ for the simplicity of notations. It follows from the bounds \eqref{BSD6} that for $w\in\R$,
\begin{equation}\label{GVSL10}
\Big\|\sup_{j\in\Z}\Big\|e^{\varkappa_k|v|}\Big[\big|J_{1\epsilon}(v,w)\big|+|k|^{-1}\big|\partial_vJ_{1\epsilon}(v,w)\big|\Big]\Big\|_{L^2(v\in[j,j+2])}\Big\|_{L^2(w\in\R)}\lesssim_{k,w_\ast}\epsilon,
\end{equation}
where we do not have to be precise on the implied constants which might depend on $k,w_\ast$ since in the end $\epsilon\to0+$.
For $w_\ast\in\R$, we notice that the bounds \eqref{GVSM3} imply that
\begin{equation}\label{GVSL7.1}
\Big\|\Psi^\ast(v)\Big[\big|A_k\big[Q_k^{w_\ast}(v,\cdot)\big](w)\big|+|k|^{-1}\big|\partial_vA_k\big[Q_k^{w_\ast}(v,\cdot)\big](w)\big|\Big]\Big\|_{L^2(v\in\R,w\in\R)}\lesssim e^{-\varkappa_k|w_\ast|}/|k|.
\end{equation}
Using lemma \ref{gbo3}, \eqref{GVS13.1} and \eqref{GVSL7.1}, we have
\begin{equation}\label{GVSL11}
\begin{split}
&\Big\|\sup_{j\in\Z}\Big\|e^{|k||j|}\Psi^\ast(v-j)\Big[\big|J_2(v,w)\big|+|k|^{-1}\big|\partial_vJ_2(v,w)\big|\Big]\Big\|_{L^2(v\in[j,j+2])}\Big\|_{L^2(w\in\R)}\\
&\lesssim|k|^{-1}\Big\|\int_{\R^3}\langle k,\eta\rangle^{-1}e^{-\langle\eta+\alpha\rangle^{3/4}-\delta_0\langle\beta\rangle^{1/2}}\Big|\int_\R e^{i(\alpha+\gamma)\rho}Q_k^{w_\ast}(\rho,w)\,d\rho\Big| d\alpha d\beta d\gamma \Big\|_{L^2(\eta\in\R, w\in\R)}\\
&\lesssim e^{-|\varkappa_k||w_\ast|}/|k|.
\end{split}
\end{equation}
By the limiting absorption principle for each $w\in\R$ with $|w-w_\ast|\leq5$, see \eqref{LAP11.5'} (with a shift in $v\to v+w$), and square integrating in $w$, we conclude from \eqref{GVSL8} and the bounds \eqref{GVSL10}-\eqref{GVSL11},
\begin{equation}\label{GVSL12}
\Big\|\sup_{j\in\Z}\Big\|e^{\varkappa_k|v|}\Big[\big|Q_k^{w_\ast}(v,w)\big|+|k|^{-1}\big|\partial_vQ_k^{w_\ast}(v,w)\big|\Big]\Big\|_{L^2(v\in[j,j+2])}\Big\|_{L^2(w\in\R)}\lesssim C_{k,w_\ast}\epsilon+ |k|^{-1}e^{-\varkappa_k|w_\ast|}.
\end{equation}
Sending $\epsilon\to0+$, we have,
\begin{equation}\label{GVSL12.1}
\Big\|\sup_{j\in\Z}\Big\|e^{\varkappa_k|v|}\Big[\big|Q_k^{w_\ast}(v,w)\big|+|k|^{-1}\big|\partial_vQ_k^{w_\ast}(v,w)\big|\Big]\Big\|_{L^2(v\in[j,j+2])}\Big\|_{L^2(w\in\R)}\lesssim  |k|^{-1}e^{-\varkappa_k|w_\ast|}.
\end{equation}
To prove the higher regularity estimates, we again use the commutator argument, apply the Fourier multiplier $A_k$ to equation \eqref{GVSL8}, multiply with $\Psi^\ast(w-w_\ast)$, and obtain that 
\begin{equation}\label{GVSL12.2}
\begin{split}
&\Psi^\ast(w-w_\ast)A_k\big[Q_k^{w_\ast}(v,\cdot)\big](w)+\int_\R \frac{e^{-|k||v-\rho|}}{2|k|} \frac{e^{2\rho+2w}D(\rho+w)\Psi^\ast(w-w_\ast)}{B(\rho+w)-B(w)+i\epsilon}A_k\big[Q_k^{w_\ast}(\rho,\cdot)\big](w)\,d\rho\\
&=\Psi^\ast(w-w_\ast)\Big\{A_k\big[\digamma^\ast_k(\cdot)\big](w)\frac{e^{-|k||v|}}{2|k|}+A_k\big[J_{1\epsilon}(v,\cdot)\big](w)+A_k\big[J_{2\epsilon}(v,\cdot)\big](w)+\mathcal{C}_Q(v,w)\Big\},
\end{split}
\end{equation}
where the commutator term $\mathcal{C}_Q$ (we have again suppressed the dependence on $k, w_\ast$ for the simplicity of notations) is given by (recall the definitions \eqref{GVS6} and the bound \eqref{GVS7} for the kernel $K$)
\begin{equation}\label{GVSL13}
\begin{split}
\mathcal{C}_Q(v,w):=&\int_\R \frac{e^{-|k||v-\rho|}}{2|k|} e^{2\rho+2w_\ast}\frac{1}{B'(w_\ast)}\Big\{K(\rho,w)A_k\big[Q_k^{w_\ast}(\rho,\cdot)\big](w)-A_k\big[K(\rho,\cdot)Q_k^{w_\ast}(\rho,\cdot)\big](w)\Big\}\,d\rho.
\end{split}
\end{equation}
Straightforward calculations using \eqref{GVSM3} show that 
\begin{equation}\label{GVSL16}
\begin{split}
&\sup_{j\in\Z}\Big\|e^{\varkappa_k|v|}\Big\{\big|A_k\big[\digamma^\ast_k(\cdot)\big](w)\big|+|k|^{-1}\big|\partial_vA_k\big[\digamma^\ast_k(\cdot)\big](w)\big|\Big\}\Big\|_{L^2(v\in[j,j+2]),w\in\R}\lesssim e^{-\varkappa_k|w_\ast|}/|k|,\\
&\sup_{j\in\Z}\Big\|e^{\varkappa_k|v|}\Big[\big|A_k\big[J_1(v,\cdot)\big](w)\big|+|k|^{-1}\big|\partial_vA_k\big[J_1(v,\cdot)\big](w)\big|\Big]\Big\|_{L^2(v\in[j,j+2],w\in\R)}\lesssim_{k, w_\ast} \epsilon.
\end{split}
\end{equation}
Similar calculations as in \eqref{GVS10}, see the bounds \eqref{GVS9.0} on $h_{0k}^{1\ast}$, we have
\begin{equation}\label{GVSL16'}
\begin{split}
&\Big\|\sup_{j\in\Z}\Big\|e^{\varkappa_k|v|}\Big\{\big|A_k\big[J_2(v,\cdot)\big](w)\big|+|k|^{-1}\big|\partial_vA_k\big[J_2(v,\cdot)\big](w)\big|\Big\}\Big\|_{L^2(v\in[j,j+2])}\Big\|_{L^2(w\in\R)}\lesssim e^{-\varkappa_k|w_\ast|}/|k|,
\end{split}
\end{equation}

Additionally, using the same argument as in the bounds for $\mathcal{C}_k(v,w)$, see \eqref{GVS13.6}-\eqref{GVS13.7}, we can bound for any $\gamma\in(0,1)$ and suitable $C_\gamma\in(0,\infty)$,
\begin{equation}\label{GVSL17}
\begin{split}
&\Big\|\sup_{j\in\Z}\Big\|e^{\varkappa_k|v|}\Big[\big|\mathcal{C}_Q(v,w)\big|+|k|^{-1}\big|\partial_v\mathcal{C}_Q(v,w)\big|\Big]\Big\|_{L^2(v\in[j,j+2])}\Big\|_{L^2(w\in\R)}\lesssim (C_{\gamma}e^{-\varkappa_k|w_\ast|}+\gamma M)/|k|.
\end{split}
\end{equation}
Combining the bounds \eqref{GVSL16}-\eqref{GVSL17}, applying the limiting absorption principle (see \eqref{LAP11.5'}) for each $w$ with $|w-w_\ast|\leq5$, square integrating in $w$, and comparing $\Psi^\ast(w-w_\ast)A_k\big[Q_k^{w_\ast}(v,\cdot)\big](w)$ and $A_k\big[Q_k^{w_\ast}(v,\cdot)\big](w)$, we can conclude that
\begin{equation}\label{GVSL18}
M\lesssim (C_{\gamma}e^{-|k||w_\ast|}+\gamma M)/|k|.
\end{equation}
Choosing $\gamma\in(0,1)$ sufficiently small, then the desired bounds \eqref{GVSL7} follow from \eqref{GVSL18}.

We now turn to the proof of \eqref{GVSL6} for the case $w_\ast\in(-\infty,-15]$. Denote
$$M^\ast:=\sup_{j\in\Z}\,\zeta_{\varkappa_k,w_\ast}(j+w_\ast,w_\ast)\Big\|\big|A_k\big[Q_k^{w_\ast}(v,\cdot)\big](w)\big|+|k|^{-1}\big|\partial_vA_k\big[Q_k^{w_\ast}(v,\cdot)\big](w)\big|\Big\|_{L^2(v\in[j,j+2], w\in\R)}.$$
We need to show that $M^\ast\lesssim1/|k|$. Denote for $v,w\in\R$,
\begin{equation}\label{GVSL24}
\digamma_k^\dagger(v,w):=-2\frac{e^{2w}\big(D(w)\digamma_k(w)-F_{0k}(w)\big)}{B'(w)}\mathcal{G}_k^w(v+w,w).
\end{equation}
By the assumption \eqref{MARr1}, the bounds \eqref{GVS1.2} and lemma \ref{gbo3}, we have the bounds for $w_\ast\in(-\infty,-15]$ and $\digamma_{k,w_\ast}^\dagger(v,w):=\digamma_k^\dagger(v,w)\Psi^\ast(w-w_\ast)$,
\begin{equation}\label{GVSLL1}
\Big\|e^{\delta_1\langle k,\xi\rangle^{1/2}}\widehat{\,\,\digamma_{k,w_\ast}^\dagger}(v,\xi)\Big\|_{L^2(\xi\in\R)}\lesssim \varpi_{k,w_\ast}(v+w_\ast,w_\ast) e^{-\mu_{\varkappa_k}|w_\ast|}.
\end{equation}
 It follows from equation \eqref{GVSL3} that $Q_k^{w_\ast}$ satisfies the equation
\begin{equation}\label{GVSL19}
\begin{split}
&Q_k^{w_\ast}(v,w)+\int_\R \mathcal{G}_k^w(v+w,\rho+w) \Big[\frac{e^{2\rho+2w}D(\rho+w)}{B(\rho+w)-B(w)+i\epsilon}-V_{w}(\rho+w)\Big]Q_k^{w_\ast}(\rho,w)\,d\rho\\
&=\digamma_{k,w_\ast}^{\dagger}(v,w)+J^\ast_{1\epsilon}(v,w)+J^\ast_{2\epsilon}(v,w),
\end{split}
\end{equation}
where we have denoted
\begin{equation}\label{GVSL20}
\begin{split}
&J^\ast_{1\epsilon}(v,w):=\int_\R \mathcal{G}_k^w(v+w,\rho+w)\big(1-\Psi(\rho)\big)\Psi^\ast(w-w_\ast)\Big[\frac{e^{2\rho+2w}D(\rho+w)Q_k^{w_\ast}(\rho,w)}{B(\rho+w)-B(w)+i\epsilon}\\
&\hspace{3in}-{\rm P.V.} \frac{e^{2\rho+2w}D(\rho+w)Q_k^{w_\ast}(\rho,w)}{B(\rho+w)-B(w)}\Big]\,d\rho,\\
&J^\ast_{2\epsilon}(v,w):=\int_\R \mathcal{G}_k^w(v+w,\rho+w)\Psi(\rho)\Psi^\ast(w-w_\ast)\Big[\frac{e^{2\rho+2w}D(\rho+w)Q_k^{w_\ast}(\rho,w)}{B(\rho+w)-B(w)+i\epsilon}\\
&\hspace{3in}-{\rm P.V.} \frac{e^{2\rho+2w}D(\rho+w)Q_k^{w_\ast}\rho,w)}{B(\rho+w)-B(w)}\Big]\,d\rho.
\end{split}
\end{equation}
In the above we suppressed the dependence of $J^\ast_{1\epsilon}, J^\ast_{2\epsilon}$ on $k, w_\ast$. Simple calculations (as we do not need precise dependence on $k, w_\ast$) using the bounds \eqref{BSD5} show that 
\begin{equation}\label{GVSL21}
\Big\|\sup_{j\in\Z}\Big\|\zeta_{\varkappa_k,w_\ast}(v+w_\ast,w_\ast)\Big[\big|J^\ast_{1\epsilon}(v,w)\big|+|k|^{-1}\big|\partial_vJ^\ast_{1\epsilon}(v,w)\big|\Big]\Big\|_{L^2(v\in[j,j+2])}\Big\|_{L^2(w\in\R)}\lesssim_{k,w_\ast}\epsilon.
\end{equation}
Calculations similar to \eqref{GVS21.2}-\eqref{GVS21.3}, using \eqref{ENHQ1}, show that
\begin{equation}\label{GVSL22}
\Big\|\sup_{j\in\Z}\Big\|\zeta_{\varkappa_k,w_\ast}(v+w_\ast,w_\ast)\Big[\big|J^\ast_{2\epsilon}(v,w)\big|+|k|^{-1}\big|\partial_vJ^\ast_{2\epsilon}(v,w)\big|\Big]\Big\|_{L^2(v\in[j,j+2])}\Big\|_{L^2(w\in\R)}\lesssim \frac{e^{-\mu_{\varkappa_k}|w_\ast|}}{|k|}.
\end{equation}
By the limiting absorption principle for $w\in\R$ with $|w-w_\ast|\leq5$, see \eqref{LAP26.5'} (with a shift in $v\to v+w_\ast$), and square integrating in $w$, we conclude from \eqref{GVSL8} and the bounds \eqref{GVSL10}-\eqref{GVSL11} that
\begin{equation}\label{GVSL23}
\Big\|\sup_{j\in\Z}\Big\|\zeta_{\varkappa_k,w_\ast}(v+w_\ast,w_\ast)\Big[\big|Q_k^{w_\ast}(v,w)\big|+\frac{1}{|k|}\big|\partial_vQ_k^{w_\ast}(v,w)\big|\Big]\Big\|_{L^2(v\in[j,j+2])}\Big\|_{L^2(w\in\R)}\lesssim C_{w_\ast}\epsilon+ \frac{e^{-\mu_{\varkappa_k}|w_\ast|}}{|k|}.
\end{equation}
Sending $\epsilon\to0+$, we obtain that
\begin{equation}\label{GVSL23.1}
\Big\|\sup_{j\in\Z}\Big\|\zeta_{\varkappa_k,w_\ast}(v+w_\ast,w_\ast)\Big[\big|Q_k^{w_\ast}(v,w)\big|+|k|^{-1}\big|\partial_vQ_k^{w_\ast}(v,w)\big|\Big]\Big\|_{L^2(v\in[j,j+2])}\Big\|_{L^2(w\in\R)}\lesssim e^{-\mu_{\varkappa_k}|w_\ast|}/|k|.
\end{equation}
To prove the higher regularity estimates, we again use the commutator argument, apply the Fourier multiplier $A_k$ to equation \eqref{GVSL8} and multiply with $\Psi^\ast(w-w_\ast)$, and obtain that 
\begin{equation}\label{GVSL25}
\begin{split}
&\Psi^\ast(w-w_\ast)A_k\big[Q_k^{w_\ast}(v,\cdot)\big](w)+\int_\R \mathcal{F}_k(v,\rho,w)\frac{e^{2\rho+2w}D(\rho+w)\Psi^\ast(w-w_\ast)}{B(\rho+w)-B(w)+i\epsilon}A_k\big[Q_k^{w_\ast}(\rho,\cdot)\big](w)\,d\rho\\
&=\Psi^\ast(w-w_\ast)\Big\{A_k\big[\digamma^\dagger_k(v,\cdot)\big](w)+A_k\big[J^\ast_1(v,\cdot)\big](w)+A_k\big[J^\ast_2(v,\cdot)\big](w)+\mathcal{C}^\ast_Q(v,w)\Big\},
\end{split}
\end{equation}
where the commutator term $\mathcal{C}^\ast_Q$ (we have again suppressed the dependence on $k, w_\ast$ for the simplicity of notations) is given by
\begin{equation}\label{GVSL26}
\begin{split}
\mathcal{C}^\ast_Q(v,w):=&\int_\R \mathcal{F}_k^{w_\ast}(v,\rho,w)L(\rho,w)A_k\big[Q_k^{w_\ast}(\rho,\cdot)\big](w)\,d\rho-A_k\Big\{\int_\R \mathcal{F}_k^{w_\ast}(v,\rho,\cdot)L(\rho,\cdot)Q_k^{w_\ast}(\rho,\cdot)\,d\rho\Big\}.
\end{split}
\end{equation}
In the above we used the definitions \eqref{ENHQ2} for $\mathcal{F}_k$, \eqref{GVS19.1} for $\mathcal{F}_k^{w_\ast}$, and recall \eqref{GVS19.2}-\eqref{GVS19.3} for the definitions and bounds of $L$ and $L^\ast$.
Simple calculations show that 
\begin{equation}\label{GVSL29}
\begin{split}
&\Big\|\sup_{j\in\Z}\,\zeta_{\varkappa_k,w_\ast}(j+w_\ast,w_\ast)\Big\|\big|A_k\big[\digamma^\dagger_k(v,\cdot)\big](w)\big|+|k|^{-1}\big|\partial_vA_k\big[\digamma^\dagger_k(v,\cdot)\big](w)\big|\Big\|_{L^2(v\in[j,j+2])}\Big\|_{L^2(w\in\R)}\\
&\lesssim e^{-\mu_{\varkappa_k}|w_\ast|}/|k|,\\
&\Big\|\sup_{j\in\Z}\,\zeta_{\varkappa_k,w_\ast}(j+w_\ast,w_\ast)\Big\|\big|A_k\big[J^\ast_1(v,\cdot)\big](w)\big|+|k|^{-1}\big|\partial_vA_k\big[J^\ast_1(v,\cdot)\big](w)\big|\Big\|_{L^2(v\in[j,j+2])}\Big\|_{L^2(w\in\R)}\\
&\lesssim_{k,w_\ast} \epsilon.
\end{split}
\end{equation}
Calculations similar to \eqref{GVS21.2} show that
\begin{equation}
\begin{split}
&\Big\|\sup_{j\in\Z}\,\zeta_{\varkappa_k,w_\ast}(j+w_\ast,w_\ast)\Big\|\big|A_k\big[J^\ast_2(v,\cdot)\big](w)\big|+|k|^{-1}\big|\partial_vA_k\big[J^\ast_2(v,\cdot)\big](w)\big|\Big\|_{L^2(v\in[j,j+2])}\Big\|_{L^2(w\in\R)}\\
&\lesssim e^{-\mu_{\varkappa_k}|w_\ast|}/|k|.
\end{split}
\end{equation}

Additionally, we can bound for any $\gamma\in(0,1)$, completely analogous to \eqref{GVS21.9}-\eqref{GVS26}, that
\begin{equation}\label{GVSL30}
\begin{split}
&\Big\|\sup_{j\in\Z}\Big\|\zeta_{\varkappa_k,w_\ast}(j+w_\ast,w_\ast)\Big[\big|\mathcal{C}_Q^{\ast}(v,w)\big|+|k|^{-1}\big|\partial_v\mathcal{C}_Q^{\ast}(v,w)\big|\Big]\Big\|_{L^2(v\in[j,j+2])}\Big\|_{L^2(w\in\R)}\\
&\lesssim (C_{\gamma}e^{-\mu_{\varkappa_k}|w_\ast|}+\gamma M^\ast)/|k|.
\end{split}
\end{equation}
Combining \eqref{GVSL30} with the bounds \eqref{GVSL16}-\eqref{GVSL17}, applying the limiting absorption principle, and comparing $\Psi^\ast(w-w_\ast)A_k\big[Q_k^{w_\ast}(v,\cdot)\big](w)$ and $A_k\big[Q_k^{w_\ast}(v,\cdot)\big](w)$, we can conclude that
\begin{equation}\label{GVSL31}
M^\ast\lesssim (C_{\gamma}e^{-\mu_{\varkappa_k}|w_\ast|}+\gamma M^\ast)/|k|.
\end{equation}
Choosing $\gamma\in(0,1)$ sufficiently small, then the desired bounds \eqref{GVSL6} follow from \eqref{GVSL31}.

\end{proof}

\section{Proof of main theorems}\label{sec:mainth}
In this section we give the proof of our main theorems, beginning with the proof of Theorem \ref{MTH2}. 

\subsection{Proof of Theorem \ref{MTH2}} We can normalize and assume that $M_k^{\dagger}+|\sigma_k|=1$. In view of the equation \eqref{GVSL1}, the bounds \eqref{GVS1.2}, and the bounds \eqref{GVSL6}-\eqref{GVSL7}, it remains to prove \eqref{MTH2.3} and \eqref{MTH2.9}. it follows from \eqref{GVSL1} that for $v,w\in\R$ with $|v|\ge\frac{1}{2}$, $\Theta_k(v,w)$ satisfies 
\begin{equation}\label{PMT1}
(k^2-\partial_v^2)\Theta_k(v,w)+\frac{e^{2v+2w}D(v+w)}{B(v+w)-B(w)}\Theta_k(v,w)=0.
\end{equation}
We note that for $|v|\ge1/2$ the coefficient $\frac{e^{2v+2w}D(v+w)}{B(v+w)-B(w)}$ is Gevrey-2 regular in both $v$ and $w$. The desired bounds \eqref{MTH2.3} then follow from the bounds \eqref{GVSL6}-\eqref{GVSL7}, and lemma \ref{EllipReg} applied to \eqref{PMT1} for each $w\in\R$ with $|w-w_\ast|\leq5$ and then square integrated in $w$.

We now turn to the proof of \eqref{MTH2.8}-\eqref{MTH2.9}. Choose a smooth function $\Psi_9\in C^{\infty}_0(-10,10)$ with $\Psi_9\equiv1$ on $[-9,9]$ and $\sup_{\xi\in\R}e^{\langle\xi\rangle^{5/6}}\big|\widehat{\Psi_9}(\xi)\big|\lesssim1$. Fix a small parameter $\sigma\in(0,1)$ to be determined below. 
For notational conveniences, we assume that $\zeta(v)$ is extended to be defined on $\R$, satisfying $|\zeta'(v)|\approx1$ on $\R$ and $\zeta'\in \widetilde{G}_{M'}^{1/2}(\R)$ with a suitable $M'\in(0,\infty)$ depending on $M$. 
Define for $v, \rho, w\in\R$, 
\begin{equation}\label{PMT8}
P_k(v,\rho,w):=\Theta_k(\zeta(v+\rho)-\zeta(\rho),w),\,\, P^\ast_k(v,\rho,w):=P(v,\rho,w)\Psi^\ast(w-w_\ast)\Psi^\ast(v/\sigma)\Psi^\ast(\rho),
\end{equation}
then $P_k^\ast$ satisfies for $v, \rho, w\in\R$,
\begin{equation}\label{PMT9}
\begin{split}
&\Big[k^2-(\zeta'(v+\rho))^{-2}\partial_v^2+\frac{\zeta''(v+\rho)}{(\zeta'(v+\rho))^3}\partial_v\Big]P_k^\ast(v,\rho,w)+{\rm P.V.}\, \frac{a^\ast(v,\rho,w)}{v}P_k^\ast(v,\rho,w)=R_\sigma(v,\rho,w),
\end{split}
\end{equation} 
where in the above 
\begin{equation}\label{PMT10}
\begin{split}
&a^\ast(v,\rho,w):=v\frac{e^{2\zeta(v+\rho)-2\zeta(\rho)+2w}D(\zeta(v+\rho)-\zeta(\rho)+w)}{B\big(\zeta(v+\rho)-\zeta(\rho)+w\big)-B(w)}\Psi_9(v)\Psi_9(\rho)\Psi_9(w-w_\ast),\\
&R_\sigma(v,\rho,w):=(1/\sigma)\frac{\zeta''(v+\rho)}{(\zeta'(v+\rho))^3}\partial_v\Psi^\ast(v/\sigma)P_k(v,\rho,w)\Psi^\ast(\rho)\Psi^\ast(w-w_\ast)\\
&-(\zeta'(v+\rho))^{-2}\Big[(2/\sigma)\partial_vP(v,\rho,w)\partial_v\Psi^\ast(v/\sigma)+(1/\sigma^2)P(v,\rho,w)\partial_v^2\Psi^\ast(v/\sigma)\Big]\Psi^\ast(\rho)\Psi^\ast(w-w_\ast)\\
&-\frac{2\pi e^{2w}(D(w)\digamma_k(w)-F_{0k}(w))}{B'(w)\zeta'(\rho)}\Psi^\ast(\rho)\Psi^\ast(w-w_\ast)\delta(v).
\end{split}
\end{equation}
The main point of the equation \eqref{PMT9} is that $R_\sigma(v,\rho,w)$ is Gevrey regular in both $\rho$ and $w$, and the coefficient $a^\ast(v,\rho,w)$ is Gevrey-2 smooth in both $v, \rho$ and $w$ for $|v|\leq 20, |\rho|\leq 20$ and $|w-w_\ast|\leq 20$. Indeed, we have for some $\delta_0'\in(0,1)$ depending on $M$ and $\delta_0$,
\begin{equation}\label{PMT10.1}
\big\|e^{\delta_0'\langle\alpha,\beta,\gamma\rangle^{1/2}}\widehat{\,a^\ast}(\alpha,\beta,\gamma)\big\|_{L^2(\R^3)}\lesssim1.
\end{equation}
To see \eqref{PMT10.1}, we can rewrite 
\begin{equation}\label{PMT11}
\begin{split}
a^\ast(v,\rho,w):=&e^{2\zeta(v+\rho)-2\zeta(\rho)+2w}D(\zeta(v+\rho)-\zeta(\rho)+w)\frac{v}{\zeta(v+\rho)-\zeta(\rho)}\\
&\times\frac{\zeta(v+\rho)-\zeta(\rho)}{B\big(\zeta(v+\rho)-\zeta(\rho)+w\big)-B(w)}\Psi_9(\rho)\Psi_9(w-w_\ast)\Psi_9(v),
\end{split}
\end{equation}
and the desired bound \eqref{PMT10.1} follows from the simple observation that for $v,\rho,\alpha,w\in\R$,
\begin{equation}\label{PMT12}
\zeta(v+\rho)-\zeta(\rho)=v\int_0^1\zeta'(\rho+sv)\,ds,\quad{\rm and}\quad B(\alpha+w)-B(w):=\alpha\int_0^1B'(w+s\alpha)\,ds.
\end{equation}
Moreover, since $\partial_v\Psi^\ast(v/\sigma)$ and $\partial_v^2\Psi^\ast(v/\sigma)$ are supported away from $v=0$, the right hand side $R$ of \eqref{PMT9} is Gevrey-2 regular in $v, \rho$ and $w$ for $|v|\leq 20, |\rho|\leq 20$ and $|w-w_\ast|\leq 20$. 

Now to prove the desired Gevrey smoothness in $\rho$ for $P^\ast$, we use the same commutator argument as before. It seems at first glance that we are in the same situation as in propositions \ref{GVSM2} and \ref{GVS17}, and would need to apply the limiting absorption principle, see Lemmas \ref{LAP10} and \ref{LAP25}, which involves weights and is quite complicated. However, in our case here, there is a key difference, namely we already know that $P$ is Gevrey smooth whenever $v$ stays away from the origin, and that allows us to localize $v$ to a small region near $v=0$ and the extra small parameter $\sigma$ ensures coercive bounds from \eqref{PMT9}, without the use of the limiting absorption principle.

By the bounds \eqref{MTH2.5}, we have
\begin{equation}\label{PMT12.01}
\big\|(|k|+|\alpha|)e^{\delta_1\langle k,\gamma\rangle^{1/2}}\widehat{\,\,P_k^\ast}(\alpha,\beta,\gamma)\big\|_{L^2(\alpha,\beta,\gamma\in\R)}\lesssim_\sigma \varpi_{\varkappa_k,w_\ast}(w_\ast,0),
\end{equation}
which we will use to control the low frequency part (for the variable $\rho$) of the desired bound \eqref{MTH2.9}.

By the bounds \eqref{PMT12.01}, equation \eqref{PMT1} and lemma \ref{EllipReg}, we can choose $\delta_1'\in(0,1)$ sufficiently small, depending on $\delta_1$ and $M$, such that 
\begin{equation}\label{PMT12.1}
\sup_{\alpha\in\R}\big\|e^{2\delta_1'\langle k,\beta\rangle^{1/2}+2\delta_1'\langle k,\gamma\rangle^{1/2}}\widehat{\,R_\sigma}(\alpha,\beta,\gamma)\big\|_{L^2(\beta,\gamma\in\R)}\lesssim_{\sigma}(M_k^{\dagger}+|\sigma_k|)\varpi_{\varkappa_k,w_\ast}(w_\ast,0).
\end{equation}
Define the Fourier multiplier $A_k^\ast$ such that for any $h\in L^2(\R)$
\begin{equation}\label{PMT12.3}
\widehat{A_k^\ast h}(\beta):=e^{\delta_1'\langle k,\beta\rangle^{1/2}}\widehat{\,h\,}(\beta), \qquad{\rm for}\,\,\beta\in\R.
\end{equation}
Applying the Fourier multiplier $A_k^\ast$ (in the variables $\rho$) to equation \eqref{PMT9}, we obtain that for $v,\rho,w\in\R$,
\begin{equation}\label{PMT12.4}
\begin{split}
&\Big[k^2-(\zeta'(v+\rho))^{-2}\partial_v^2+\frac{\zeta''(v+\rho)}{(\zeta'(v+\rho))^3}\partial_v\Big]A_k^\ast\big[P_k^\ast(v,\cdot,w)\big](\rho)+{\rm P.V.}\, \frac{a^\ast(v,\rho,w)}{v}A_k^\ast\big[P_k^\ast(v,\cdot,w)\big](\rho)\\
&=A_k^\ast\big[R_\sigma(v,\cdot,w)\big](\rho)+C^{11}_\sigma(v,\rho,w)+C^{12}_\sigma(v,\rho,w)+C^2_\sigma(v,\rho,w),
\end{split}
\end{equation}
where we have defined
\begin{equation}\label{PMT12.5}
\begin{split}
&C^{11}_\sigma(v,\rho,w):=\partial_v H(v,\rho,w),\\
&H_\sigma(v,\rho,w):=\Big[-(\zeta'(v+\rho))^{-2}\partial_v\Big]A_k^\ast\big[P_k^\ast(v,\cdot,w)\big](\rho)-A_k^\ast\Big[-(\zeta'(v+\cdot))^{-2}\partial_vP_k^\ast(v,\cdot,w)\Big](\rho),\\
&C^{12}_\sigma(v,\rho,w):=\Big[\partial_v(\zeta'(v+\rho))^{-2}\partial_v+\frac{\zeta''(v+\rho)}{(\zeta'(v+\rho))^3}\partial_v\Big]A_k^\ast\big[P_k^\ast(v,\cdot,w)\big](\rho)\\
&\hspace{1in}-A_k^\ast\Big[\Big(\big(\partial_v(\zeta'(v+\cdot))^{-2}\partial_v\big)+\frac{\zeta''(v+\cdot)}{(\zeta'(v+\cdot))^3}\partial_v\Big)P_k^\ast(v,\cdot,w)\Big](\rho),\\
&C^2_\sigma(v,\rho,w):={\rm P.V.}\, \frac{a^\ast(v,\rho,w)}{v}A^\ast_k\Big[P_k^\ast(v,\cdot,w)\Big](\rho)-A_k^\ast\Big[{\rm P.V.}\, \frac{a^\ast(v,\cdot,w)}{v}P_k^\ast(v,\cdot,w)\Big](\rho).
\end{split}
\end{equation}
Simple computation using the regularity of $\zeta$ shows 
\begin{equation}\label{PMT12.51}
\begin{split}
&\big\|\Psi_9(\rho)C^{12}_\sigma(v,\rho,w)\big\|_{L^2(v,\rho,w\in\R)}+\big\|\Psi_9(\rho)H_\sigma(v,\rho,w)\big\|_{L^2(v,\rho,w\in\R)}\\
&\lesssim \big\|\langle k,\partial_\rho\rangle^{-1/2}\partial_v A_k^\ast\big[P_k^\ast(v,\cdot,w)\big](\rho)\big\|_{L^2(v,\rho,w\in\R)}.
\end{split}
\end{equation}

To treat the term $C^{11}_\sigma$ and $C^2_\sigma$ below, we use the following lemma.
\begin{lemma}\label{PMT12.6}
Assume that $\sigma\in(0,1)$ and $I:=(-20\sigma, -20\sigma)$. Then for any $h_1, h_2\in H^1_0(I)$ we have the bounds
\begin{equation}\label{PMT12.7}
\Big|{\rm P.V.}\,\int_I \frac{h_1(v)}{v}h_2(v)\,dv\Big|\lesssim \sigma^{2/3}\big\|h_1\big\|_{H^1(I)}\big\|h_2\|_{H^1(I)}.
\end{equation}
\end{lemma}

\begin{proof}[Proof of lemma \ref{PMT12.6}]
By Sobolev inequality and interpolation inequality, we have
\begin{equation}\label{PMT12.71}
\big\|h\big\|_{L^2(I)}\lesssim \sigma\|\partial_vh\|_{L^2(I)},\qquad\big\|h\big\|_{H^{2/3}(\R)}\lesssim \sigma^{1/3}\|\partial_vh\|_{L^2(I)},
\end{equation}
for any $h\in H^1_0(I)$. Therefore by Parsevel's identity we obtain
\begin{equation}\label{PMT12.8}
\begin{split}
&\Big|{\rm P.V.}\,\int_I \frac{h_1(v)}{v}h_2(v)\,dv\Big|\lesssim \Big|\int_{\R^2} \big|\widehat{\,h_1}(\xi-\alpha)\big|\big|\widehat{\,h_2}(\xi)\big|\,d\alpha d\xi\Big|\\
&\lesssim \big\|h_1\big\|_{H^{2/3}(\R)}\big\|h_2\big\|_{H^{2/3}(\R)}\lesssim \sigma^{2/3}\big\|h_1\big\|_{H^1(I)}\big\|h_2\|_{H^1(I)}.
\end{split}
\end{equation}
The lemma is then proved. 
\end{proof}

Denote for the simplicity of notations
\begin{equation}\label{PMT12.9}
M:=|k|\big\|A^\ast_k\big[P_k^\ast(v,\cdot,w)\big](\rho)\big\|_{L^2(v,\rho,w\in\R^3)}+\big\|\partial_vA^\ast_k\big[P_k^\ast(v,\cdot,w)\big](\rho)\big\|_{L^2(v,\rho,w\in\R^3)}.
\end{equation}
Now multiplying $(\Psi_9(\rho))^2A^\ast_k\big[P_k^\ast(v,\cdot,w)\big](\rho)$ to \eqref{PMT12.4}, integrating in $v, \rho, w\in\R$, using integration by parts, the bounds \eqref{PMT12.1}, \eqref{PMT12.51} and lemma \ref{PMT12.6}, and noting that the support of $A^\ast_k\big[P_k^\ast(v,\cdot,w)\big](\rho)$ is contained in $[-5\sigma,5\sigma]\times \R\times [-5,5]$, we obtain for some $C_\sigma\in(0,\infty)$ that 
\begin{equation}\label{PMT12.91}
\begin{split}
&\int_{\R^3} \Psi_9(\rho)\big)^2\Big[k^2\big|A_k^\ast\big[P^\ast(v,\cdot,w)\big](\rho)\big|^2+\big|\partial_vA_k^\ast\big[P^\ast(v,\cdot,w)\big](\rho)\big|^2\Big]\,dv d\rho dw\\
&\lesssim \int_{\R^3}\sigma^{3/2}\big( \Psi_9(\rho)\big)^2\big|\partial_vA_k^\ast\big[P^\ast(v,\cdot,w)\big](\rho)\big|^2+\big|\langle k,\partial_\rho\rangle^{-1}\partial_vA_k^\ast\big[P^\ast(v,\cdot,w)\big](\rho)\big|^2\,dv d\rho dw\\
&\qquad+C_\sigma \big(\varpi_{\varkappa_k,w_\ast}(w_\ast,0)\big)^2.
\end{split}
\end{equation}
Noting that $A_k^\ast\big[P^\ast(v,\cdot,w)\big](\rho)=A_k^\ast\big[\Psi_9(\cdot)P^\ast(v,\cdot,w)\big](\rho)$, a simple commutator argument shows that for $m\in\{0,1\}$,
\begin{equation}\label{PMT12.92}
\begin{split}
&\big\|\Psi_9(\rho)\partial^m_vA_k^\ast\big[P^\ast(v,\cdot,w)\big](\rho)-\partial^m_vA_k^\ast\big[P^\ast(v,\cdot,w)\big](\rho)\big\|_{L^2(v,\rho,w\in\R)}\\&\lesssim\big\|\langle k,\partial_\rho\rangle^{-1/2}\partial^m_v A_k^\ast\big[P_k^\ast(v,\cdot,w)\big](\rho)\big\|_{L^2(v,\rho,w\in\R)}. 
\end{split}
\end{equation}

Combining \eqref{PMT12.91} and \eqref{PMT12.92}, in view of \eqref{PMT12.01}, we obtain that for any $\gamma\in(0,1)$ and suitable $C_{\gamma,\sigma}\in(0,\infty)$,
\begin{equation}\label{PMT13}
M\lesssim \sigma^{1/3} M+\gamma M+ \big[e^{-\varkappa_k|w_\ast|}{\bf 1}_{w_\ast>0}+e^{-\mu_{\varkappa_k}|w_\ast|}{\bf 1}_{w_\ast\leq0}\big]C_{\gamma,\sigma} .
\end{equation}
Choosing $\sigma, \gamma\in(0,1)$ sufficiently small, the desired bounds \eqref{MTH2.9} then follow from \eqref{PMT13} and \eqref{PMT12.01}. The proof of theorem \ref{MTH2} is now complete.

\subsection{Proof of Theorem \ref{MTH2'}}
We can now give the proof of Theorem \ref{MTH2'} using Theorem \ref{MTH2}.
Decompose $f_k(t,v)$ as in \eqref{MTH23.1}, then the bound \eqref{MTH23.3} on $F_{k,v_\ast}^2(t,v)$ follows directly from the equation \eqref{MTH2.6} and the bounds \eqref{MTH2.3} by simple integration by parts in $w$. To prove the bound \eqref{MTH23.3} on $F_{k,v_\ast}^1(t,v)$, we observe that for suitable constants $C_0, C_1\in\R$ and all $v\in\R$,
\begin{equation}\label{PMT14}
F_{k,v_\ast}^1(t,v)=C_0\aleph_k(t,v)-C_1\Big[D(v)\digamma_k(v)-F_{0k}(v)\Big]\Phi_k^\ast(v-v_\ast),
\end{equation}
where
\begin{equation}\label{PMT15}
\aleph_k(t,v):=\Phi_k^\ast(v-v_\ast)\int_\R e^{-ik(B(w)-B(v))t}{\rm P.V.}\frac{D(v)\Theta_k(v-w,w)}{B(v)-B(w)}\Phi^\ast(v-w)B'(w)\,dw.
\end{equation}
It suffices to analyze the smoothness of $\aleph_k$ in $v$, we make the change of variables for $|\rho-v_\ast|\leq10$,
\begin{equation}\label{PMT16}
\nu:=\frac{B(\rho)-B(v_\ast)}{B'(v_\ast)}, \qquad \rho:=\zeta(\nu)+v_\ast.
\end{equation}
We note that $\zeta$ (with suitable extensions if necessary) satisfies the assumptions in (iv) of Theorem \ref{MTH2}.
Recalling the identity \eqref{MTH2.4} and setting for $v,w\in\R$ with $|v-v_\ast|\leq 10$ and $|w-v_\ast|\leq10$, 
\begin{equation}\label{PMTM16.1}
\nu_1(v):=\frac{B(v)-B(v_\ast)}{B'(v_\ast)},\qquad \nu_2(w):=\frac{B(w)-B(v_\ast)}{B'(v_\ast)},
\end{equation}
we note that $\nu_1(v)$ is Gevrey-2 regular in $v$. We can write 
\begin{equation}\label{PMTM16.2}
\begin{split}
&\aleph_k(t,v)=\Phi^\ast(v-v_\ast)D(v)\int_\R e^{-ik(B(w)-B(v))t}\,{\rm P.V.}\frac{\Theta_k(v-w,w)}{B(v)-B(w)}\Phi^\ast(v-w)B'(w)\,dw\\
&=\Phi^\ast(v-v_\ast)D(v)\int_\R e^{-ik(\nu_2-\nu_1)B'(v_\ast)t}\,{\rm P.V.}\frac{\Theta_k\big(\zeta(\nu_1)-\zeta(\nu_2),\zeta(\nu_2)+v_\ast\big)}{\nu_1-\nu_2}\Phi^\ast\big(\zeta(\nu_1)-\zeta(\nu_2)\big)\,d\nu_2\\
&=\Phi^\ast(v-v_\ast)D(v)\int_\R e^{-ik\rho B'(v_\ast)t}\,{\rm P.V.}\frac{\Theta_k\big(\zeta(\nu_1)-\zeta(\nu_1+\rho),\zeta(\nu_1+\rho)+v_\ast\big)}{-\rho}\\
&\hspace{2.4in}\times\Phi^\ast\big(\zeta(\nu_1)-\zeta(\nu_1+\rho)\big)\,d\rho.
\end{split}
\end{equation}
Set $P(\rho,v):=\Theta_k\big(\zeta(\nu_1)-\zeta(\nu_1+\rho),\zeta(\nu_1+\rho)+v_\ast\big)\Phi^\ast\big(\zeta(\nu_1)-\zeta(\nu_1+\rho)\big)\Phi^{\ast\ast}(v-v_\ast)$. By \eqref{MTH2.8}-\eqref{MTH2.9}, $P(\rho,v)$ is Gevrey-2 regular in $v$ (recall the relation \eqref{PMTM16.1}) and satisfies for sufficiently small $\delta_1''\in(0,1)$ depending on $\delta_1$,
\begin{equation}\label{PMTM17}
\big\|(|k|+|\alpha|)e^{2\delta_1''\langle k,\beta\rangle^{1/2}}\widehat{\,\,P\,\,}(\alpha,\beta)\big\|_{L^2(\alpha,\beta\in\R)}\lesssim (M_k^{\dagger}+|\sigma_k|)\varpi_{\varkappa_k,v_\ast}(v_\ast).
\end{equation}
Therefore, we obtain from \eqref{PMTM16.2}-\eqref{PMTM17} that
\begin{equation}\label{PMTM18}
\begin{split}
&\big\|e^{\delta_1''\langle k,\xi\rangle^{1/2}}\widehat{\,\,\aleph_k}(t,\xi)\big\|_{L^2(\R)}\lesssim \frac{1}{1+e^{8v_\ast}}\Big\| \int_{\R}e^{\delta_1''\langle k,\xi\rangle^{1/2}}\big|\widehat{\,\,P\,\,}(\alpha,\xi)\big|\,d\alpha\Big\|_{L^2(\R)}\lesssim\frac{M_k^{\dagger}+|\sigma_k|}{1+e^{8v_\ast}}\varpi_{\varkappa_k,v_\ast}(v_\ast).
\end{split}
\end{equation}
The desired conclusion then follows from \eqref{PMTM18}. The theorem is now proved.

\end{document}